\documentclass[12pt]{article}
\usepackage{amsfonts,amssymb,amsmath,amsthm,dsfont,graphicx,hyperref,mathrsfs,euscript,enumitem,framed, bm}
\usepackage[title, titletoc]{appendix}
\usepackage[top=1in,bottom=1in,left=.9in,right=.9in]{geometry} %
\usepackage[onehalfspacing]{setspace}
\usepackage{ifthen,caption,subcaption,lscape}
\hypersetup{pdfborder = {0 0 0},colorlinks=true,linkcolor=blue,filecolor=blue,urlcolor=blue,citecolor=blue}
\usepackage{chngcntr} 
\usepackage{natbib}
\usepackage{longtable}
\usepackage{placeins}
\usepackage{threeparttable}
\newcommand{\blind}{1}

\newtheorem{thm}{Theorem}
\newtheorem{coro}{Corollary}
\newtheorem{lem}{Lemma}
\newtheorem{assumption}{Assumption}

\theoremstyle{definition}

\theoremstyle{definition}
\newtheorem{remark_tmp}{Remark}[section]
\newenvironment{remark}
{ \begin{remark_tmp} 	}
	{ 
		\medskip\hfill{\large$\lrcorner$}
	\end{remark_tmp} 
}

\newtheorem{exmp}{Example}

\allowdisplaybreaks[1]

\DeclareMathOperator*{\argmin}{arg\,min}

\DeclareMathOperator*{\tr}{Tr}

\DeclareMathOperator*{\diam}{diam}

\DeclareMathOperator*{\diag}{diag}

\renewcommand{\P}{\mathbb{P}}
\newcommand{\E}{\mathbb{E}}

\newcommand{\I}{\mathds{1}}

\newcommand{\subi}{\mathchoice
  {\mathcal{N}_i}{\mathcal{N}_i}
  {\scalebox{0.9}{$\scriptstyle\mathcal{N}_i$}}
  {\scalebox{0.9}{$\scriptscriptstyle\mathcal{N}_i$}}}
\newcommand{\gr}[1]{\cdot\mathcal{C}_{i#1}}

\newcommand{\bA}{\bm{A}}
\newcommand{\bH}{\bm{H}}
\newcommand{\bD}{\bm{D}}
\newcommand{\bB}{\bm{B}}
\newcommand{\bE}{\bm{E}}
\newcommand{\bF}{\bm{F}}
\newcommand{\bG}{\bm{G}}

\newcommand{\bI}{\bm{I}}

\newcommand{\bM}{\bm{M}}

\newcommand{\bP}{\bm{P}}
\newcommand{\bQ}{\bm{Q}}
\newcommand{\bR}{\bm{R}}
\newcommand{\bS}{\bm{S}}
\newcommand{\bU}{\bm{U}}
\newcommand{\bT}{\bm{T}}
\newcommand{\bV}{\bm{V}}
\newcommand{\bW}{\bm{W}}
\newcommand{\bX}{\bm{X}}
\newcommand{\bY}{\bm{Y}}

\newcommand{\bc}{\bm{c}}
\newcommand{\bg}{\bm{g}}
\newcommand{\be}{\bm{e}}

\newcommand{\bmf}{\bm{f}}
\newcommand{\bh}{\bm{h}}

\newcommand{\bq}{\bm{q}}

\newcommand{\bu}{\bm{u}}
\newcommand{\bv}{\bm{v}}
\newcommand{\bx}{\bm{x}}

\newcommand{\bw}{\bm{w}}
\newcommand{\bz}{\bm{z}}

\newcommand{\bXi}{\boldsymbol{\Xi}}
\newcommand{\balpha}{\boldsymbol{\alpha}}

\newcommand{\bvth}{\boldsymbol{\vartheta}}

\newcommand{\bSigma}{\boldsymbol{\Sigma}}
\newcommand{\bUpsilon}{\boldsymbol{\Upsilon}}

\newcommand{\blambda}{\boldsymbol{\lambda}}
\newcommand{\bLambda}{\boldsymbol{\Lambda}}
\newcommand{\bmeta}{\boldsymbol{\eta}}
\newcommand{\bvrho}{\boldsymbol{\varrho}}

\newcommand{\lvarsig}{\underline{\varsigma}}
\newcommand{\uvarsig}{\bar{\varsigma}}

\newcommand{\ttd}{\mathsf{d}}

\setenumerate[1]{label=\bf(\alph*)}

\renewcommand{\thesection}{\arabic{section}}

\renewcommand{\thefigure}{\arabic{figure}}
\renewcommand{\thetable}{\arabic{table}}

\renewcommand{\thedefinition}{\arabic{definition}}

\numberwithin{assumption}{section}
\numberwithin{equation}{section}
\numberwithin{thm}{section}
\numberwithin{lem}{section}
\numberwithin{coro}{thm}
\numberwithin{exmp}{section}

\newcommand{\forloop}[5][1]{\setcounter{#2}{#3}\ifthenelse{#4}{#5\addtocounter{#2}{#1}\forloop[#1]{#2}{\value{#2}}{#4}{#5}}{}}

\usepackage{tocloft}
\newcommand{\bDelta}{\bm{\Delta}}

\usepackage{bibunits}
\hypersetup{hypertexnames=false,
  pdftitle={Optimal Estimation of Large-Dimensional Nonlinear Factor Models},
  pdfauthor={Yingjie Feng},
  pdfsubject={Main paper and supplemental appendix}}
\date{September 30, 2026}
\makeatletter
\let\arxivmaketitle\maketitle
\let\arxivinternalmaketitle\@maketitle
\let\arxivtitle\title
\let\arxivauthor\author
\let\arxivdate\date
\let\arxivthanks\thanks
\let\arxivand\and
\newcommand{\restoretitlecommands}{%
  \let\maketitle\arxivmaketitle
  \let\@maketitle\arxivinternalmaketitle
  \let\title\arxivtitle
  \let\author\arxivauthor
  \let\date\arxivdate
  \let\thanks\arxivthanks
  \let\and\arxivand
  \date{September 30, 2026}}
\newcommand{\bibliographyprefix}[1]{%
  \def\@extra@binfo{#1}\def\@extra@b@citeb{#1}}
\makeatother

\begin{document}
\begin{bibunit}[hapalike]
\bibliographyprefix{main}
\if1\blind
{
\title{\vspace{-0.25in} \Large Optimal Estimation of Large-Dimensional Nonlinear Factor Models\thanks{We thank Sebastian Calonico, Matias Cattaneo, Richard Crump, Jianqing Fan, Max Farrell, Andreas Hagemann, Michael Jansson, Lutz Kilian, Xinwei Ma, Kenichi Nagasawa, Roc\'{\i}o Titiunik, Gonzalo Vazquez-Bare, Jingshen Wang, and participants at various seminars, workshops, and conferences for helpful comments and discussions.
Feng gratefully acknowledges financial support from
the National Natural Science Foundation of China (NSFC) through grants 72203122 and 72522001.}
	\bigskip }
\author{
	Yingjie Feng\thanks{School of Economics and Management, Tsinghua University. 
		E-Mail: \href{mailto:fengyj@sem.tsinghua.edu.cn}{fengyj@sem.tsinghua.edu.cn}.
}}
\maketitle
} \fi

\if0\blind
{
	\bigskip
	\bigskip
	\bigskip
	\begin{center}
		{\Large Optimal Estimation of Large-Dimensional Nonlinear Factor Models}
	\end{center}
	\medskip
} \fi

\vspace{.5em}

\begin{abstract}
This paper studies optimal estimation of large-dimensional nonlinear factor models. The key challenge is that the observed variables are possibly nonlinear functions of some latent variables, with the functional forms left unspecified. A local principal component analysis method combining $K$-nearest neighbors matching and principal component analysis is proposed to estimate the factor structure and recover information on latent variables and latent functions. Large-sample properties are established, including a sharp bound on the matching discrepancy of nearest neighbors, sup-norm error bounds for estimated local factors and factor loadings,
and the uniform convergence rate of the factor structure estimator. Under mild conditions our estimator of the latent factor structure can achieve the optimal rate of uniform convergence for nonparametric regression.
The method is illustrated with a Monte Carlo experiment and an empirical application studying the effect of tax cuts on economic growth.
\end{abstract}


\thispagestyle{empty}

\clearpage

\doublespacing
\setcounter{page}{1}

\pagestyle{plain}

\section{Introduction} \label{sec: introduction}

High-dimensional data have become increasingly available due to technological advances in data collection, 
which are typically characterized by a large number of cross-sectional units and a large number of features. 
Factor analysis is a useful tool for summarizing information in such big data sets and has wide applications in economics, statistics, and many other data science disciplines.  

One crucial  insight of factor analysis is that much of the data variation can be explained by the interaction between a few important,  usually \textit{unobserved} characteristics associated with two dimensions (``cross section'' and ``features''). 
For example, a classical factor model has the following representation
\[
x_{il}=\bmf_l'\balpha_i+u_{il}, \quad 1\leq i\leq n, \; 1\leq l\leq p,
\]
where $x_{il}$ is the $l$-th feature of the $i$-th cross-sectional unit; $\bmf_l\in\mathbb{R}^r$ denotes the vector of common factors that only vary across $l$; $\balpha_i\in\mathbb{R}^r$ denotes the vector of factor loadings that only vary across $i$; and $u_{il}$ is an idiosyncratic error.
Both the number of cross-sectional units $n$ and the number of features $p$ are large. 
Such specifications are common in many problems. For instance, in panel data analysis $x_{il}$ may be a variable collected over time, and $\bmf_l$ and $\balpha_i$ are time-specific and individual-specific effects respectively; in analysis of a recommender system $x_{il}$ may be a variable representing the preference of an individual $i$ for an item $l$, which is explained by  item-specific features $\bmf_l$ and individual-specific  features $\balpha_i$; and in causal inference and program evaluation $x_{il}$ could be repeated measurements of some underlying unobserved confounders $\balpha_i$, and $\bmf_l$ is a (linear) transformation specific to the $l$-th measurement.

The linear structure above, however, is usually restrictive and may be unrealistic in many problems. For example, test scores in multiple subjects or time periods are often used to measure fundamental abilities of students in empirical research \cite[e.g.,][]{Cunha-Heckman-Schennach_2010_ECMA}. It is difficult to justify a linear relationship between test scores and unobserved abilities in this context. A more appealing approach is to consider a \textit{possibly nonlinear} factor model
\citep{Yalcin-Amemiya_2001_SS}
\begin{equation}\label{eq: model}
x_{il}=\eta_l(\balpha_i)+u_{il}.
\end{equation}
The observed variables $x_{il}$ (e.g., test scores) and the unobserved variables $\balpha_i$ (e.g., latent abilities) are linked through a ``production function'' $\eta_l$, whose functional form is left \textit{unspecified}.
This setup encompasses the classical linear factor model \citep{Bai-Wang_2016_ARE} as a special case, but also allows for other possibly nonlinear relationships between the observables and the unobservables.

In this paper we analyze the nonlinear factor model \eqref{eq: model} based on  local principal subspace approximation. 
The procedure begins with $K$-nearest neighbors ($K$-NN) matching  approximation for each unit on the observed features $x_{il}$'s.  
Within each local neighborhood formed by the $K$ matches, the underlying possibly nonlinear factor structure is then approximated by a linear factor structure that can be estimated using principal component analysis (PCA). Given the locality of this procedure, we term it \textit{local PCA} in this paper.

This methodology has an intuitive geometric interpretation. The family of latent functions $(\eta_l:1\leq l\leq p)$ generates a \textit{low}-dimensional, possibly nonlinear ``surface'' embedded in a \textit{high}-dimensional space, if the number of features $p$ is large but the number of latent variables $r$ is small. Suppose that different values of latent variables $\balpha_i$ can induce non-negligible differences in  \textit{many} observed features $x_{il}$'s. In this case, the $K$ nearest neighbors of each unit as appropriately measured by the observed features should also be close in terms of the latent variables, thus forming a local neighborhood in the high-dimensional space. Furthermore, such units lie approximately on a local principal surface, up to errors governed by the number of nearest neighbors and the number of principal components extracted in each local neighborhood.
The availability of \emph{many} observables as ``measurements'' of the latent variables $\balpha_i$ is crucial for validity of this approximation, which affects the matching discrepancy of nearest neighbors and the estimation error of local principal components.

The idea of locally approximating a nonlinear latent surface embedded in a high-dimensional space has been widely used in the modern machine learning literature and is popular in applications such as face recognition, motion segmentation, and text classification. Typical examples include 
\cite{Zhang-Zha_2004_SIAM,
	Peng-Lu-Wang_2015_NN,Arias-Lerman-Zhang_2017_JMLR}, among others. This paper provides a formal theoretical foundation for such local approximation methods.

We establish the statistical properties of our proposed local PCA in large-dimensional settings and make several contributions to the literature.
First, we derive a sharp bound on the implicit discrepancy of nearest neighbors in terms of the latent variables (Theorem \ref{thm: IMD}). 
This result is established under a generic choice of the distance function, encompassing and extending previous studies of matching techniques based on specific metrics \citep[e.g.,][]{Zhang-Zha_2004_SIAM, Zhang-Levina-Zhu_2017_BIMA}.
Crucially, we show that the closeness of nearest neighbors, indirectly obtained through matching on noisy measurements, relies on two conditions: (i) the selected distance can ``denoise'' the data to reveal the latent factor structure $\eta_l(\balpha_i)$, 
and (ii) the distance of the noise-free structure $\eta_l(\balpha_i)$ needs to be informative about that of the unobserved variables $\balpha_i$. 
Therefore, our first contribution provides theoretical guidance for practitioners who have to rely on noisy measurements to match on some unobserved variables of interest. 

Second, we derive sup-norm error bounds for the estimated local factors and factor loadings, with their target quantities defined by applying PCA to the noise-free data for the nearest neighbors of each unit (Theorem \ref{thm: uniform convergence of factors}). Our results further show that, under suitable conditions for local approximation, these target quantities characterize the local variation of the latent functions $\eta_l$ and the latent variables $\balpha_i$, respectively.
Importantly, due to the potential nonlinearity, the strength of factors in the local approximation is possibly heterogeneous, which is properly taken into account in our theory. This result appears to be new in the literature, complementing the studies of linear factor models with weak or semi-strong factors \citep{onatski2012asymptotics, Wang-Fan_2017_AoS}, as well as work on entrywise bounds for eigenvectors \citep[e.g.,][]{Abbe-et-al_2020_AoS}.
Furthermore, we note that though the latent functions $\eta_l$'s and latent variables $\balpha_i$'s cannot be separately recovered without additional restrictions, the factors and loadings from local PCA suffice for flexible out-of-sample forecasts based on nonparametric regression and can be used in general causal inference problems with mismeasured confounders. Applications of this approach to causal inference are developed in \cite{Feng_2024_wp}.

Third, building on the first two results, we show that local PCA can consistently estimate the nonlinear factor structure $\eta_l(\balpha_i)$, deriving a convergence rate that is uniform over both individuals and features (Corollary \ref{coro: consistency of latent mean}). 
Under rather general conditions, 
the local PCA estimator can attain a uniform convergence rate that coincides with the optimal one for the infeasible cross-sectional nonparametric estimation of the heterogeneous functions $\eta_l$'s  \citep{stone1982optimal}. 
To the best of our knowledge, this paper is the first to show this rate can be achieved in this general nonlinear factor model,  
contributing to the literature on low-rank approximation of data matrices \citep[][]{udell2019big, fernandez2021low}.

Fourth, in Section \ref{subsec: extensions} and the online Supplemental Appendix we extend the basic nonlinear factor model \eqref{eq: model} by including observable regressors that have  high-rank variation in both dimensions. This provides a new tool for studying, for example, linear regression models with nonlinear fixed effects,  complementing the vast literature on panel regression with interactive fixed effects \citep{Bai_2009_ECMA,Bai-Li_2014_AoS}.

Finally, we apply local PCA to certain matrix completion problems with a few missing entries (Theorem \ref{thm: robustness}). An empirical application studying the effect of tax cuts on economic growth is used to illustrate the potential usefulness of the proposed method in policy evaluation settings such as synthetic controls \citep{Abadie_2020_JEL}.

The rest of the paper is organized as follows. Section \ref{sec: setup} formally sets up the nonlinear factor model. Section \ref{sec: estimation} gives a detailed description of the estimation procedure. Section \ref{sec: main results} presents the
main theoretical results. Section \ref{sec: simulations} summarizes Monte Carlo results. An empirical application to synthetic  control problems is given in Section \ref{sec: application}. 
Section \ref{sec: conclusion} concludes. 
Appendix \ref{sec: appendix matching} collects several technical results that may be of independent interest, including properties of several distance functions and verification of the local approximation of the latent nonlinear factor structure.
The online Supplemental Appendix (SA hereafter) contains all omitted proofs and additional technical and numerical results. 

\section{Nonlinear Factor Model} \label{sec: setup}

Let $\bm{x}_i=(x_{i1}, \cdots, x_{ip})'$ be a $p$-vector of observed variables for the $i$-th unit in the sample. 
We can write the nonlinear factor model as
\begin{equation}\label{eq: HD covariate}
	\bx_{i}=\bmeta(\balpha_i)+\bu_{i}, \quad \E[\bu_{i}|\mathscr{F}]=0, \quad i=1, \cdots, n,
\end{equation}
where $\balpha_i\in\mathbb{R}^{r}$ is a vector of latent variables,
$\bmeta=(\eta_{1}, \cdots, \eta_{p})': \mathbb{R}^r\mapsto\mathbb{R}^p$ is a vector of latent functions, $\bu_{i}=(u_{i1},\cdots,u_{ip})'$ is the idiosyncratic error, and $\mathscr{F}=\sigma(\balpha_i:i\in[n])$ is the $\sigma$-field generated by the latent variables.
In matrix notation,
$$\bX=\bH+\bU,$$ 
where $\bX=(\bx_1,\cdots, \bx_n)'$, $\bH=(\bmeta(\balpha_1),\cdots, \bmeta(\balpha_n))'$ and $\bU=(\bu_1,\cdots, \bu_n)'$ are $n\times p$ matrices. The asymptotic theory in this paper is developed assuming $n$ and $p$ simultaneously increase to infinity.

Conceptually, the latent functions $\bmeta$ can be either fixed or random, depending on the specific context. For example, in the classical linear factor model, or more generally in a possibly nonlinear model with two-way latent heterogeneity, we can write $\eta_l(\balpha_i)=\eta(\balpha_i,\bmf_l)$, where $\eta(\cdot,\cdot)$ is an unknown fixed scalar-valued function and $\bmf_l\in\mathbb{R}^r$ is a vector of feature-specific latent variables. The heterogeneity and potential randomness across $l$ may then be driven by $\bmf_l$. Such a setup is also common in network analysis, including graphon models \citep{Gao-Lu-Zhou_2015_AoS}. More generally, $(\eta_l:1\leq l\leq p)$ may be viewed as a sequence of random functions, such as a functional time series \citep[see, e.g.,][for more details about functional data]{Ferraty-Vieu_2006_book}. For ease of presentation, we treat $\bmeta$ as fixed (conditioned upon) throughout this paper and remain agnostic about its specific structure and potential dependence over $l$. Thus, probability statements concern the sampled latent variables and idiosyncratic errors. For the analysis of the idiosyncratic errors, conditioning on $\mathscr F$ fixes the latent variables and the signal matrix $\bH$.

In the linear factor model $\eta_l(\balpha_i)=\bmf_l'\balpha_i$, due to the linearity of $\bmeta$, the latent mean structure $\bH$ is exactly low-rank ($r\ll p\wedge n$), while in our more general setup, the potential nonlinearity of the latent function $\bmeta$ could make $\bH$ have much higher rank, making traditional methods based on assuming a linear factor structure become inappropriate. 
The main insight in nonlinear factor analysis is that due to the low dimensionality of $\balpha_i$, the variation of the large-dimensional $\bx_{i}$ can still be explained by a few low-dimensional components in a possibly nonlinear way, which implies that $\bH$ is still \textit{approximately} low-rank. 

To gain some intuition, consider the local linear approximation approach widely used in the manifold learning literature (e.g., \citealp{Zhang-Zha_2004_SIAM}):
\begin{equation}\label{eq: linear approx}
\eta_l(\balpha_j)\approx \eta_l(\balpha_i)+\nabla\eta_l(\balpha_i)'(\balpha_j-\balpha_i) \quad \text{for}\quad  \balpha_j\approx \balpha_i,
\end{equation}
where $\nabla\eta_l(\balpha_i)$ is the vector of first-order partial derivatives of $\eta_l$, evaluated at $\balpha_i$.
This representation amounts to an approximately linear factor model with $r+1$ local ``factors'' (i.e., an intercept plus $r$ partial derivatives of $\eta_l$), which motivates the application of usual principal component analysis to the local neighborhood of $\balpha_i$.

More generally, if $\eta_l$ is sufficiently smooth, a higher-order approximation can be employed, which amounts to extracting more \textit{local} factors from the approximation error in equation \eqref{eq: linear approx}. 
Intuitively, as derivatives of $\bmeta(\cdot)$, the local factors in such representations  
signify the degree of nonlinearity of the underlying latent structure, and their loadings reflect the magnitude of different approximation terms.

In practice, since $\balpha_i$ is never observed by the researcher, one has to first employ some indirect strategy to construct the local neighborhood for each unit based on observables, and then conduct principal component analysis locally. This estimation procedure is described in Section \ref{sec: estimation} and then theoretically formalized in Section \ref{sec: main results}.

Now, before we proceed to the estimation procedure, we summarize the regularity conditions on the latent variables $\balpha_i$, latent functions $\bmeta$ and the error terms $\bu_i$ in the next assumption, which are imposed throughout our main analysis.

\begin{assumption}[Regularities] 
	\label{Assumption: Regularities}\leavevmode
	\begin{enumerate}[label=(\alph*)]
		\item 
		$(\balpha_i:1\leq i\leq n)$ is i.i.d. over a compact convex support $\mathcal{A}$ with a density bounded and bounded away from zero;
		
		\item Each $\eta_l(\cdot)$, $1\leq l\leq p$, is bounded by a universal constant uniformly over $\mathcal A$;
		
		\item $(u_{il}: 1\leq i\leq n, 1\leq l\leq p)$ is independent over $i$ and $l$ conditional on $\mathscr{F}$, and 
		for some fixed $\nu>0$, the conditional $(2+\nu)$-th moments of $u_{il}$ are bounded almost surely by a fixed constant, uniformly over $i$ and $l$.
	\end{enumerate}
\end{assumption}
Part (a) is a standard sampling condition. Part (b) imposes only boundedness of the latent functions, while additional smoothness will be useful when verifying our high-level conditions for indirect matching and local approximations imposed later (see Appendix \ref{sec: appendix matching} for details). Part (c) is a standard condition on idiosyncratic errors in factor analysis and graphon estimation. The conditional independence requirement of errors does not exclude unconditional dependence over the two dimensions $i$ and $l$; it is imposed for technical simplicity and can be further relaxed to allow for some weak (conditional) dependence in one or both dimensions.

\subsection{Notation}
 
\textbf{Matrices.}
For a vector $\bv\in\mathbb{R}^\mathsf{d}$, $\|\bv\|=\sqrt{\bv'\bv}$ is the Euclidean norm of $\bv$, and for an $m\times n$ matrix $\bA$,  $\|\bA\|_{\max}=\max_{1\leq i\leq m, 1\leq j\leq n}|a_{ij}|$ is the entrywise sup-norm of $\bA$. $s_{\max}(\bA)$ and $s_{\min}(\bA)$ denote the largest and smallest singular values of $\bA$ respectively.
Moreover,  $\bA_{i\cdot}$ and $\bA_{\cdot j}$ denote the $i$-th row and the $j$-th column of $\bA$ respectively.
$\bm{1}_{\ttd}$ denotes the $\ttd$-vector of ones, and $\bI_m$ denotes the identity matrix of size $m$. 

\textbf{Asymptotics.} For sequences of numbers or random variables,  $a_n\lesssim b_n$ or $a_n=O(b_n)$ denotes $\limsup_n|a_n/b_n|$ is finite, $a_n\lesssim_\P b_n$ or $a_n=O_\P(b_n)$ means that $a_n/b_n$ is bounded in probability, $a_n=o(b_n)$ implies $a_n/b_n\rightarrow 0$, and $a_n=o_\P(b_n)$ implies that $a_n/b_n\rightarrow_\P 0$, where $\rightarrow_\P$ denotes convergence in probability. $a_n\asymp b_n$ implies that $a_n\lesssim b_n$ and $b_n\lesssim a_n$. Similarly, $a_n\gtrsim_\P b_n$ means that $b_n\lesssim_\P a_n$, and $a_n\asymp_\P b_n$ means that both $a_n\lesssim_\P b_n$ and $b_n\lesssim_\P a_n$ hold. The abbreviation ``w.p.a. 1'' means ``with probability approaching one''.

\textbf{Derivatives and others.} For each component function $\eta_l$, $1\leq l\leq p$, set $\nabla^0\eta_l=\eta_l$. For $\ell\geq1$, let $\nabla^{\ell}\eta_l(\cdot)$ be the column vector of all distinct $\ell$th-order partial derivatives of $\eta_l$. Define $\mathscr{D}^{\kappa}\eta_l(\cdot)=(\nabla^{0}\eta_l(\cdot)', \cdots, \nabla^{\kappa}\eta_l(\cdot)')'$, i.e., a column vector that stores all partial derivatives of $\eta_l$ up to order $\kappa\geq 0$. The derivatives on the boundary of $\mathcal A$ are understood as limits with the arguments ranging within the support.
When $\ell=1$, $\nabla\eta_l(\cdot):=\nabla^1\eta_l(\cdot)$ is the gradient vector, and the Jacobian matrix of $\bmeta$ is
$\nabla\bmeta(\cdot):=(\nabla\eta_1(\cdot), \cdots, \nabla\eta_p(\cdot))'$.
For two numbers $a$ and $b$, $a\vee b=\max\{a,b\}$ and $a\wedge b=\min\{a,b\}$. For a finite set $\mathcal{S}$, $|\mathcal{S}|$ denotes its cardinality. We use $[m]$ to denote the set $\{1, 2, \cdots, m\}$ for any positive integer $m$. Finally, let $\delta_{Kp}=(K^{1/2}\wedge p^{1/2})/\sqrt{\log (n\vee p)}$.

\section{Estimation Procedure} \label{sec: estimation}

This section describes the main procedure for local PCA, which typically consists of two steps. 
First, choose a proper function $\rho:\mathbb{R}^p\times\mathbb{R}^p\mapsto\mathbb{R}$ to define the ``distance'' between different units in the sample based on the observed variables. We use the notion of distance in a loose sense, that is, $\rho(\cdot, \cdot)$ does \textit{not} have to satisfy all axioms in the standard definition of the distance function. 
Second, apply principal component analysis to $K$ nearest neighbors of each unit defined by 
the distance calculation in the first step. 
See Algorithm \hyperlink{t2}{1} for a short summary. 
The main tuning parameters in this procedure are the number of nearest neighbors $K$ and the number of extracted principal components  $\ttd_i$ in the local neighborhood of each unit $i$. In the asymptotic analysis, we assume $K\rightarrow\infty$ as $n,p\rightarrow\infty$, while $\ttd_i$ is treated as a fixed tuning parameter.

\begin{table}[hbt!]\label{algorithm}
	\raisebox{\ht\strutbox}{\hypertarget{t2}{}}
	\centering
	\small
	\renewcommand{\arraystretch}{1.5}
	\begin{tabular}{p{\dimexpr\linewidth-2\tabcolsep\relax}}
		\hline\hline
		\textbf{Algorithm 1} (Local Principal Component Analysis) \\
		\hline
		\textbf{Input:} data matrix $\bX\in\mathbb{R}^{n\times p}$, tuning parameters $K$, $\ttd_i$\\
		\textbf{Output:} $\mathcal{N}_i$,  $\widehat{\bF}_{\subi}$, $\widehat{\bLambda}_{\subi}$\\
		Feature-wise (column-wise) split $\bX$ into two submatrices $\bX^{\dagger}\in\mathbb{R}^{n\times p^\dagger}$ and $\bX^{\ddagger}\in\mathbb{R}^{n\times p^\ddagger}$\\
		For each $i=1, \cdots, n$,\\
		(1) use $\bX^{\dagger}$ to obtain the set $\mathcal{N}_i$ of the $K$ nearest neighbors of unit $i$ based on the distance $\rho$:
		\begin{equation*}
		\mathcal{N}_i=\{j_k(i):1\leq k\leq K\}.
		\end{equation*}
	    
		(2) apply PCA to the $K\times p^\ddagger$ submatrix $\bX_{\subi}$ of $\bX^\ddagger$ formed by rows indexed by $\mathcal{N}_i$:
		\begin{equation*}
		(\widehat{\bF}_{\subi}, \widehat{\bLambda}_{\subi})=
		\underset{\scalebox{0.6}{ $\tilde\bF_{\subi}\in\mathbb{R}^{p^\ddagger\times \ttd_i},  \tilde\bLambda_{\subi}\in\mathbb{R}^{K\times \ttd_i}$}}
		{\argmin}\tr\Big[\Big(\bX_{\subi}-\tilde\bLambda_{\subi}\tilde\bF_{\subi}'\Big)\Big(\bX_{\subi}-\tilde\bLambda_{\subi}\tilde\bF_{\subi}'\Big)'\Big]
		\end{equation*}
		 \hspace{1.2em}such that $\frac{1}{p^\ddagger}\tilde\bF_{\subi}'\tilde\bF_{\subi}=\bI_{\ttd_i}$ and $\frac{1}{K}\tilde\bLambda_{\subi}'\tilde\bLambda_{\subi}$ is diagonal.
		\medskip
		\\
		\hline
	\end{tabular}
\end{table}

\subsubsection*{Feature-wise Splitting}
We recommend users separate the $K$ nearest neighbors matching and principal component analysis by feature-wise (i.e., column-wise) splitting, which guarantees desired theoretical properties of local PCA as will be explained below.
Specifically, split the feature (i.e., column) index set $[p]$ of $\bX$ into two non-overlapping subsets: $[p]=\mathcal{C}^\dagger\cup\mathcal{C}^\ddagger$ with $p^\dagger=|\mathcal{C}^\dagger|$,  $p^\ddagger=|\mathcal{C}^\ddagger|$ and $p^\dagger\asymp p^\ddagger\asymp p$. Accordingly, the data matrix $\bX$ is divided into two submatrices $\bX^{\dagger}$ and $\bX^{\ddagger}$ with column indices in $\mathcal{C}^\dagger$ and $\mathcal{C}^\ddagger$ respectively.  $\bH^\dagger$, $\bH^\ddagger$, $\bU^\dagger$ and $\bU^\ddagger$ are defined similarly. By Assumption \ref{Assumption: Regularities}, $\bX^\dagger$ and $\bX^\ddagger$ are independent conditionally on $\mathscr{F}$.
In principle, the two portions of the data only need to be ``approximately'' independent conditional on $\mathscr{F}$, thus allowing for weakly dependent errors. In practice, if columns of $\bX$ are conditionally independent, one could, for example, randomly split the index set into two portions; for unit-time panel data, one could use, for example, the first half of time periods for matching and then apply PCA to the second half,
thus respecting the original time series structure.

\subsubsection*{$K$-Nearest Neighbors Matching} 
\label{subsec: estimation, KNN}

This step makes use of the subsample labeled with $\dagger$, i.e., the submatrix of $\bX$ with column indices in $\mathcal{C}^{\dagger}$. For a generic unit $i\in[n]$, form $\mathcal{N}_i$ by including unit $i$ and its $K-1$ nearest other units in terms of the distance $\rho(\cdot,\cdot)$, with ties resolved by the original indices. Write these $K$ indices in their original index order as
\begin{equation*}
\mathcal{N}_i=\{j_k(i):1\leq k\leq K\}.
\end{equation*}

Our theory is established for generic choices of the distance $\rho(\cdot,\cdot)$ under high-level conditions, which covers many
usual choices in practice such as 
(i) the usual Euclidean distance  
$\rho(\bX^{\dagger}_{i\cdot}, \bX^{\dagger}_{j\cdot})=\frac{1}{\sqrt{p^\dagger}}\|\bX^{\dagger}_{i\cdot}-\bX^{\dagger}_{j\cdot}\|$, 
(ii) the pseudo-max distance   
$\rho(\bX^{\dagger}_{i\cdot}, \bX^{\dagger}_{j\cdot})=\max_{l\neq i, j} |\frac{1}{p^\dagger}(\bX^{\dagger}_{i\cdot}-\bX^{\dagger}_{j\cdot})(\bX^{\dagger}_{l\cdot})'|$
proposed in \cite{Zhang-Levina-Zhu_2017_BIMA}, 
and (iii) the distance of averages $\rho(\bX_{i\cdot}^\dagger, \bX_{j\cdot}^\dagger)
=\frac{1}{p^\dagger}|(\bX_{i\cdot}^\dagger-\bX_{j\cdot}^\dagger)\bm{1}_{p^\dagger}|$. 
These distance functions have different properties and may affect the behavior of the resulting nearest neighbors. For example, the pseudo-max distance can reveal the differences of units in terms of the noise-free factor structure even when the error $\bu_i$ is heteroskedastic conditional on $\mathscr{F}$, while the Euclidean distance cannot in this case. 
More detailed discussion is available in Section \ref{subsec: main result, knn} and Appendix \ref{sec: appendix matching}. 

Note that when the observed variables differ in scale or importance for revealing information on the latent variables, it may be desirable to rescale or reweight different features when searching for nearest neighbors. Such transformations can be viewed as particular choices of the distance. See Remark \ref{remark: other metrics} below for more discussion.

\medskip
\begin{figure}[!h]
	\centering
	\small
	\caption{$K$-Nearest Neighbors Matching}
	\includegraphics[width=.45\textwidth, height=.3\textheight]{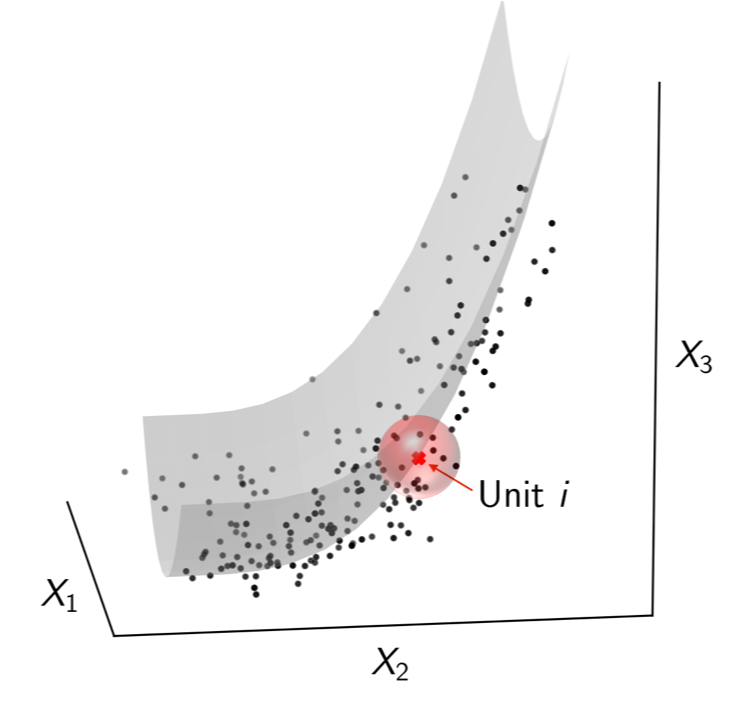}
	\label{figure:knn}
\end{figure}

In general, the distance function $\rho(\cdot,\cdot)$ needs to fulfill two purposes: (i) the distance of observables can be translated into that of noise-free factor structure (``denoising''), and 
(ii) the distance of the noise-free structure can be translated into that of the unobservables $\balpha_i$ (``informativeness''). 
To grasp some intuition, take the pseudo-max distance as an example. For each pair of units, this distance averages the products of their feature differences and the corresponding features of a third unit, then takes the largest absolute average over all choices of the third unit. If the error $u_{il}$ is independent or weakly dependent across $l$, its impact on the distance becomes negligible as the dimension $p^\dagger$ grows large, in which sense we denoise the measurements $\bx_i$ and recover the noise-free component $\bmeta(\balpha_i)$. On the other hand, if $\bmeta(\balpha_i)$ is informative about $\balpha_i$ in the sense of Assumption \ref{Assumption: Matching} below, any two points found close in terms of the factor structure should also be close in terms of the underlying latent variables. Therefore, the nearest neighbors obtained by matching on the observables are similar in terms of the unobservables, which is the key building block of subsequent analysis. 
Figure \ref{figure:knn} gives a conceptual illustration of this idea. An artificial two-dimensional surface, corresponding to the latent noise-free structure, is embedded in a three-dimensional space of observables.
If the requirements outlined before are satisfied,  $K$-NN matching for a particular unit $i$ (colored in red) would 
generate a local neighborhood (the region within the red ball).

\subsubsection*{Local Principal Component Analysis}
\label{subsec: estimation, PCA}
This step makes use of the subsample labeled with $\ddagger$, i.e., the submatrix of $\bX$ with column indices in $\mathcal{C}^{\ddagger}$.   
Given a set of nearest neighbors $\mathcal{N}_i$ from the previous step, define $\bX_{\subi}$ as the $K\times p^\ddagger$ submatrix of $\bX^\ddagger$ formed by rows indexed by $\mathcal{N}_i$. Define $\bH_{\subi}$ and $\bU_{\subi}$ analogously from $\bH^\ddagger$ and $\bU^\ddagger$, respectively.

As illustrated in \eqref{eq: linear approx}, if the neighbors we find for each unit $i$ are truly close in the latent variables, i.e., $\balpha_j\approx\balpha_i$ for all $j\in\mathcal{N}_i$, then the possibly full-rank matrix $\bH_{\subi}$ could be locally approximated by a low-rank structure:
$\bH_{\subi}\approx
\bLambda_{\subi}\bF_{\subi}'
$,
with $\bF_{\subi}\in\mathbb{R}^{p^\ddagger\times \ttd_i}$ and $\bLambda_{\subi}\in\mathbb{R}^{K\times \ttd_i}$
termed \textit{local factors} and \textit{local factor loadings} respectively, and 
$\ttd_i$ is a user-specified parameter that governs the number of approximation terms. 
In general, we can write an approximately linear factor model:
\begin{equation}\label{eq: local representation}
	\bX_{\subi}=\bLambda_{\subi}\bF_{\subi}'+\bXi_{\subi}+\bU_{\subi},
\end{equation} 
where
$\bXi_{\subi}$ is the corresponding approximation/smoothing bias. We will formally define the target local factors and their loadings as the solution to the noise-free PCA problem (see \eqref{eq: lpca noise-free} below).

\eqref{eq: local representation} motivates the application of PCA to $\bX_{\subi}$:
\begin{equation}\label{eq: lpca}
	(\widehat{\bF}_{\subi}, \widehat{\bLambda}_{\subi})=
	\underset{\scalebox{0.7}{ $\tilde\bF_{\subi}\in\mathbb{R}^{p^\ddagger\times \ttd_i},  \tilde\bLambda_{\subi}\in\mathbb{R}^{K\times \ttd_i}$}}
	{\argmin}\tr\Big[\Big(\bX_{\subi}-\tilde\bLambda_{\subi}\tilde\bF_{\subi}'\Big)\Big(\bX_{\subi}-\tilde\bLambda_{\subi}\tilde\bF_{\subi}'\Big)'\Big]
\end{equation}
such that $\frac{1}{p^\ddagger}\tilde\bF_{\subi}'\tilde\bF_{\subi}=\bI_{\ttd_i}$ and $\frac{1}{K}\tilde\bLambda_{\subi}'\tilde\bLambda_{\subi}$ is diagonal, with its diagonal entries arranged in nonincreasing order.

The idea underlying \eqref{eq: lpca} is similar to the step of learning local tangent spaces in  \cite{Zhang-Zha_2004_SIAM}. The main difference is that  $K$-NN matching and PCA in the above procedure are conducted on different features. This is motivated by the fact that searching for nearest neighbors has implicitly used the information on  $\bu_i$. Without feature-wise splitting, for units within the same local neighborhood, the nonlinear factor components $\bLambda_{\subi}\bF_{\subi}'+\bXi_{\subi}$ would be correlated with the noise $\bU_{\subi}$, rendering the standard PCA technique inapplicable. Feature-wise splitting is a simple remedy, when the noise $u_{il}$ is independent (or weakly dependent) across $l$.

Figure \ref{figure:local tangent} gives a conceptual illustration of local PCA.  Units around the red dot lie approximately on a (local) linear tangent plane (colored in purple). Intuitively, this approximation is analogous to the local linear regression in the nonparametrics literature, though conditioning variables in this context are unobserved. More generally, if more leading local factors can be distinguished from the noise, then a local \textit{nonlinear} principal surface can be constructed for a higher-order approximation of the underlying surface. However, the extracted principal components from the noisy data may not always be informative about the latent structure $\bH_{\subi}$. Instead, they could be partially or completely determined by the noise matrix $\bU_{\subi}$, and the potential local degeneracy of the nonlinear structure may further complicate this issue. Thus, we recommend users only take a few leading principal components associated with large eigenvalues. See formal discussion in Section \ref{subsec: main result, pca}.

\medskip
\FloatBarrier
\begin{figure}[!h]
	\centering
	\small
	\caption{Local Tangent Space Approximation}
	\includegraphics[width=.45\textwidth, height=.3\textheight]{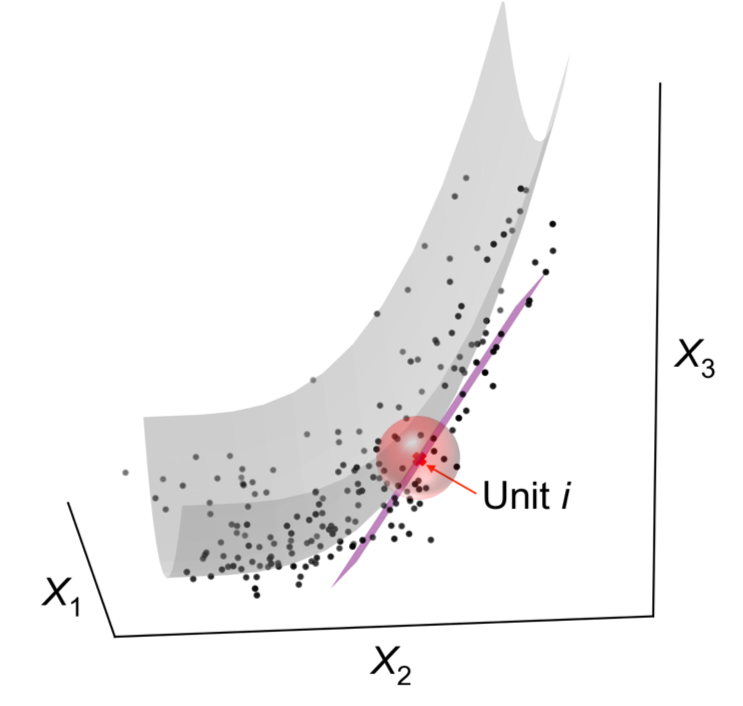}
	\label{figure:local tangent}
\end{figure}
\FloatBarrier

\section{Main Results} \label{sec: main results}

To theoretically formalize the estimation procedure in Section \ref{sec: estimation}, we need to (i) ensure the closeness between $\balpha_j$ and $\balpha_i$ for $j\in\mathcal{N}_i$, and (ii) show that the linear factor structure locally extracted from the observables is ``consistent'' in some proper sense for the nonlinear factor structure (``signals'') of our interest. The first task is nontrivial since $\balpha_i$ is not observed by the researcher, and as described above, some indirect strategy is usually adopted to construct the desired neighborhood. Then, the key challenge is translating the distance measured in observables into that in unobservables  under appropriate conditions.
On the other hand, the second task is complicated by the fact that the factors in approximation \eqref{eq: local representation} are of different strength and might be degenerate at some point(s) $\balpha\in\mathcal{A}$. In the following we will discuss each task and provide formal results.

\subsection{$K$-Nearest Neighbors Matching}\label{subsec: main result, knn}
The local neighborhood described in Section \ref{sec: estimation} is constructed indirectly based on observed noisy measurements $\bx_i$ of the latent $\balpha_i$. To show the closeness of the resulting nearest neighbors, we typically need to guarantee the noise $\bu_i$ is approximately negligible in the distance calculation and the noise-free structure $\bmeta(\balpha_i)$ is informative about $\balpha_i$, which are formalized in Assumption \ref{Assumption: Matching}.
\begin{assumption}[Indirect Matching]
	\label{Assumption: Matching}\leavevmode
	For some fixed constant $\rho_0\in\mathbb{R}$, fixed positive constants $\lvarsig$ and $\uvarsig$, and some positive sequence $a_p=o(1)$, the following conditions hold:
	\begin{enumerate}[label=(\alph*)]
		\item $\underset{1\leq i,j\leq n,\;i\neq j}{\max}\;|\rho(\bX_{i\cdot}^\dagger, \bX_{j\cdot}^\dagger)-
		\rho(\bH_{i\cdot}^\dagger, \bH_{j\cdot}^\dagger)
		-\rho_0|\lesssim_\P a_p$;
		
		\item 
		$\underset{1\leq i,j\leq n \atop \balpha_i\neq \balpha_j}{\max}\;\frac{\rho(\bH_{i\cdot}^\dagger,\, \bH_{j\cdot}^\dagger)}{\|\balpha_i-\balpha_j\|^{\uvarsig}}\lesssim_\P 1$ and 
		$\underset{1\leq i,j\leq n \atop \balpha_i\neq \balpha_j}{\min}\;\frac{\rho(\bH_{i\cdot}^\dagger,\, \bH_{j\cdot}^\dagger)}{\|\balpha_i-\balpha_j\|^{\lvarsig}}\gtrsim_\P 1$.
	\end{enumerate}
\end{assumption}

Condition (a) formalizes the idea of denoising the data by choosing a proper distance $\rho(\cdot, \cdot)$. It guarantees that the distance of the observables between any pair of units is approximately determined by that of the noise-free components, up to a fixed constant $\rho_0$. Then, the closeness in terms of the observables can be translated into that of the latent factor structure. This requirement is
usually mild, if we have many features and average them in a proper way.

On the other hand, condition (b) concerns the noise-free structure only. 
The upper bound condition is mild. For a particular choice of $\rho(\cdot,\cdot)$, it can usually be deduced from Lipschitz continuity of $\bmeta$. Appendix \ref{sec: appendix matching} imposes the smoothness needed for this primitive verification.
By contrast, the lower bound  is the key requirement for $\bx_i$ to be informative about $\balpha_i$. Intuitively, it says if $\bH^\dagger_{i\cdot}$ is close to $\bH^\dagger_{j\cdot}$ in terms of the distance $\rho(\cdot,\cdot)$, 
$\balpha_i$ needs to be close to $\balpha_j$. The parameters $\lvarsig$ and $\uvarsig$ govern how the distance of the unobservables and that of the observables are linked. Typically, these requirements depend not only on the nonlinear factor structure itself, but also the chosen distance function.
This informativeness requirement is conceptually related to the completeness conditions used in econometric identification problems \citep{schennach2020mismeasured}. Intuitively, completeness requires a sufficiently rich family of conditional distributions to identify unknown functions. In our setting, the lower bound in (b) requires separation between latent variables to produce a corresponding separation between their noise-free feature vectors under the chosen distance.
However, our informativeness requirement applies to the features collectively: individual features may be uninformative, and the informative features need not be identified in advance.

To gain some intuition about the two conditions, we discuss several choices of $\rho(\cdot,\cdot)$ that are common in the literature. Formal technical results are deferred to Appendix \ref{sec: appendix matching}.

\begin{exmp}[Squared Euclidean distance]\label{exmp: euclidean}
	Let $\rho(\bv_i, \bv_j)=\frac{1}{p^\dagger}\|\bv_i-\bv_j\|^2$ for any $\bv_i, \bv_j\in\mathbb{R}^{p^\dagger}$.
	Since $\bU$ and $\mathscr{F}$ are mean independent and $u_{il}$ is independent over $i$ and $l$ conditional on $\mathscr{F}$, 
	\begin{align*}
	\frac{1}{p^\dagger}\|\bX^\dagger_{i\cdot}-\bX^\dagger_{j\cdot}\|^2
	&\approx
	\frac{1}{p^\dagger}\|\bH^\dagger_{i\cdot}-\bH^\dagger_{j\cdot}\|^2+
	\frac{1}{p^\dagger}\|\bU^\dagger_{i\cdot}\|^2+
	\frac{1}{p^\dagger}\|\bU^\dagger_{j\cdot}\|^2.
	\end{align*}
	It implies that to make condition \ref{Assumption: Matching}(a) hold one has to assume (conditional) homoskedasticity of $\bu_i$:  $\frac{1}{p^\dagger}\sum_{l\in\mathcal{C}^\dagger}\E[u_{il}^2|\mathscr{F}]$ is a fixed constant $c_u>0$ for all $i\in [n]$, which is usually unrealistic.  
	
	In Appendix \ref{sec: appendix matching} we also verify condition \ref{Assumption: Matching}(b) under intuitive sufficient conditions. The key requirement is that for every $\varepsilon>0$,
	\begin{equation}\label{eq: nonsingularity}
	\underset{\Delta\rightarrow 0}{\lim}\;\underset{n,p^\dagger\rightarrow\infty}{\limsup}\;
	\P\Big\{\underset{i\in[n]}{\max}\;
	\underset{j:\rho(\bH^\dagger_{i\cdot},\bH^\dagger_{j\cdot})<\Delta}{\max}\;
	\|\balpha_i-\balpha_j\|>\varepsilon\Big\}=0.
	\end{equation}
	This can be understood as an ``identification'' condition for $\balpha_i$, which says 
	the difference in latent $\balpha_i$ can be revealed by the noise-free structure 
	$\bH^\dagger_{i\cdot}$ as $n,p^\dagger\rightarrow\infty$, though exact identification of $\balpha_i$ is impossible without further restrictions.
	Intuitively, this is an asymptotic version of requiring the map from $\balpha_i$ to $\bH^\dagger_{i\cdot}$ to be one-to-one with a uniformly continuous inverse. Importantly, this requirement applies to the feature vector as a whole, and individual feature functions $\eta_l$ need not be one-to-one. See Section \ref{SA-sec: discussion of conditions} of the SA for more details.
	
	When conditions described above hold, we formally show in Theorem \ref{thm: IMD hom} that
	\[
	\pushQED{\qed}
	\renewcommand{\qedsymbol}{\large$\lrcorner$} 
	a_p=\sqrt{(\log n)/p}, \quad \rho_0=2c_u,\quad \lvarsig=\uvarsig=2.
	\qedhere\popQED
	\]
\end{exmp}

\begin{exmp}[Pseudo-max distance]\label{exmp: pseudo-max}
	Let $\rho(\bv_i, \bv_j)=\frac{1}{p^\dagger}\max_{\ell\neq i, j}|(\bv_i-\bv_j)'\bv_\ell|$ for a sequence of $p^\dagger$-vectors $(\bv_i: i\in[n])$, which was proposed by \cite{Zhang-Levina-Zhu_2017_BIMA} in the graphon estimation context.
	Since the distance between any pair of vectors is measured using a \textit{third} vector, we have 
	\[
	\frac{1}{p^\dagger}(\bX^\dagger_{i\cdot}-\bX^\dagger_{j\cdot})(\bX^\dagger_{\ell\cdot})'
	\approx 
	\frac{1}{p^\dagger}(\bH^\dagger_{i\cdot}-\bH^\dagger_{j\cdot})(\bH^\dagger_{\ell\cdot})'.
	\]
	The conditional homoskedasticity assumption is unnecessary in this case, making the pseudo-max distance more appealing than the (squared) Euclidean distance. 
	
	To verify condition \ref{Assumption: Matching}(b), we impose the same condition as in \eqref{eq: nonsingularity} except that $\rho(\cdot,\cdot)$ is taken to be the pseudo-max distance. We also require nondegeneracy of the normalized Jacobian Gram matrix, along with other technical regularity conditions stated in Appendix \ref{sec: appendix matching}.
	
	When conditions described above hold, we formally show in Theorem \ref{thm: IMD hsk} that 
	\[
	\pushQED{\qed}
	\renewcommand{\qedsymbol}{\large$\lrcorner$} 
	a_p=\sqrt{(\log n)/p},  \quad \rho_0=0, \quad \lvarsig=\uvarsig=1.
	\qedhere\popQED
	\]
\end{exmp}

\begin{exmp}[Distance of averages]\label{exmp: dist of average}
	Let $\rho(\bv_i, \bv_j)=\frac{1}{p^\dagger}|\bm{1}_{p^\dagger}'(\bv_i-\bv_j)|$ for any $\bv_i, \bv_j\in\mathbb{R}^{p^\dagger}$. This amounts to simply averaging the features in $\mathcal{C}^\dagger$ and then taking the distance of the scalar-valued aggregate feature between each pair of units.
	In this case, condition \ref{Assumption: Matching}(a) can be easily verified based on the mild assumption that $\E[u_{il}|\mathscr{F}]=0$. 
	
	To verify part (b), the key condition typically required is that (the limit of) the average function $\frac{1}{p^\dagger}\sum_{l\in\mathcal{C}^\dagger}\eta_l(\cdot)$ is \textit{strictly monotonic}, which is only possible for a scalar $\alpha_i$.
	Even in the scalar case, this condition may be too stringent. For example, in a linear factor model with 
	$\eta_l(\alpha_i)=f_l\alpha_i$, $\frac{1}{p^\dagger}\sum_{l\in\mathcal{C}^\dagger}\eta_{l}(\alpha_i)$ could be nearly uninformative about $\alpha_i$ if $\frac{1}{p^\dagger}\sum_{l\in\mathcal{C}^\dagger}f_{l}\approx 0$. However, many $\eta_l(\alpha_i)$'s remain informative about $\alpha_i$ as long as $f_l$'s are nonzero,  and thus the difference in latent variables may still be revealed using, for instance, the pseudo-max distance discussed before.
	
	When conditions described above hold, we can formally show in Theorem \ref{thm: dist of average} that 
	\[
	\pushQED{\qed}
	\renewcommand{\qedsymbol}{\large$\lrcorner$} 
	a_p=\sqrt{(\log n)/p},  \quad \rho_0=0, \quad \lvarsig=\uvarsig=1.
	\qedhere
	\popQED
	\]
\end{exmp}

\begin{remark}[Data transformations]\label{remark: other metrics}
	We emphasize that $\rho(\cdot, \cdot)$ in Assumption \ref{Assumption: Matching} is generic, which can accommodate transformations of the original features other than the average in Example \ref{exmp: dist of average}. 
	For instance, for each $i\in [n]$, define a possibly vector-valued function $\bh_i$ that transforms unit $i$'s observed features 
	$\bX^\dagger_{i\cdot}$ into a $\ttd_h$-vector of new features $\bh_i(\bX^\dagger_{i\cdot})$, and then conduct matching on the transformed features in terms of the Euclidean norm. In this case, the distance between units $i$ and $j$ is given by
	\[
	\rho(\bX^\dagger_{i\cdot}, \bX^\dagger_{j\cdot})=
	\|\bh_i(\bX^\dagger_{i\cdot})-\bh_j(\bX^\dagger_{j\cdot})\|.
	\]
	In practice, introducing such transformations may be useful since it allows for rescaling or reweighting different observed variables to obtain features that are more informative about the latent variables.
\end{remark}

Now, we present our first main result, which characterizes the indirect matching discrepancy of nearest neighbors.
\begin{thm}[Indirect Matching] \label{thm: IMD}
	Suppose that Assumption \ref{Assumption: Regularities}(a) and Assumption \ref{Assumption: Matching} hold.
	If $\frac{\log (n/K)}{K}=o(1)$ and $\frac{K\log n}{n}=o(1)$, then, 
	\[
	\max_{1\leq i\leq n}\max_{1\leq k\leq K}
	\|\balpha_i-\balpha_{j_k(i)}\|
	\lesssim_\P (K/n)^{\uvarsig/(\lvarsig r)}+a_p^{1/\lvarsig}.
	\]
\end{thm}

As shown in the above theorem, the matching can be made up to errors consisting of two terms in an asymptotic sense. The first part $(K/n)^{\uvarsig/(\lvarsig r)}$ reflects the direct matching discrepancy for $\balpha_i$. It grows quickly with the number of latent variables, which coincides with the results in the nearest neighbors matching literature (e.g., \citealp{Gyorfi-et-al_2002_bookchapter}). The second term $a_p^{1/\lvarsig}$ arises from the existence of the idiosyncratic error $\bu_i$. In the three examples described above, the distance is defined based on certain averages across different features, and thus the impact of $\bu_i$ vanishes as $p$ grows large.

Note that if $\balpha_i$'s were observed, matching could be directly implemented on them with the number of matches $K$ fixed. In this paper, however, $\balpha_i$ is unobservable, and matching can only be done on their noisy measurements, leading to the indirect matching discrepancy characterized by the second term above. Using a fixed (or small) number of nearest neighbors is unable to further reduce bias and thus is not recommended in this scenario.

Theorem \ref{thm: IMD} generalizes the existing results in the literature that rely on specific metrics and provides a way to precisely quantify the indirect matching discrepancy. For example, \cite{Zhang-Levina-Zhu_2017_BIMA} proposes the pseudo-max distance for neighborhood smoothing in the graphon estimation context, but they provide no results regarding the distance in the latent variables. Moreover, if the conditions specified in Theorem \ref{thm: IMD hsk} hold and $\eta_l(\balpha_i)=\eta_l(\alpha_i, \varpi_l)$ for scalar latent variables $\alpha_i$ and $\varpi_l$, Theorem \ref{thm: IMD} implies that the neighborhood smoothing (local average) estimator of the factor structure can achieve a sup-norm convergence rate of order $O(n^{-\frac{1}{3}})$, up to $\log n$ terms, which improves upon the $L^2$-type convergence rate of order $O(n^{-1/4})$, up to $\log n$ terms, given in \cite{Zhang-Levina-Zhu_2017_BIMA}. See more detailed discussion about the uniform convergence rate in Section \ref{subsec: main result, pca}.

\subsection{Local Principal Component Analysis}\label{subsec: main result, pca}
Now, we proceed to discuss the properties of local principal component analysis. 
Recall that PCA is applied to the submatrix $\bX_{\subi}$ of $\bX$ formed by a subset of observed features indexed by $\ddagger$ for the $K$ nearest neighbors of each unit $i$. 
We define the target local factors and loadings as the principal components of the noise-free local matrix $\bH_{\subi}$. Specifically, let
\begin{equation}\label{eq: lpca noise-free}
	(\bF_{\subi}, \bLambda_{\subi}):=
	\underset{\scalebox{0.7}{ $\tilde\bF_{\subi}\in\mathbb{R}^{p^\ddagger\times \ttd_i},  \tilde\bLambda_{\subi}\in\mathbb{R}^{K\times \ttd_i}$}}
	{\argmin}\tr\Big[\Big(\bH_{\subi}-\tilde\bLambda_{\subi}\tilde\bF_{\subi}'\Big)\Big(\bH_{\subi}-\tilde\bLambda_{\subi}\tilde\bF_{\subi}'\Big)'\Big]
\end{equation}
such that $\bF_{\subi}'\bF_{\subi}/p^\ddagger=\bI_{\ttd_i}$ and $\bLambda_{\subi}'\bLambda_{\subi}/K$ is diagonal, with its diagonal entries arranged in nonincreasing order. Accordingly, 
$\bXi_{\subi}=\bH_{\subi}-\bLambda_{\subi}\bF_{\subi}'$
is the approximation error from this noise-free local PCA problem. By construction, $\bUpsilon_{\subi}:=(\bLambda_{\subi}'\bLambda_{\subi}/K)^{1/2}$ is diagonal, and then we can write $\bUpsilon_{\subi}=\diag\{\upsilon_{1,\subi}, \cdots, \upsilon_{\ttd_i,\subi}\}$, where $\upsilon_{1, \subi}\geq \upsilon_{2,\subi}\geq\cdots\geq \upsilon_{\ttd_i,\subi}>0$. Thus, $\upsilon_{j,\subi}^2$ is the $j$th leading eigenvalue of $\bH_{\subi}'\bH_{\subi}/(p^\ddagger K)$.

Now, we impose some high-level conditions that directly restrict the noise-free principal components and will govern the local approximation quality of the above procedure. Recall we define $\delta_{Kp}=(K^{1/2}\wedge p^{1/2})/\sqrt{\log (n\vee p)}$.
\begin{assumption}[Local Approximation] \label{Assumption: LPCA}
	\leavevmode
	\begin{enumerate}[label=(\alph*)]
		\item
		$\max_{i\in[n]}\|\bLambda_{\subi}\bUpsilon_{\subi}^{-1}\|_{\max}\lesssim_\P 1$, and $\max_{i\in[n]}\|\bF_{\subi}\|_{\max}\lesssim_\P 1$;

		\item
		For all $i\in[n]$ and $1\leq j<\ttd_i$, either $\upsilon_{j,\subi}/\upsilon_{j+1,\subi}\lesssim_\P 1$ or $\upsilon_{j+1,\subi}/\upsilon_{j,\subi}=o_\P(1)$;

		\item $\max_{i\in[n]}\delta_{Kp}^{-1}/\upsilon_{\ttd_i,\subi}=o_\P(1)$ and $\max_{i\in[n]}\mathfrak r_{\eta,\subi}/\upsilon_{\ttd_i,\subi}=o_\P(1)$, where $\mathfrak r_{\eta,\subi}:=\|\bXi_{\subi}\|_{\max}$.
	\end{enumerate}
\end{assumption}

Part (a) requires uniformly bounded entries of the factors and normalized loadings. Under the PCA normalization, these entry bounds prevent a component from being concentrated in a vanishing fraction of features or neighboring units. These are natural regularity conditions under the local Taylor construction discussed below.
Part (b) formalizes the possibility that the local approximation terms may have different asymptotic strengths. This is a mild technical condition that allows consecutive components to have either comparable strengths or asymptotically different orders of magnitude.
The sampling-noise condition $\max_{i\in[n]}\delta_{Kp}^{-1}/\upsilon_{\ttd_i,\subi}=o_\P(1)$ in part (c) is key for the local factors to be consistently estimable. Intuitively, $\upsilon_{\ttd_i,\subi}$ measures the strength of the weakest factor one desires to extract, which must dominate the sampling noise of order $\delta_{Kp}^{-1}$.
Finally, the approximation-error condition in part (c) requires the largest entrywise approximation error to be negligible relative to the weakest retained component.

These requirements are best understood in the context of a local Taylor approximation of sufficiently smooth latent functions $\bmeta$. For example, let $\widetilde\blambda_{\subi}(\balpha)$ be an $r$-variate monomial basis of degree no greater than $m-1$ centered at $\balpha_i$ (including the constant term), and let $\widetilde\bLambda_{\subi}$ collect its evaluations over the units in $\mathcal N_i$. The associated coefficient vectors consist of derivatives of the latent functions. We can then write $\bH_{\subi}=\widetilde\bLambda_{\subi}\widetilde\bF_{\subi}'+\widetilde\bXi_{\subi}$, where $\widetilde\bLambda_{\subi}$ and $\widetilde\bF_{\subi}$ are auxiliary Taylor-approximation objects and $\widetilde\bXi_{\subi}$ is the Taylor remainder.

Theorem \ref{thm: verify LPCA} in Appendix \ref{sec: appendix matching} uses this representation to verify Assumption \ref{Assumption: LPCA} under primitive conditions. When indirect matching yields neighbors close in the latent variables, monomials of different degrees have different scales. Let $h_n$ denote a rate for the maximum matching discrepancy in the latent variables. For instance, under the conditions of Theorem \ref{thm: IMD}, one may take $h_n=(K/n)^{1/r}$ when $\bar\varsigma=\underline\varsigma$ and $a_p^{1/\lvarsig}\lesssim (K/n)^{1/r}$. The constant monomial equals one, while the linear and higher-order monomials are of orders $O_\P(h_n),\ldots,O_\P(h_n^{m-1})$, uniformly over the neighborhoods. Under the additional primitive conditions in Theorem \ref{thm: verify LPCA}, the noise-free PCA targets $\bF_{\subi}$ and $\bLambda_{\subi}$ inherit the relevant properties of the auxiliary matrices $\widetilde\bF_{\subi}$ and $\widetilde\bLambda_{\subi}$. Specifically, the entry bounds in part (a) hold, and the component strengths have orders $1,h_n,\ldots,h_n^{m-1}$, satisfying part (b). Moreover, $\mathfrak r_{\eta,\subi}\lesssim_\P h_n^m$ and $\upsilon_{\ttd_i,\subi}\asymp_\P h_n^{m-1}$ uniformly over $i$. Thus, the approximation-error condition in part (c) follows from $h_n=o(1)$, while its sampling-noise condition requires $\delta_{Kp}^{-1}=o(h_n^{m-1})$.

Among the sufficient conditions imposed in Theorem \ref{thm: verify LPCA} to verify Assumption \ref{Assumption: LPCA}, the main substantive requirement is nondegeneracy of the derivative coefficient vectors. The degree of nonlinearity may be heterogeneous across evaluation points, so the derivative vectors may vanish or become collinear at some points. A universal choice such as $\ttd_i=\binom{r+m-1}{r}$ therefore need not satisfy this condition, motivating the use of different numbers of local factors across neighborhoods. To gain more intuition, we discuss two concrete examples below.

\begin{exmp}[Linear Factor Model]
	Consider a linear factor model with an intercept: $\eta_l(\alpha_i)=c_0+\varpi_{l}\alpha_{i}$.
	With $c_0\neq 0$ and $\varpi_l$ varying sufficiently across $l$, the noise-free local matrix has rank two and can be represented without approximation error by taking $\ttd_i=2$. By contrast, when $c_0=0$, it has rank one and $\ttd_i=1$ suffices if $\frac{1}{p^\ddagger}\sum_{l\in\mathcal{C}^\ddagger}\varpi_l^2\gtrsim1$.
	\hfill\large$\lrcorner$
\end{exmp}

\begin{exmp}[Cosine function]
		Consider a nonlinear factor model based on a cosine transformation: $\eta_l(\alpha_i)=\cos(\varpi_l\alpha_i)$, where the fixed values $\varpi_l$ for $l\in\mathcal C^\ddagger$ are approximately uniformly spread over $[0,1]$. To achieve linear approximation of $\eta_l(\cdot)$ locally at $\alpha_i$, we generally need a linear two-factor model: $\eta_l(\alpha)\approx \cos(\varpi_l\alpha_i) -\varpi_l\sin(\varpi_l\alpha_i)(\alpha-\alpha_i)$ for $\alpha\approx\alpha_i$.
		Let $\bg(\varpi;\alpha)=(\cos(\varpi\alpha),\,-\varpi\sin(\varpi\alpha))'$ collect the function value and its first derivative with respect to $\alpha$. The normalized Gram matrix of these vectors can then be approximated by
		
		\[
		\frac{1}{p^\ddagger}\sum_{l\in\mathcal C^\ddagger}\bg(\varpi_l;\alpha_i)\bg(\varpi_l;\alpha_i)'\approx\int_{[0,1]}\bg(\varpi;\alpha_i)\bg(\varpi;\alpha_i)'\,\mathrm{d}\varpi.
		\]
		
	At $\alpha_i=0$, the first derivative is zero for every $\varpi$, and a single retained component suffices for the local linear approximation. Away from this point, the integral matrix is generally non-degenerate, and the noise-free local PCA may retain two components.
    \hfill\large$\lrcorner$
\end{exmp}

Appendix \ref{sec: appendix matching} provides a general result about the verification of Assumption \ref{Assumption: LPCA} under primitive conditions, and Section \ref{SA-sec: discussion of conditions} in the SA gives further discussion with concrete examples.

Now, we present the uniform convergence properties of the estimated local factors and loadings, the second main result of this paper. 

\begin{thm}
	\label{thm: uniform convergence of factors}
	Under Assumptions \ref{Assumption: Regularities}, \ref{Assumption: Matching} and \ref{Assumption: LPCA}, if  $(np)^{\frac{2}{\nu}}\delta_{Kp}^{-2}\lesssim 1$, then
	there exist matrices $\bR_{\subi}$, $i\in[n]$, such that
	\[
	\begin{split}
		&\max_{i\in[n]}\max_{j\in[\ttd_i]}\;\frac{\|\widehat{\bF}_{\cdot j,\subi}-\bF_{\subi}((\bR_{\subi}')^{-1})_{\cdot j}\|_{\max}}{(\delta_{Kp}^{-1}+\mathfrak r_{\eta,\subi})\upsilon_{j,\subi}^{-1}}\lesssim_\P 1,\\
		&\max_{i\in[n]}\max_{j\in[\ttd_i]}\;\frac{\|\widehat{\bLambda}_{\cdot j,\subi}-\bLambda_{\subi}\bR_{\cdot j,\subi}\|_{\max}}{\delta_{Kp}^{-1}+\mathfrak r_{\eta,\subi}}\lesssim_\P 1.
	\end{split}
	\]
	Moreover, $1\lesssim_\P \underset{i\in [n]}{\min}\;s_{\min}(\bR_{\subi})\leq 
	\underset{i\in [n]}{\max}\;s_{\max}(\bR_{\subi})\lesssim_\P 1$.
\end{thm}

Recall that Theorem \ref{thm: IMD} showed that the distances between \(\balpha_i\) and its nearest neighbors \(\balpha_{j_k(i)}\) converge to zero as \(n\) and \(p\) diverge. Theorem \ref{thm: uniform convergence of factors} further shows that the target loading and factor matrices can be consistently estimated up to a rotation. As discussed above, under suitable conditions for local approximation, the loadings may represent transformations of local variation in the latent variables, while the factors may represent transformations of derivatives of the latent functions.

The result above establishes convergence in the matrix sup norm, uniformly over all local neighborhoods $\subi$, $i\in[n]$.
The estimation error of both $\widehat{\bF}_{\subi}$ and $\widehat{\bLambda}_{\subi}$ consists of two parts.
The sampling-error term in each rate, i.e., $\delta_{Kp}^{-1}\upsilon_{\ell, \subi}^{-1}$ and $\delta_{Kp}^{-1}$, reflects the estimation variance. Since latent variables $\balpha_i$'s are not observed, the variability of the local factor and loading estimators relies on the sizes of both dimensions $K$ and $p$, which is akin to the results in classical linear factor analysis \citep{Bai_2003_ECMA}.
The remaining term in each rate, i.e., $\mathfrak r_{\eta,\subi}\upsilon_{\ell, \subi}^{-1}$ and $\mathfrak r_{\eta,\subi}$, reflects the approximation error of the noise-free local PCA, which vanishes in an exact linear factor model when all signal components are retained.

By Assumption \ref{Assumption: LPCA}(b), the local factors in $\bF_{ \subi}$ may be of heterogeneous strength, reflected by the magnitude of the associated loadings. Consequently, the estimation error of $\widehat{\bF}_{\subi}$ involves a penalty factor $\upsilon_{\ell, \subi}^{-1}$ that is inversely related to the strength of different factors. This is in line with the intuition: higher-order approximation terms are weaker signals about the latent structure and thus are estimated less precisely. As emphasized before, the rate condition $\delta_{Kp}^{-1}/\upsilon_{\ttd_i,\subi}=o_\P(1)$ in Assumption \ref{Assumption: LPCA}(c) is key to ensure the weakest local factors can still be distinguished from the sampling noise in Equation \eqref{eq: local representation}.

To illustrate the rate conditions, suppose $p\asymp n$ and consider again the local Taylor approximation of degree $m-1$ discussed above. Under the stated primitive conditions, $h_n=(K/n)^{1/r}$ is the rate of the maximum matching discrepancy, and the weakest retained component has strength $\upsilon_{\ttd_i,\subi}\asymp_\P h_n^{m-1}$ uniformly over $i$.
The sampling-noise condition in Assumption \ref{Assumption: LPCA}(c) then requires $(n/K)^{\frac{2m-2}{r}}=o(K/\log n)$, while the additional rate restriction in Theorem \ref{thm: uniform convergence of factors} simplifies to $n^{\frac{4}{\nu}}\lesssim K/\log n$. For $K=n^{A}$, these two restrictions hold if $A>\max\{\frac{4}{\nu}, \frac{2m-2}{2m-2+r}\}$. In particular, for sufficiently large $\nu$, one may take $K\asymp n^{\frac{2m}{2m+r}}$, the MSE-optimal rate for the number of nearest neighbors in the infeasible cross-sectional regression of $x_{il}$ on $\balpha_i$.

Finally, using Theorem \ref{thm: uniform convergence of factors}, we immediately have the following corollary.
\begin{coro}\label{coro: consistency of latent mean}
	Under the conditions of Theorem \ref{thm: uniform convergence of factors},
	\[
	\max_{i\in[n]}\;\frac{\Big\|\widehat{\bLambda}_{\subi}\widehat{\bF}_{ \subi}'-\bH_{\subi}\Big\|_{\max}}{\delta_{Kp}^{-1}+\mathfrak r_{\eta,\subi}}\lesssim_\P 1.
	\]
\end{coro}
This corollary shows that the latent nonlinear factor component can be consistently estimated, and this error bound holds uniformly over all units and features in $\mathcal{C}^\ddagger$.
Note that if $\balpha_i$'s were observed, the natural alternative to estimate the heterogeneous functions $\eta_l$'s (and thus $\bH$) would be the cross-sectional nonparametric regression of $x_{il}$ on $\balpha_i$ for each $l\in[p]$, whose optimal uniform convergence rate is $O(n^{-\frac{m}{2m+r}})$, up to $\log n$ terms \citep{stone1982optimal}.
Corollary \ref{coro: consistency of latent mean} shows that this optimal rate can be attained  by the local PCA estimator
when $K\lesssim p$, $\bar{\varsigma}=\underline{\varsigma}$, and an optimal number of nearest neighbors is used to balance the two terms in the above error bound.
This result appears to be new to the literature, to the best of our knowledge. 

Also note that our setup in this paper allows the latent functions $\eta_l$'s to be heterogeneous across $l$ in general. When $\eta_l(\balpha_i)$ admits a special structure such as  $\eta_l(\balpha_i)=\eta(\balpha_i,\bm\varpi_l)$ for some bivariate function $\eta$ and some random variables $\bm\varpi_l$, the optimal convergence rate for estimating the global function $\eta$ may be different than the optimal rate for the cross-sectional nonparametric regression described above, depending on the smoothness of $\eta$ and the dimensions of the arguments $\balpha_i$ \textit{and} $\bm\varpi_l$. See \cite{Gao-Lu-Zhou_2015_AoS} for a discussion of the $L^2$-type minimax optimal rate in the special case of graphon estimation.

\begin{remark}[Comparison with global approximation methods]
	It is well known in the literature that the nonlinear factor structure \eqref{eq: model} can also be globally approximated 
	using the idea of singular value decomposition of functions. Suppose that $\eta_l(\balpha_i)=\eta(\balpha_i, \bm\varpi_l)$ for some bivariate function $\eta$ and random variables $\bm\varpi_l$. It can be shown that, under some regularity conditions, $\eta(\balpha, \bm\varpi)=\sum_{\ell=1}^{\infty}s_\ell u_\ell(\balpha)v_\ell(\bm\varpi)\approx \sum_{\ell=1}^{R}s_\ell u_\ell(\balpha)v_\ell(\bm\varpi)$ for some singular values $s_\ell$ and eigenfunctions $u_\ell$ and $v_\ell$, if the number of terms $R\rightarrow\infty$  \citep[see, e.g.,][]{griebel2014approximation,griebel2019singular}. This motivates low-rank approximation of nonlinear factor models. For instance, one can apply PCA to the whole matrix $\bX$ directly with the number of principal components growing large as $n,p\rightarrow\infty$, or design other methods based on similar ideas \citep[see, e.g.,][]{fernandez2021low}. By contrast, the proposed method in this paper relies on a local approximation of the nonlinear factor structure, where the number of factors in each local neighborhood can be fixed while the approximation bias is determined by the number of nearest neighbors $K$. As discussed before, the proposed estimator can achieve the optimal uniform convergence rate for nonparametric regression, while it is unknown if methods based on the global approximation strategy described above can achieve the same rate.
\end{remark}

\begin{remark}[Selecting the number of local factors]\label{remark: number of factors}
	In this nonlinear factor model, $\ttd_i$ is the user-specified number of ``factors'' extracted in each local neighborhood $\subi$, which plays a role similar to the degree of the polynomial in local polynomial regression.
	One can investigate the strength of the (local) eigenvalues and extract all factors (approximation terms in \eqref{eq: local representation}) that are asymptotically stronger than the idiosyncratic errors. 
	This ensures that the smoothing bias is no greater than the variance asymptotically.
\end{remark}

\begin{remark}[Determining the number of latent variables]\label{remark: number of latent confounders}
	In this nonlinear factor model, the true number of latent variables $r$ is also the dimension of local tangent spaces of the underlying manifold (see Figure \ref{figure:local tangent}). This implies that $r$ can be determined by examining the number of linear terms in the local approximation of the latent functions $(\eta_l: l\in\mathcal{C}^\ddagger)$. To fix ideas, consider the first-order Taylor expansion of $\eta_l(\cdot)$ at the $i$-th unit:
	\[
	x_{jl}=\eta_l(\balpha_j)+u_{jl}
	=\eta_l(\balpha_i)+\nabla\eta_l(\balpha_i)'(\balpha_j-\balpha_i)+\xi_{jl}+u_{jl},\quad j\in\mathcal{N}_i,
	\]
	where $\xi_{jl}$ is the approximation error.
	Typically, if the magnitude of the noise is relatively small, the leading factor associated with the largest eigenvalue in local PCA at $\balpha_i$ corresponds to the ``local constant term'' $\eta_l(\balpha_i)$ (monomial basis of degree zero). The next few factors are associated with much smaller eigenvalues than the first one and correspond to the ``local linear terms'' $\nabla\eta_l(\balpha_i)'(\balpha_j-\balpha_i)$, but they are still stronger than the remainder asymptotically. The number of such linear terms is also the true number of latent variables $\balpha_i$. Using this fact, we can design a feasible procedure to determine $r$ in practice.
	For instance, we can start with a relatively large $K$, investigate the differing strength of local factors, and in particular check the number of local factors associated with eigenvalues of the second largest magnitude, say $r_i$. As discussed before, the nonlinearity pattern of the latent space may be complex and varies across the evaluation points. Thus, one may want to repeat this procedure for different units and take $r=\max_{1\leq i\leq n}r_{i}$. A formal study of this procedure is left for future research. 
\end{remark}

So far we have focused on the case where the data matrix $\bX$ is complete, but in many applications $\bX$ may have missing entries. The proposed local PCA method is robust with respect to mild missing value issues and thus can be applied to some policy evaluation problems such as synthetic controls. Detailed discussion is deferred to Section \ref{sec: application}.

\subsection{Covariates Adjustment} \label{subsec: extensions}

The analysis so far has focused on Equation \eqref{eq: HD covariate}, which assumes $\bx_i$ takes a purely nonlinear factor structure. However, high-rank components may exist in $\bx_i$, and it is the residuals that take a possibly nonlinear factor structure, as described by Equation \eqref{eq: HD covariate, high-rank} below:
\begin{equation}\label{eq: HD covariate, high-rank}
	\bx_{i}=
	\bW_i\bvth+\bmeta(\balpha_i)+\bu_{i}, \quad \E[\bu_{i}|(\balpha_j,\bW_j:j\in[n])]=\bm{0},
\end{equation}
where $\bW_i=(\bw_{i,1}, \cdots, \bw_{i,q})\in\mathbb{R}^{p\times q}$ is a matrix of observed covariates, and $\bvth\in\mathbb{R}^q$ is an unknown coefficient vector common to all units and features.
It can be viewed as a linear regression of $\bx_i$ on $q$ regressors $\bw_{i,1}, \cdots, \bw_{i,q}$ with possibly nonlinear fixed effects $\bmeta(\balpha_i)$. 
In this context $\bW_i$ needs to have sufficiently high-rank variation, otherwise it will be too collinear with the unspecified low-rank component $\bmeta(\balpha_i)$ and $\bvth$ cannot be identified. This is similar to the identification condition for panel data regression with interactive fixed effects.

The main analysis of this paper can be applied once a consistent estimator of $\bvth$ is available. It can be obtained using the idea of partially linear regression. Specifically, we propose the following procedure: 
\begin{enumerate}[label=(\alph*)]\setlength\itemsep{.01em}
	\item Split the feature index set into three (non-overlapping) portions: $[p]=\mathcal{C}^\dagger\cup
	\mathcal{C}^\ddagger\cup\mathcal{C}^\wr$.
	
	\item On $\mathcal{C}^\dagger\cup\mathcal{C}^\ddagger$, for each $\ell\in[q]$,
	apply Algorithm \hyperref[algorithm]{1} in Section \ref{sec: estimation} to $(\bw_{i,\ell}: i\in[n])$. Obtain the residual vectors on $\mathcal{C}^\ddagger$, $\widehat{\be}_{i,\ell}:=\bw_{i,\ell}^\ddagger-\widehat{\bw}_{i,\ell}^\ddagger$. Here $\widehat{\bw}_{i,\ell}^\ddagger$ is the transposed row corresponding to unit $i$ of the local PCA reconstruction $\widehat{\bLambda}_{\subi}\widehat{\bF}_{\subi}'$ for this covariate, using the neighborhood centered at $i$. Use $\mathcal{C}^\dagger$ for $K$-NN matching and $\mathcal{C}^\ddagger$ for (local) PCA.
	
	\item  On $\mathcal{C}^\dagger\cup\mathcal{C}^\ddagger$, apply Algorithm 1 to $(\bx_{i}: i\in[n])$. Let the residual vector on $\mathcal{C}^\ddagger$ be $\widehat{\bu}^{\natural}_{i}=\bx_{i}^\ddagger-\widehat{\bx}_{i}^\ddagger$. Define $\widehat{\bx}_{i}^\ddagger$ analogously as the transposed reconstructed row for unit $i$ in its own neighborhood.
	
	\item Let $\widehat{\be}_{i}=(\widehat{\be}_{i,1},\cdots, \widehat{\be}_{i,q})'$. Estimate $\bvth$ by
	$
	\widehat{\bvth}=\left(\frac{1}{np^\ddagger}\sum_{i=1}^{n}\widehat{\be}_{i}\widehat{\be}_{i}'\right)^{-1}\left(\frac{1}{np^\ddagger}\sum_{i=1}^{n}\widehat{\be}_{i}\widehat{\bu}^{\natural}_{i}\right)$.
	
	\item On $\mathcal{C}^\ddagger\cup\mathcal{C}^\wr$, apply Algorithm 1 to $(\bx_{i}-\bW_{i}\widehat{\bvth}: i\in[n])$. Use $\mathcal{C}^\ddagger$ for $K$-NN matching and $\mathcal{C}^\wr$ for (local) PCA. The final output of interest is the index sets for nearest neighbors $\mathcal{N}_i$, local factors $\widehat{\bF}_{\subi}$ and loadings $\widehat{\bLambda}_{\subi}$ from this step.
\end{enumerate}
Under additional regularity conditions on $\bW_i$, it can be shown that $\widehat{\bvth}$ converges to $\bvth$ sufficiently fast and the main results established previously still hold for $\mathcal{N}_i$, $\widehat{\bF}_{\subi}$ and $\widehat{\bLambda}_{\subi}$. 
Formal analysis is available in Section \ref{SA-sec: main results} of the SA and is omitted here to conserve space.

\section{Simulations} \label{sec: simulations}

We conduct a Monte Carlo investigation of the finite sample performance of the proposed method. Write $\eta_{il}:=\eta_l(\alpha_i)$ for the noise-free mean of entry $(i,l)$. We consider two nonlinear factor models for continuous data and one for discrete data:
\begin{itemize}
	\item Model 1: $\eta_{il}=\frac{1}{0.1\sqrt{2\pi}}\exp(-100(\alpha_i-\varpi_l)^2)$, 
	$x_{il}\sim \mathsf{N}(\eta_{il}, \,0.5^2)$.
	\item Model 2: $\eta_{il}=
	\exp(-10|\alpha_i-\varpi_l|)$,  
	$x_{il}\sim \mathsf{N}(\eta_{il},\, 0.5^2)$.
	\item Model 3: $\eta_{il}=
	1-(1+\exp(15(0.8|\alpha_i-\varpi_l|)^{0.8}-0.1))^{-1}$,
	$x_{il}\sim\mathsf{Bernoulli}(\eta_{il})$.
\end{itemize}
The latent variables $\alpha_i$ and $\varpi_l$ follow the uniform distributions on $[0, 1]$.

We consider $2\,000$ simulations with $n=p=1\,000$.
To implement the proposed method, we use the first half of columns of $\bX$ for $K$-NN matching and the second half for principal component analysis. We take the pseudo-max distance and set the number of nearest neighbors $K=\mathsf{c}\times n^{2/3}$ for $\mathsf{c}=0.5, 1$ and $1.5$. 
To avoid extracting too weak local factors,  we let the number of local principal components 
$\ttd_i=2$ if $\widehat\upsilon_{2, \subi}/\widehat\upsilon_{3,\subi}\geq \log\log K$ and $\ttd_i=1$ otherwise, for each $i\in[n]$. Here $\widehat\upsilon_{j,\subi}$ is the nonnegative square root of the $j$th largest eigenvalue of $\bX_{\subi}'\bX_{\subi}/(p^\ddagger K)$.
We also compare the proposed method (LPCA) with a natural alternative---global principal component analysis (GPCA).
Specifically, we use the entire matrix $\bX$ to extract principal components and form the mean prediction for each entry. 
The number of factors is determined by using the eigenvalue ratio test \citep{Ahn_2013_ECMA} based on doubly demeaned data. 
Let $\widehat{\eta}_{il}$ denote the fitted mean of entry $(i,l)$. For LPCA, it is the entry for unit $i$ and feature $l$ of $\widehat{\bLambda}_{\subi}\widehat{\bF}_{\subi}'$ from the neighborhood centered at $i$; for GPCA, it is the corresponding entry of the global PCA reconstruction.
We report the following results for each method: (1) maximum absolute error (MAE), i.e., 
$\max_{i\in[n], l\in\mathcal{C}^\ddagger}|\widehat{\eta}_{il}-\eta_{il}|$;
and (2) the prediction error for  three missing entries in the last column of $\bX$. For (2), in each simulated dataset we take the three units with the values of $\alpha_i$ equal to the $0.1$-, $0.5$-, and $0.9$-quantiles of the sample respectively, and then 
replace their corresponding values in the last column of $\bX$ with zeros (``missing values'').
This operation mimics the data pattern in some policy evaluation settings: a unit whose individual feature is at the low, medium or high level relative to the whole distribution gets treated in the last period and thus has a missing value in the last column of $\bX$. 
Finally, all these measurements of performance are averaged across $2\,000$ repetitions.

\FloatBarrier
\begin{table}[h]
	\centering
	\scriptsize
	\caption{Simulation Results, $n=1\,000$, $p=1\,000$, $2\,000$ replications}
	\label{table:simul}
	\resizebox{.55\textwidth}{1.1\height}{\footnotesize 
\begin{tabular}{lrrrcr}
\hline\hline
\multicolumn{1}{l}{\bfseries }&\multicolumn{3}{c}{\bfseries LPCA}&\multicolumn{1}{c}{\bfseries }&\multicolumn{1}{c}{\bfseries GPCA}\tabularnewline
\cline{2-4} \cline{6-6}
\multicolumn{1}{l}{}&\multicolumn{1}{c}{$K$=49}&\multicolumn{1}{c}{$K$=99}&\multicolumn{1}{c}{$K$=149}&\multicolumn{1}{c}{}&\multicolumn{1}{c}{}\tabularnewline
\hline
{\bfseries Model 1}&&&&&\tabularnewline
~~MAE&$0.680$&$0.937$&$2.203$&&$1.148$\tabularnewline
~~$q_\alpha$=.1&$0.075$&$0.068$&$0.091$&&$0.111$\tabularnewline
~~$q_\alpha$=.5&$0.077$&$0.078$&$0.119$&&$0.095$\tabularnewline
~~$q_\alpha$=.9&$0.076$&$0.068$&$0.088$&&$0.112$\tabularnewline
\hline
{\bfseries Model 2}&&&&&\tabularnewline
~~MAE&$0.599$&$0.636$&$0.706$&&$0.870$\tabularnewline
~~$q_\alpha$=.1&$0.067$&$0.055$&$0.051$&&$0.111$\tabularnewline
~~$q_\alpha$=.5&$0.066$&$0.054$&$0.052$&&$0.097$\tabularnewline
~~$q_\alpha$=.9&$0.071$&$0.054$&$0.051$&&$0.112$\tabularnewline
\hline
{\bfseries Model 3}&&&&&\tabularnewline
~~MAE&$0.475$&$0.461$&$0.461$&&$0.470$\tabularnewline
~~$q_\alpha$=.1&$0.035$&$0.032$&$0.039$&&$0.047$\tabularnewline
~~$q_\alpha$=.5&$0.038$&$0.038$&$0.043$&&$0.047$\tabularnewline
~~$q_\alpha$=.9&$0.034$&$0.032$&$0.039$&&$0.046$\tabularnewline
\hline
\end{tabular}
}
	\flushleft\footnotesize{\textbf{Notes}:
		MAE = maximum absolute error; $q_\alpha$:  quantile of $(\alpha_i: 1\leq i\leq n)$.
	}\newline
\end{table}
\FloatBarrier
\vspace{-1em}

The proposed method performs well  across the three settings. Compared with GPCA, LPCA usually has smaller or at least comparable maximum absolute errors and prediction errors for the three missing entries. 
It can be seen that such improvement is more pronounced when the nonlinearity of $\eta_l$ is more severe (e.g., Model 1) and a sufficiently small $K$ is used.
The results are relatively stable across different choices of $K$ except that in Model 1---a highly nonlinear one---using a very large $K$ leads to a substantial increase in MAE. In the SA we report additional simulation evidence for datasets with smaller $p$ and for LPCA with a larger fraction of columns for $K$-NN matching (and thus a smaller fraction for obtaining the principal components). The results are largely close to the main findings above.

\section{Application: Synthetic Controls}\label{sec: application}

The proposed method is readily applicable to certain matrix completion problems. For instance, in the classical synthetic control design \citep{Abadie_2020_JEL}, the researcher is interested in the effect of a policy that only affects one single unit, and all other units remain untreated throughout the observation period. Formally, we can define the potential outcome under the treatment $y_{il}(1)$ and that in the absence of the treatment $y_{il}(0)$  for each unit $i\in[n]$ and time $l\in[p]$. The observed outcome $y_{il}=d_{il}y_{il}(1)+(1-d_{il})y_{il}(0)$ where 
$d_{il}=\I(i=1,l>p_0)$ for some $p_0<p$. That is, a policy intervention was implemented at time $p_0+1$ and affected the first unit only. The key task in this problem is to predict the missing counterfactual outcome $y_{1l}(0)$ for the treated unit in each post-treatment period $l> p_0$. Usually, the $n\times p$ matrix of the potential outcomes $\bY^{\mathsf{N}}$, whose typical entry is given by $y_{il}(0)$, is assumed to admit a linear factor structure, which justifies the key idea of synthetic controls that uses a linear combination of untreated units to predict the counterfactual of the treated. 

The proposed local PCA method relaxes the linearity requirement by allowing for a possibly nonlinear factor structure.  If the previous procedure could be applied to the matrix $\bY^{\mathsf{N}}$, we would immediately have an estimate of the mean of the counterfactual outcome of the treated in the post-treatment period, and Corollary \ref{coro: consistency of latent mean} would apply. However, several entries at the top-right corner of $\bY^{\mathsf{N}}$ are unobserved. 
Let 
$x_{il}=y_{il}\I(i\neq 1 \text{ or } l\leq p_0)$.
Thus, the matrix $\bX$ is identical to $\bY^{\mathsf{N}}$  except that the missing outcomes of the treated in $\bY^{\mathsf{N}}$ are replaced with zeros. Then we apply the proposed procedure to the observed matrix $\bX$. 
Assume the number of post-treatment periods $p-p_0$ is fixed, which is common in classical synthetic control problems. Thus $\bX$ can be viewed as a slightly perturbed version of $\bY^{\mathsf{N}}$. 
The following theorem shows that the proposed mean estimator $\widehat{\bLambda}_{\subi}\widehat{\bF}_{\subi}'$ is robust to such ``small'' perturbations.

\begin{thm}\label{thm: robustness}
	Suppose that Assumptions \ref{Assumption: Regularities}, \ref{Assumption: Matching} and \ref{Assumption: LPCA} hold for the counterfactual outcome matrix $\bY^{\mathsf{N}}$. Let $\widehat{\bF}_{\subi}$ and $\widehat{\bLambda}_{\subi}$ be outputs obtained by applying Algorithm \hyperref[algorithm]{1} to $\bX$ defined above. Assume $\mathcal{C}^\dagger\subseteq[p_0]$, $p-p_0$ is fixed, and $(np)^{\frac{2}{\nu}}\delta_{Kp}^{-2}\lesssim 1$. Then,
	\[
	\max_{i\in[n]}\;\frac{\Big\|\widehat{\bLambda}_{\subi}\widehat{\bF}_{\subi}'-\bH_{\subi}\Big\|_{\max}}{\delta_{Kp}^{-1}+\mathfrak r_{\eta,\subi}}\lesssim_\P 1.
	\]
\end{thm}
This result relies on the robustness of leading singular vectors with respect to small perturbation of $\bY^{\mathsf{N}}$. Similar results still hold for missing data patterns other than synthetic controls. The key requirement is that there are not many missing entries in each row and in each column. See the proof of Theorem \ref{thm: robustness} in the SA for more details. When missingness leads to relatively large perturbation of $\bY^{\mathsf{N}}$, local PCA is not robust and thus may not be directly applied. \cite{Feng_2024_wp} discusses the scenario where there are many missing entries in one or a few columns (or rows) and shows that the extracted local factors and loadings obtained from the complete part of the matrix can be used for prediction of missing entries.

\subsection{An Empirical Example}\label{subsec: example}

To illustrate the proposed method, we reanalyze the effect of 2012 Kansas tax cuts on economic growth (see \citealp{Rickman-Wang_2018_RSUE} for more details). 
The second quarter of 2012 (2012Q2)---when the governor Brownback signed the tax cut bill into law---is used as the starting time of the policy. As described above, this can be viewed as a synthetic control problem where the goal is to predict the counterfactual economic performance of Kansas after 2012Q2 had the tax cut policy not been implemented. 
We have data on quarterly gross state product (GSP) per capita growth rates for 50 states ($n=50$) from 1990Q2 to 2016Q1 ($p=104$).
The growth rates of Kansas after 2012Q2 are replaced with  missing values. 
The same strategy as described in Section \ref{sec: simulations} is used for local PCA: we take the pseudo-max distance, set $K=n^{2/3}\approx 14$ and choose the number of local factors based on the strengths of leading singular values. The first $40$ periods are used for $K$-NN matching, and the remaining periods are used for PCA. For comparison, we also implement the classical synthetic control (SC) method that imposes a simplex constraint on weights.

The result is shown in Figure \ref{figure:kansas, growth} below.
It turns out that in 9 out of 16 post-treatment periods, the observed Kansas grows slower than predicted by LPCA. Note that LPCA estimates the mean of the GSP growth, so the difference between the two growth paths may arise from both estimation error and an idiosyncratic noise in each period. To reduce the impact of the idiosyncratic noise, we also compute  the average growth rate over the entire post-treatment period for the counterfactual Kansas, which  is $0.53$ percentage points higher than that of the observed Kansas. 
By contrast, SC yields a smaller effect, with an average growth rate $0.33$ percentage points higher  than that of the observed Kansas. 
Also, note that the SC results may be more vulnerable to large temporary shocks in a few pre-treatment periods, 
as the weights used for SC prediction are constructed to make the pre-treatment fit as close as possible to the observed sequence.

To see the impact of the tax cut policy on the level of GSP per capita, we take the log GSP per capita of Kansas in 2012Q1 as the initial value and translate the predicted growth rates into a counterfactual log GSP per capita trajectory for Kansas. 
In this case, we compare our LPCA results with SC using the data of log GSP per capita levels rather than growth rates, which is common in the SC literature.
The results are given in Figure \ref{figure:kansas, level}. Both methods show that overall the tax cut policy has a pronounced negative shock on the GSP growth trajectory of Kansas, but the effect given by SC is relatively small in magnitude, especially in later post-treatment periods.

\begin{figure}[!h]
	\small
	\begin{center}\caption{Effect of 2012 Kansas Tax Cut on Economic Growth}
		\begin{subfigure}{0.47\textwidth}
			\includegraphics[width=\textwidth, height=0.24\textheight]{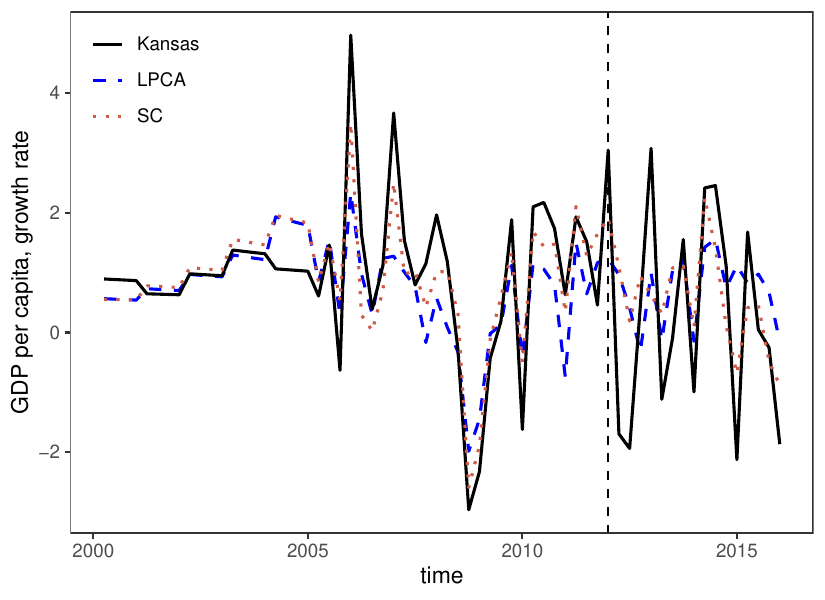}
			\caption{GSP per capita growth rate}
			\label{figure:kansas, growth}
		\end{subfigure}
		\begin{subfigure}{0.47\textwidth}
			\includegraphics[width=\textwidth, height=0.24\textheight]{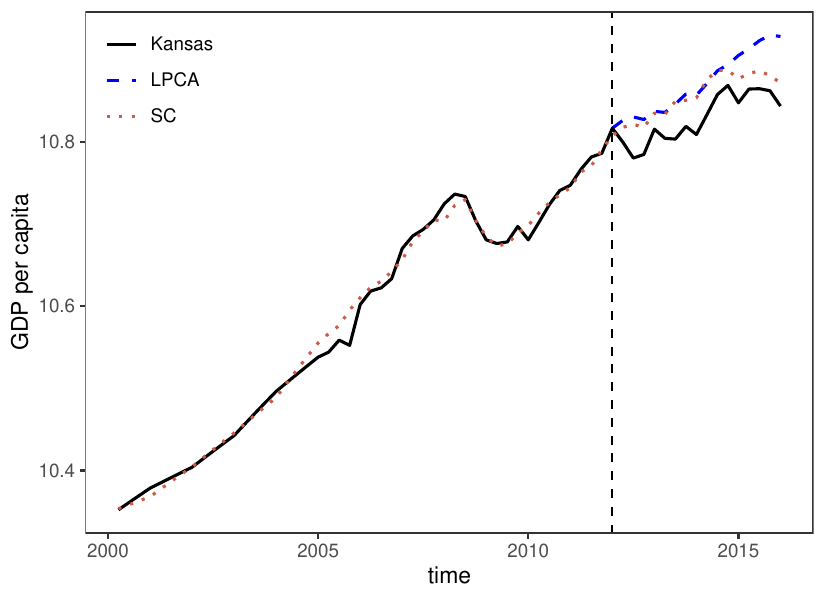}
			\caption{log GSP per capita}
			\label{figure:kansas, level}
		\end{subfigure}
	\end{center}
\vspace{-2em}
	\flushleft\footnotesize{\textbf{Notes}:
	LPCA: local principal component analysis; SC: synthetic control. The vertical dashed line stands for the start of the policy.
}
\end{figure}

\section{Conclusion} \label{sec: conclusion}
This paper studied optimal estimation of large-dimensional nonlinear factor models. The observed variables were modeled as possibly nonlinear functions of some latent variables with the functional forms unspecified. A local principal component analysis method that combines $K$ nearest neighbors matching and principal component analysis was proposed to estimate the nonlinear factor structure and recover the information on latent variables and latent functions. The large-sample properties of the proposed estimators were established, including a sharp bound on the matching discrepancy of nearest neighbors under generic distance functions, the sup-norm error bounds for estimated local factors and loadings, and the uniform convergence rate for the estimator of the factor structure. The method can also be applied to some matrix completion problems such as synthetic controls without requiring the mean matrix to be exactly low-rank.

\begin{appendices}\label{sec: appendix}

\section{Verification of High-Level Conditions}\label{sec: appendix matching}

We imposed two high-level conditions, Assumptions \ref{Assumption: Matching} and \ref{Assumption: LPCA}, for $K$-NN matching and local PCA.
In this section we discuss how they can be verified for specific cases.

Throughout this appendix, we impose the following primitive smoothness condition.
\begin{assumption}[Primitive Smoothness]\label{Assumption: Primitive Smoothness}
	Each $\eta_l(\cdot)$, $1\leq l\leq p$, is $\bar m$-times continuously differentiable for some fixed integer $\bar m\geq2$, with all partial derivatives of order no greater than $\bar m$ bounded by a universal constant.
\end{assumption}

We first verify Assumption \ref{Assumption: Matching} for each choice of distance given in Examples \ref{exmp: euclidean}--\ref{exmp: dist of average}. 
The following theorem gives a set of sufficient conditions for the squared Euclidean distance.
\begin{thm}[Euclidean distance] \label{thm: IMD hom}
	Let $\rho(\cdot, \cdot)$ be the squared Euclidean distance in Example \ref{exmp: euclidean}. Suppose that Assumptions \ref{Assumption: Regularities} and \ref{Assumption: Primitive Smoothness} hold, together with the following conditions:
	\begin{enumerate}[label=(\roman*)]
		\item 	For some fixed $c_u\geq 0$, $\frac{1}{p^\dagger}\sum_{l\in\mathcal{C}^\dagger}\E[u_{il}^2|\mathscr{F}]=c_u$ for all $i\in [n]$;
		\item  The conditional $(4+2\nu)$-th moments of $u_{il}$ are bounded almost surely by a fixed constant, uniformly over $i\in[n]$ and $l\in\mathcal{C}^\dagger$;
		\item 	For every $\varepsilon>0$,
		$\underset{\Delta\rightarrow 0}{\lim}\;\underset{n,p^\dagger\rightarrow\infty}{\limsup}\;
		\P\Big\{\underset{i\in [n]}{\max}\;
		\underset{j:\rho(\bH^\dagger_{i\cdot},\bH^\dagger_{j\cdot})<\Delta}{\max}\;
		\|\balpha_i-\balpha_j\|>\varepsilon\Big\}=0$;
		\item 	For some $\underline{c}>0$,
		$\underset{n,p^\dagger\rightarrow\infty}{\lim}\;\P\Big\{
		\underset{i\in [n]}{\min}\;s_{\min}\Big(\frac{1}{p^\dagger}
		\sum_{l\in\mathcal{C}^\dagger}\nabla\eta_l(\balpha_i)\nabla\eta_l(\balpha_i)'\Big)\geq\underline{c}\Big\}=1$.
	\end{enumerate}
	If $\frac{n^{\frac{2}{\nu}}\log^{1+\frac{2}{\nu}} (n\vee p)}{p}\lesssim 1$, then
	Assumption \ref{Assumption: Matching} holds with
	$a_p=\sqrt{(\log n)/p}$, $\rho_0=2c_u$, $\lvarsig=\uvarsig=2$.
\end{thm}
\begin{remark}[Plausibility of (iii)] \label{remark: nonsingularity}
	Condition (iii) is a requirement for informativeness of measurements. To get a sense of its plausibility, consider a linear factor model with an intercept: $\eta_l(\alpha_i)=c_0+\varpi_l\alpha_i$ for $c_0$, $\varpi_l\in\mathbb{R}$. When $\varpi_l\neq 0$ if and only if $l=1$ (the factor is too ``sparse''), most measurements in $\bx_i$ are uninformative about $\alpha_i$. For $\alpha_i\neq\alpha_j$, their difference is revealed only by the first measurement. The distance between $\bH^{\dagger}_{i\cdot}$ and $\bH^{\dagger}_{j\cdot}$ becomes negligible as $p$ diverges, which violates condition (iii). However, as long as a non-negligible \textit{fraction} of the coefficients $(\varpi_l: l\in\mathcal{C}^\dagger)$ are bounded away from zero, the corresponding measurements suffice to differentiate the two units. Thus, the proposed method only requires some but not all measurements to be informative, and importantly, it is unnecessary to know their identities a priori.
\end{remark}

Next, we consider the pseudo-max distance in Example \ref{exmp: pseudo-max}.

\begin{thm}[Pseudo-max distance] \label{thm: IMD hsk}
	Let $\rho(\cdot, \cdot)$ be the pseudo-max distance in Example \ref{exmp: pseudo-max}. 
	Suppose that Assumptions \ref{Assumption: Regularities} and \ref{Assumption: Primitive Smoothness} hold, together with the following two conditions:
	\begin{enumerate}[label=(\roman*)]
		\item 	For every $\varepsilon>0$,
		$\underset{\Delta\rightarrow 0}{\lim}\;
		\underset{n,p^\dagger\rightarrow\infty}{\limsup}\;
		\P\Big\{\underset{i\in [n]}{\max}\;
		\underset{j:\rho(\bH^\dagger_{i\cdot},\bH^\dagger_{j\cdot})<\Delta}{\max}\;
		\|\balpha_i-\balpha_j\|>\varepsilon\Big\}=0$;
		\item 	For some $\underline{c}>0$,
		$\underset{n,p^\dagger\rightarrow\infty}{\lim}\;\P\Big\{
		\underset{i\in [n]}{\min}\;s_{\min}\Big(\frac{1}{p^\dagger}
		\sum_{l\in\mathcal{C}^\dagger}\nabla\eta_l(\balpha_i)\nabla\eta_l(\balpha_i)'\Big)\geq\underline{c}\Big\}=1$;
	\end{enumerate} 
	If $\frac{n^{\frac{4}{\nu}}\log n}{p}\lesssim 1$, then Assumption \ref{Assumption: Matching} holds with
	$a_p=\sqrt{(\log n)/p}$, $\rho_0=0$, $\lvarsig=\uvarsig=1$.
\end{thm}

\begin{thm}[Distance of averages] \label{thm: dist of average}
	Let $\rho(\cdot, \cdot)$ be the distance of averages given in Example \ref{exmp: dist of average}. 
	Suppose $r=1$. Also suppose that Assumptions \ref{Assumption: Regularities} and \ref{Assumption: Primitive Smoothness} hold, together with the following conditions:
	\begin{enumerate}[label=(\roman*)]
		\item $\underset{\balpha\in\mathcal{A}}{\sup}\;\Big|\frac{1}{p^\dagger}\sum_{l\in\mathcal{C}^\dagger}\eta_l(\balpha)-\bar{\eta}(\balpha)\Big|=o(1)$ for some fixed function $\bar{\eta}(\cdot)$;
		\item $\bar{\eta}(\cdot)$ is continuously differentiable and $\bar{\eta}^{(1)}(\alpha)\neq0$ for all $\alpha\in\mathcal{A}$.
	\end{enumerate}
	If $\frac{n^{\frac{2}{\nu}}\log^{1+\frac{2}{\nu}} (n\vee p)}{p}\lesssim 1$, then
	Assumption \ref{Assumption: Matching} holds with
	$a_p=\sqrt{(\log n)/p}$, $\rho_0=0$, $\lvarsig=\uvarsig=1$.
\end{thm}

Next, we verify Assumption \ref{Assumption: LPCA} using the auxiliary Taylor representation described in Section \ref{subsec: main result, pca}. Let $h_n=(K/n)^{1/r}$. In that representation, the normalized monomial basis has scaling matrix $\widetilde\bUpsilon_{\subi}=\diag(1,h_n\bI_{\ell_1},\ldots,h_n^{m-1}\bI_{\ell_{m-1}})$, where $\ell_j=\binom{r+j-1}{j}$ is the number of $r$-variate monomials of degree $j$. The theorem concerns the factors and loadings defined by the noise-free PCA problem \eqref{eq: lpca noise-free}, not the auxiliary monomial basis itself.

\begin{thm}[Verification of Assumption \ref{Assumption: LPCA}]\label{thm: verify LPCA}
	Suppose that Assumptions \ref{Assumption: Regularities}, \ref{Assumption: Matching}, and \ref{Assumption: Primitive Smoothness} hold, together with the following conditions:
	\begin{enumerate}[label=(\roman*)]
		\item For some constant $c>0$,
		$s_{\min}(\frac{1}{p^\ddagger}\sum_{l\in\mathcal{C}^\ddagger}
		(\mathscr{D}^{m-1}\eta_{l}(\balpha_i))(\mathscr{D}^{m-1}\eta_{l}(\balpha_i))')\geq c$, for all $i\in[n]$, w.p.a.1;
		\item $m$ is a fixed integer satisfying $1\leq m\leq\bar m$, $K\log n/n=o(1)$, $\delta_{Kp}^{-1}/h_{n}^{m-1}=o(1)$, $a_p^{1/\underline{\varsigma}}=o(h_n)$, and $\bar{\varsigma}=\underline{\varsigma}$.
	\end{enumerate}
	Then, Assumption \ref{Assumption: LPCA} is satisfied by setting $\ttd_i=\binom{m-1+r}{r}$, and $\max_{i\in[n]}\mathfrak r_{\eta,\subi}\lesssim_\P h_n^m$.
\end{thm}

In this theorem we require, for simplicity, that all derivatives of the latent functions not be too collinear, but recall that Assumption \ref{Assumption: LPCA} is general enough to cover other cases with more heterogeneous degrees of nonlinearity across different local neighborhoods.

\end{appendices}

\putbib[references]

\end{bibunit}

\clearpage
\newgeometry{top=1in,bottom=1in,left=1in,right=1in}
\restoretitlecommands
\onehalfspacing
\setcounter{page}{1}
\setcounter{section}{0}
\setcounter{subsection}{0}
\setcounter{subsubsection}{0}
\setcounter{equation}{0}
\setcounter{thm}{0}
\setcounter{lem}{0}
\setcounter{coro}{0}
\setcounter{assumption}{0}
\setcounter{definition}{0}
\setcounter{exmp}{0}
\setcounter{remark_tmp}{0}
\setcounter{figure}{0}
\setcounter{table}{0}
\setcounter{footnote}{0}
\renewcommand{\thesection}{SA-\arabic{section}}
\renewcommand{\thefigure}{SA-\arabic{figure}}
\renewcommand{\thetable}{SA-\arabic{table}}
\renewcommand{\thedefinition}{SA-\arabic{definition}}
\numberwithin{equation}{section}
\numberwithin{thm}{section}
\numberwithin{lem}{section}
\numberwithin{assumption}{section}
\numberwithin{coro}{thm}

\makeatletter
\let\arxivtableofcontents\tableofcontents
\renewcommand{\tableofcontents}{%
  \let\arxivstarttoc\@starttoc
  \def\@starttoc##1{\arxivstarttoc{satoc}}%
  \arxivtableofcontents
  \let\@starttoc\arxivstarttoc}
\let\arxivaddcontentsline\addcontentsline
\renewcommand{\addcontentsline}[3]{%
  \def\arxivcontentskind{#1}\def\arxivtockind{toc}%
  \ifx\arxivcontentskind\arxivtockind
    \arxivaddcontentsline{satoc}{#2}{#3}%
  \else
    \arxivaddcontentsline{#1}{#2}{#3}%
  \fi}
\makeatother
\begin{bibunit}[hapalike]
\bibliographyprefix{supplement}
\if1\blind
{
\title{\vspace{-0.25in}\Large Optimal Estimation of Large-Dimensional Nonlinear Factor Models\\
	\medskip
	Supplemental Appendix
	\bigskip }
\author{
	Yingjie Feng\thanks{School of Economics and Management, Tsinghua University.}}
\maketitle

} \fi

\if0\blind
{
	\bigskip
	\bigskip
	\bigskip
	\begin{center}
		{\Large Optimal Estimation of Large-Dimensional Nonlinear Factor Models\\
			\medskip
			Supplemental Appendix}
	\end{center}
	\medskip
} \fi

\vspace{1em}

\begin{abstract}
This supplement gives omitted theoretical proofs of the results discussed in the main paper, and additional technical and numerical results, which may be of independent interest. Section \ref{SA-sec: discussion of conditions} gives further discussions of assumptions imposed in the main paper. Section \ref{SA-sec: notation} collects notation used in this supplement. Section \ref{SA-sec: proofs} gives the omitted proofs for the main paper. Section \ref{SA-sec: main results} presents the formal theory and proofs for the covariates-adjusted nonlinear factor model. Section \ref{SA-sec: simulation} gives additional simulation results.
\end{abstract}

\thispagestyle{empty}

\newpage\tableofcontents

\thispagestyle{empty}
\clearpage

\setcounter{page}{1}
\pagestyle{plain}

\section{Further Discussions of Assumptions}
\label{SA-sec: discussion of conditions}

\subsubsection*{Informativeness}
As discussed in the main paper, Assumption \ref{Assumption: Matching}(b) is key for validity of $K$ nearest neighbors matching.
Appendix \ref{sec: appendix matching} shows how it can be verified based on ``informativeness'' conditions that are specific to each distance function. We give some further discussion in this section.

We will first focus on condition (iii) in Theorem \ref{thm: IMD hom}, the key informativeness condition for Euclidean distance.
Throughout the paper, the latent functions $(\eta_l:1\leq l\leq p)$ are treated as fixed (conditioned upon). Thus, the probability in this condition is over the sampled latent variables. If the maxima over the evaluation points are replaced by suprema over the whole support, the resulting condition is a deterministic restriction on the latent functions.
For a concrete interpretation, consider the specification
$\eta_l(\balpha_i):=\eta(\bm{\varpi}_l;\balpha_i)$,
where $\eta(\cdot;\cdot)$ is a fixed function and $(\bm{\varpi}_l:1\leq l\leq p)$ is a fixed sequence in $\mathcal{Z}$ with limiting distribution $F_{\varpi}(\cdot)$.
To simplify the discussion, attention will be mostly restricted to this specific case in the following discussion.

Recall that condition (iii) in Theorem \ref{thm: IMD hom} says that
\begin{equation}\label{eq: def of nonsingularity}
	\lim_{\Delta\rightarrow 0}\limsup_{n,p^\dagger\rightarrow\infty}\;
	\P\Big\{\max_{i\in[n]}\;
	\max_{j:\rho(\bH_{i\cdot}^\dagger,\bH_{j\cdot}^\dagger)<\Delta}\;
	\|\balpha_i-\balpha_j\|>\varepsilon\Big\}=0.
\end{equation}

\eqref{eq: def of nonsingularity} can be viewed as the continuity of the inverse map of $\bH_{i\cdot}^\dagger$.
To see this, suppose that $\eta_l(\balpha_i)=\eta(\bm{\varpi}_l;\balpha_i)$, which is a function of $\bm\varpi_l$ indexed by $\balpha_i\in\mathcal{A}$. Suppose the fixed sequence $(\bm\varpi_l:1\leq l\leq p)$ satisfies the following uniform convergence:
\[
\sup_{\balpha,\balpha'\in\mathcal{A}}\Big|
\frac{1}{p^\dagger}\sum_{l\in\mathcal{C}^\dagger}\Big(\eta(\bm{\varpi}_l;\balpha)-\eta(\bm{\varpi}_l;\balpha')\Big)^2-
\int\Big(\eta(\bm{\varpi};\balpha)-\eta(\bm{\varpi};\balpha')\Big)^2\,\mathrm{d}F_{\varpi}(\bm\varpi)\Big|=o(1).
\]
Then, condition \eqref{eq: def of nonsingularity} reduces to a restriction on the limit: for any $\varepsilon>0$, there exists $\Delta$ such that for all $\balpha,\balpha'\in\mathcal{A}$,
\[
\int \Big(\eta(\bm{\varpi};\balpha)-
\eta(\bm{\varpi};\balpha')\Big)^2\,\mathrm{d}F_{\varpi}(\bm{\varpi})< \Delta
\Rightarrow\|\balpha-\balpha'\|< \varepsilon,
\]
where ``$\Rightarrow$'' means ``implies''.
Conceptually, if we define $g: \balpha\mapsto\eta(\cdot;\balpha)$, a map from $\mathcal{A}$ to $L^2(\mathcal{Z})$ (a set of square-integrable functions on $\mathcal{Z}$ equipped with $L^2$-norm), then the above condition essentially says that the inverse map $g^{-1}$ exists and is uniformly continuous.

\begin{exmp}[Polynomial]
	Let $\eta(\varpi;\alpha)=1+\varpi\alpha+\varpi^2\alpha^2$,  $\mathcal{Z}=\mathcal{A}=[0,1]$, and $F_\varpi$ be the uniform distribution. Then, for $\alpha,\alpha'\in\mathcal{A}$,
	$\int_{[0,1]}(\eta(\varpi;\alpha)-
	\eta(\varpi;\alpha'))^2\,\mathrm{d}F_{\varpi}(\varpi)=
	(\alpha-\alpha',\alpha^2-(\alpha')^2)
	(\int_{[0,1]}(\varpi,\varpi^2)'(\varpi,\varpi^2)\,\mathrm{d}\varpi)
	(\alpha-\alpha',\alpha^2-(\alpha')^2)'$.
	The matrix in the middle is simply one block in the Hilbert matrix with the minimum eigenvalue bounded away from zero. Then, the desired result immediately follows.
	\hfill\qedsymbol
\end{exmp}

\begin{exmp}[Trigonometric function]
	Let $\eta(\varpi;\alpha)=\sin(\pi(\varpi+\alpha))$, $\mathcal{Z}=\mathcal{A}=[0,1]$, and $F_\varpi$ be the uniform distribution. For any $\alpha,\alpha'\in\mathcal{A}$,
	\[
	\int_{[0,1]}\big(\eta(\varpi;\alpha)-\eta(\varpi;\alpha')\big)^2\,\mathrm{d}\varpi
	=2\sin^2\left(\frac{\pi(\alpha-\alpha')}{2}\right).
	\]
	Thus, for every $\varepsilon\in(0,1)$ and all $|\alpha-\alpha'|>\varepsilon$, the integral is bounded below by $2\sin^2(\pi\varepsilon/2)>0$. This is the contrapositive of the desired result.
	\hfill\qedsymbol
\end{exmp}

\begin{remark}[Regime shift]
The ``smoothness'' of $\eta_l$ over the $l$ dimension is unnecessary in our general framework. For example, consider a simple artificial specification:
$\eta_l(\alpha)=\alpha^2+1$ for some $l=l_0\in\mathcal{C}^{\dagger}$ and $\eta_l(\alpha)=2\alpha^2+1$ for $l\neq l_0$.
Clearly, the case of $l=l_0$ is distinct from others. However, for any $\alpha,\tilde{\alpha}\in\mathcal{A}=[0,1]$ such that $|\alpha-\tilde{\alpha}|>\varepsilon$,
$\frac{1}{p^\dagger}\sum_{l\in\mathcal{C}^\dagger}
(\eta_{l}(\alpha)-\eta_{l}(\tilde{\alpha}))^2
= \frac{4(p^\dagger-1)+1}{p^\dagger}(\alpha^2-\tilde{\alpha}^2)^2 \geq \varepsilon^4$.
Multiple regime shifts and more complex factor structures are allowed as long as there are sufficiently many ``dimensions'' in which the difference in values of $\alpha$ can be detected.
\end{remark}

Next, note that the condition described above for the squared Euclidean distance is also sufficient for the informativeness requirement corresponding to the pseudo-max distance (condition (i) in Theorem \ref{thm: IMD hsk}). To see this, suppose that
$\sup_{\balpha,\balpha',\balpha''\in\mathcal{A}}
|\frac{1}{p^\dagger}\sum_{l\in\mathcal{C}^\dagger}(\eta(\bm{\varpi}_l;\balpha)-
\eta(\bm{\varpi}_l;\balpha'))\eta(\bm{\varpi}_l;\balpha'')-\int(\eta(\bm{\varpi};\balpha)-\eta(\bm{\varpi};\balpha'))\eta(\bm{\varpi};\balpha'')\,\mathrm{d}F_\varpi(\bm\varpi)|=o(1)$, and note that
\begin{align*}
&\sup_{\balpha''\in\mathcal{A}\atop \balpha''\neq \balpha,\balpha'}\left|\int(\eta(\bm{\varpi};
\balpha)-\eta(\bm{\varpi};\balpha'))\eta(\bm{\varpi};\balpha'')\,\mathrm{d}F_\varpi(\bm\varpi)\right|=o(1)\\
\Rightarrow
&\int(\eta(\bm{\varpi};\balpha)-\eta(\bm{\varpi};\balpha'))^2\,\mathrm{d}F_\varpi(\bm\varpi)=o(1).
\end{align*}
The previous condition for the squared Euclidean distance therefore establishes the informativeness requirement \eqref{eq: def of nonsingularity} with respect to the pseudo-max distance.

\subsubsection*{Non-degeneracy}

Theorem \ref{thm: verify LPCA} in the main paper verifies the high-level conditions in Assumption \ref{Assumption: LPCA}. The key requirement therein is the non-degeneracy of derivative vectors.
As discussed before, under mild conditions, we can make the sample average across $l$ converge and restate the condition in terms of the limit. For instance, consider $\mathscr{D}^{1}\eta_l$, which contains the $0$th- and $1$st-order derivatives. Then the requirement becomes
\[
s_{\min}\left\{\int \Big(\eta(\bm{\varpi};\balpha),\,\frac{\partial}{\partial\balpha}\eta(\bm{\varpi};\balpha)\Big)'\Big(\eta(\bm{\varpi};\balpha),\,\frac{\partial}{\partial\balpha}\eta(\bm{\varpi};\balpha)\Big)\,\mathrm{d}F_\varpi(\bm{\varpi})\right\}\geq c>0
\]
uniformly over $\balpha$. We can verify this sufficient condition for the two examples discussed before.

\begin{exmp}[Polynomial, continued]
	The first-order derivative is $\frac{\mathrm{d}}{\mathrm{d}\alpha}\eta(\varpi;\alpha)=\varpi+2\varpi^2\alpha$. Let $\bg(\varpi;\alpha)=(1+\varpi\alpha+\varpi^2\alpha^2,\, \varpi+2\varpi^2\alpha)'$. Then,
	\begin{align*}
		\int_{[0,1]}\bg(\varpi;\alpha)\bg(\varpi;\alpha)'\,\mathrm{d}\varpi&=
		\tilde{\bg}(\alpha)\Big(\int_{[0,1]}(1,\varpi,\varpi^2)'(1,\varpi,\varpi^2)\,\mathrm{d}\varpi\Big)
		\tilde{\bg}(\alpha)'\\
		&\succeq c\,\tilde{\bg}(\alpha)\tilde{\bg}(\alpha)',\quad\text{for some constant }c>0,\quad
		\tilde{\bg}(\alpha)=\bigg(\begin{array}{ccc}
			1&\alpha&\alpha^2\\
			0&1&2\alpha
		\end{array}
		\bigg).
	\end{align*}
	$\tilde{\bg}(\alpha)\tilde{\bg}(\alpha)'$ is positive definite for any $\alpha$. The coefficients of its characteristic polynomials are continuous on a compact support, and thus by Theorem 3.9.1 of \cite{Tyrtyshnikov_2012_book}, their eigenvalues are also continuous on the support. Then, the desired result follows.
	\hfill\qedsymbol
\end{exmp}

\begin{exmp}[Trigonometric function, continued]
	Note that $\frac{\mathrm{d}}{\mathrm{d}\alpha}\eta(\varpi;\alpha)=\pi\cos(\pi(\varpi+\alpha))$. Let $\bg(\varpi;\alpha)=(\sin(\pi(\varpi+\alpha)), \pi\cos(\pi(\varpi+\alpha)))'$. Then,
	\[
	\begin{split}
		&\int_{[0,1]}\bg(\varpi;\alpha)\bg(\varpi;\alpha)'\,\mathrm{d}\varpi=
		\tilde{\bc}\Big(\int_{[0,1]}\tilde{\bg}(\varpi;\alpha)\tilde{\bg}(\varpi;\alpha)'\,\mathrm{d}\varpi\Big)\tilde{\bc}
		=\frac{1}{2\pi}\tilde{\bc}\bigg(\begin{array}{cc}
			\pi&0\\
			0&\pi
		\end{array}\bigg)\tilde{\bc},
		\quad\text{where}\\
		&\tilde{\bg}(\varpi;\alpha)=(\sin(\pi(\varpi+\alpha)), \cos(\pi(\varpi+\alpha)))', \quad\tilde{\bc}=\bigg(\begin{array}{cc}
			1&0\\
			0&\pi
		\end{array}
		\bigg).
	\end{split}
	\]
	Then, the desired result follows.
	\hfill\qedsymbol
\end{exmp}

\section{Notation}\label{SA-sec: notation}

\textbf{Latent functions.} Set $\nabla^0\eta_l=\eta_l$. For $\ell\geq1$, $\nabla^\ell\eta_l$ is the column vector of distinct $\ell$th-order partial derivatives, each included once in a common fixed order across features and evaluation points. Define $\mathscr D^\kappa\eta_l=(\nabla^0\eta_l',\cdots,\nabla^\kappa\eta_l')'$ for $\kappa\geq0$. Boundary derivatives are limits from within the support. Write $\nabla\eta_l:=\nabla^1\eta_l$ for the gradient and $\nabla\bmeta:=(\nabla\eta_1,\cdots,\nabla\eta_p)'$ for the Jacobian. These conventions also apply to feature subvectors and other scalar functions.

\textbf{Matrices.}
For vectors, $\|\bv\|=\sqrt{\bv'\bv}$ is the Euclidean norm. For an $m\times n$ matrix $\bA$, let $s_j(\bA)$ be its $j$th largest singular value, $1\leq j\leq m\wedge n$, with $s_{\max}(\bA)=s_1(\bA)$ and $s_{\min}(\bA)=s_{m\wedge n}(\bA)$. The operator, Frobenius, and entrywise sup-norms are $\|\bA\|=s_{\max}(\bA)$, $\|\bA\|_{\mathtt F}=\sqrt{\tr(\bA\bA')}$, and $\|\bA\|_{\max}=\max_{i,j}|a_{ij}|$. Also, $\|\bA\|_\infty=\max_i\sum_j|a_{ij}|$ and $\|\bA\|_{2\rightarrow\infty}=\max_i\|\bA_{i\cdot}\|$.
The matrix $\bI_m$ is the $m\times m$ identity; for symmetric matrices, $\bA\succeq\bB$ means $\bA-\bB$ is positive semidefinite. Write $\bA_{i\cdot}$ and $\bA_{\cdot j}$ for rows and columns, and $\bA_{\mathcal R\cdot}$ and $\bA_{\cdot\mathcal C}$ for submatrices with rows $\mathcal R\subseteq[m]$ and columns $\mathcal C\subseteq[n]$.
For full-column-rank $\bA$, $\bP_{\bA}=\bA(\bA'\bA)^{-1}\bA'$ and $\bM_{\bA}=\bI_m-\bP_{\bA}$ project onto its column space and orthogonal complement. For nonsingular square $\bA$ with singular-value decomposition $\bA=\bU\bD\bV'$, $\mathtt{sgn}(\bA)=\bU\bV'$ is its orthogonal polar factor.
Write $\bmeta^\dagger=(\eta_l:l\in\mathcal C^\dagger)'$, $\bmeta_i^\dagger=\bmeta^\dagger(\balpha_i)=(\bH^\dagger_{i\cdot})'$, $\bx_i^\dagger=(\bX^\dagger_{i\cdot})'$, and $\bu_i^\dagger=(\bU^\dagger_{i\cdot})'$, with matching features in their original index order. Superscript $\ddagger$ similarly refers to $\mathcal C^\ddagger$.

\textbf{Asymptotics.}
Write $a_n\lesssim b_n$ when $\limsup_n|a_n/b_n|<\infty$, and $a_n\lesssim_\P b_n$ or $a_n=O_\P(b_n)$ when $a_n/b_n$ is bounded in probability. Define $a_n\gtrsim b_n$ by $b_n\lesssim a_n$, and $a_n\gtrsim_\P b_n$ by $b_n\lesssim_\P a_n$; $\asymp$ and $\asymp_\P$ require the corresponding bounds in both directions. The notation $a_n=o(b_n)$ or $o_\P(b_n)$ means $a_n/b_n\rightarrow0$ or $\rightarrow_\P0$, respectively, where $\rightarrow_\P$ denotes convergence in probability. For indexed sequences and positive rates $b_{n,i}$, uniform $O_\P(b_{n,i})$ and $o_\P(b_{n,i})$ mean $\max_{i\in[n]}|a_{n,i}/b_{n,i}|=O_\P(1)$ and $o_\P(1)$, respectively. ``W.p.a. 1'' means ``with probability approaching one''.

\textbf{Others.}
Write $a\vee b=\max\{a,b\}$, $a\wedge b=\min\{a,b\}$, and $\delta_{Kp}=(K^{1/2}\wedge p^{1/2})/\sqrt{\log(n\vee p)}$. For a finite set $\mathcal S$, $|\mathcal S|$ is its cardinality, and $[m]=\{1,\cdots,m\}$ for positive integers $m$.

\section{Proofs for the Main Paper}\label{SA-sec: proofs}

\subsection{Indirect Matching}\label{SA-subsec: proofs matching}
\subsubsection{Useful Lemmas}
We collect some useful lemmas below, and their proofs are available in \ref{SA sec: proofs of lemmas}.

For any $\balpha_0\in\mathcal{A}$, select the $K$ nearest sample observations in Euclidean distance. If $\balpha_0=\balpha_i$, include $i$ and select the $K-1$ nearest other observations. Ties are resolved by the original indices. Let $(j_k^*(\balpha_0):1\leq k\leq K)$ list the selected indices in their original index order.

\begin{lem} \label{lem: direct matching}
	Suppose that Assumption \ref{Assumption: Regularities}(a) holds. If $\frac{\log (n/K)}{K}=o(1)$
	and $\frac{K\log n}{n}=o(1)$, then for some fixed constants $c,C>0$,
	\[
	c\Big(\frac{K}{n}\Big)^{\frac{1}{r}}\leq
	\inf_{\balpha_0\in\mathcal{A}}\max_{1\leq k\leq K}\|\balpha_{j_k^*(\balpha_0)}-\balpha_0\|
	\leq \sup_{\balpha_0\in\mathcal{A}}\max_{1\leq k\leq K}\|\balpha_{j_k^*(\balpha_0)}-\balpha_0\|
	\leq C\Big(\frac{K}{n}\Big)^{\frac{1}{r}},\;
	\text{w.p.a.} 1.
	\]
	Moreover, under the same conditions, for some fixed constant $c'>0$,
	\[
	\min_{1\leq i\leq n}\max_{1\leq k\leq K} \|\balpha_i-\balpha_{j_k(i)}\|\geq
	c'\Big(\frac{K}{n}\Big)^{\frac{1}{r}}, \; \text{w.p.a. } 1.
	\]
\end{lem}
\begin{lem}[Conditional Maximal Inequality]\label{SA lem: matching maximal inequality}
	Let $\mathscr{G}$ be a sigma field. For each $s\in[N]$, let $(Z_{ls}:l\in[P])$ be independent conditional on $\mathscr{G}$. Suppose $\E[Z_{ls}|\mathscr{G}]=0$ and, for some fixed $\nu_0>0$, the conditional $(2+\nu_0)$-th moments are bounded almost surely by a fixed constant, uniformly over $l$ and $s$. Independence across $s$ is unnecessary. If $N,P\rightarrow\infty$ and $N^{2/\nu_0}\log N/P\lesssim 1$, then
	\[
	\max_{s\in[N]}\left|\frac{1}{P}\sum_{l=1}^{P}Z_{ls}\right|\lesssim_\P\sqrt{\frac{\log N}{P}}.
	\]
\end{lem}

\subsubsection{Proof of Theorem \ref{thm: IMD}}
Recall the oracle indices $(j_k^*(\balpha_0):1\leq k\leq K)$ defined before Lemma \ref{lem: direct matching}.
When $\balpha_0=\balpha_i$, this oracle neighborhood includes $i$. The conclusion is immediate when $j_k(i)=i$. For every other $j_k(i)\in\mathcal{N}_i$, by Assumption \ref{Assumption: Matching}, the definition of the nearest neighbors, and Lemma \ref{lem: direct matching}, we have, uniformly over $i$ and $k$,
\begin{align*}
	\rho(\bmeta_i^\dagger,\bmeta_{j_k(i)}^\dagger)
	&\leq \rho(\bx_i^\dagger,\bx_{j_k(i)}^\dagger)-\rho_0+O_\P(a_p)\\
	&\leq \max_{\substack{1\leq q\leq K\\j_q^*(\balpha_i)\neq i}}
	\big\{\rho(\bx_i^\dagger,\bx_{j_q^*(\balpha_i)}^\dagger)-\rho_0\big\}+O_\P(a_p)\\
	&\leq \max_{\substack{1\leq q\leq K\\j_q^*(\balpha_i)\neq i}}
	\rho(\bmeta_i^\dagger,\bmeta_{j_q^*(\balpha_i)}^\dagger)+O_\P(a_p)
	\lesssim_\P (K/n)^{\uvarsig/r}+a_p.
\end{align*}
By the lower bound in Assumption \ref{Assumption: Matching}(b), we have
\[
	\|\balpha_i-\balpha_{j_k(i)}\|\lesssim_\P (K/n)^{\uvarsig/(\lvarsig r)}+a_p^{1/\lvarsig}
\]
uniformly over all noncentral neighbors. The center has zero discrepancy, so the result follows.

\subsubsection{Proofs for the Three Distance Choices in Appendix \ref{sec: appendix matching}}
\paragraph{Proof of Theorem \ref{thm: IMD hom}}
	For any $i\neq j$,
	\begin{align*}
	\rho(\bx^{\dagger}_{i},\bx^{\dagger}_{j})
	=\frac{1}{p^\dagger}\|\bmeta^{\dagger}_{i}-\bmeta^{\dagger}_{j}\|^2
	+\frac{1}{p^\dagger}\|\bu^\dagger_i-\bu^\dagger_{j}\|^2
	+\frac{2}{p^\dagger}(\bmeta^{\dagger}_{i}-\bmeta^{\dagger}_{j})'(\bu^\dagger_{i}-\bu^\dagger_{j}).
	\end{align*}
	By condition (i), the second term satisfies
	\begin{align*}
	&\frac{1}{p^\dagger}\|\bu^\dagger_i-\bu^\dagger_j\|^2-2c_u\\
	={}&\frac{1}{p^\dagger}\sum_{l\in\mathcal{C}^\dagger}\{u_{il}^2-\E[u_{il}^2|\mathscr{F}]\}
	+\frac{1}{p^\dagger}\sum_{l\in\mathcal{C}^\dagger}\{u_{jl}^2-\E[u_{jl}^2|\mathscr{F}]\}
	-\frac{2}{p^\dagger}\sum_{l\in\mathcal{C}^\dagger}u_{il}u_{jl}.
	\end{align*}
	The summands are centered and conditionally independent across $l$. Lemma \ref{SA lem: matching maximal inequality} applies to the first two sums with $N=n$, $P=p^\dagger$, and $\nu_0=\nu$. For the last sum, independence across units ensures centering, while condition (ii) bounds the $(2+2\nu)$-th moments. Hence the lemma applies with at most $N=n^2$ and $\nu_0=2\nu$. Under the theorem's rate restriction,
	\[
	\max_{i\neq j}\left|\frac{1}{p^\dagger}\|\bu^\dagger_i-\bu^\dagger_j\|^2-2c_u\right|
	\lesssim_\P\sqrt{\frac{\log n}{p}}.
	\]
	Moreover, $\E[(\eta_l(\balpha_i)-\eta_l(\balpha_j))(u_{il}-u_{jl})|\mathscr{F}]=0$. After expanding this term into four pair-indexed sums, Lemma \ref{SA lem: matching maximal inequality} applies with $N=n^2$, $P=p^\dagger$, and $\nu_0=2\nu$. It follows that
	\[
	\max_{i\neq j}\left|\frac{1}{p^\dagger}(\bmeta^{\dagger}_{i}-\bmeta^{\dagger}_{j})'(\bu^\dagger_{i}-\bu^\dagger_{j})\right|
	\lesssim_\P\sqrt{\frac{\log n}{p}}.
	\]
	Thus, Assumption \ref{Assumption: Matching}(a) holds with $a_p=\sqrt{(\log n)/p}$ and $\rho_0=2c_u$.

	For part (b), by Assumption \ref{Assumption: Primitive Smoothness}, we have
	\[
	\rho(\bmeta_i^\dagger,\bmeta_j^\dagger)=\frac{1}{p^\dagger}\|\bmeta_i^\dagger-\bmeta_j^\dagger\|^2\lesssim\|\balpha_i-\balpha_j\|^2,
	\]
	so $\uvarsig=2$. To prove the lower bound, fix a sufficiently small $\varepsilon>0$. For any $\tau>0$, condition (iii) supplies a fixed $\Delta>0$ such that, for all sufficiently large $n,p$, every pair with $\rho(\bmeta_i^\dagger,\bmeta_j^\dagger)<\Delta$ also satisfies $\|\balpha_i-\balpha_j\|\leq\varepsilon$ with probability at least $1-\tau$. Work on the intersection of this event and the event in condition (iv). For such a pair, let $\boldsymbol{\delta}_{ij}=\balpha_j-\balpha_i$. A second-order expansion gives $\bmeta_j^\dagger-\bmeta_i^\dagger=\nabla\bmeta^\dagger(\balpha_i)\boldsymbol{\delta}_{ij}+\bR_{ij}$, where $\|\bR_{ij}\|/\sqrt{p^\dagger}\lesssim\|\boldsymbol{\delta}_{ij}\|^2$. Hence
	\begin{align*}
	\sqrt{\rho(\bmeta_i^\dagger,\bmeta_j^\dagger)}
	&\geq\frac{1}{\sqrt{p^\dagger}}\|\nabla\bmeta^\dagger(\balpha_i)\boldsymbol{\delta}_{ij}\|-\frac{1}{\sqrt{p^\dagger}}\|\bR_{ij}\|\\
	&\geq\sqrt{\underline c}\,\|\boldsymbol{\delta}_{ij}\|-C\|\boldsymbol{\delta}_{ij}\|^2
	\geq\frac{\sqrt{\underline c}}{2}\|\boldsymbol{\delta}_{ij}\|.
	\end{align*}
	For the remaining pairs, $\rho(\bmeta_i^\dagger,\bmeta_j^\dagger)\geq\Delta$. Since $\mathcal{A}$ is compact,
	\[
	\rho(\bmeta_i^\dagger,\bmeta_j^\dagger)\geq\frac{\Delta}{\diam(\mathcal{A})^2}\|\balpha_i-\balpha_j\|^2.
	\]
	The two cases together prove $\lvarsig=2$.

\paragraph{Proof of Theorem \ref{thm: IMD hsk}}
	For any $j\neq i$ and $k\notin\{i,j\}$,
	\begin{align*}
	\frac{1}{p^\dagger}(\bx^{\dagger}_j-\bx^{\dagger}_i)'\bx^{\dagger}_k
	={}&\frac{1}{p^\dagger}(\bmeta^{\dagger}_j-\bmeta^{\dagger}_i)'\bmeta^{\dagger}_k
	+\frac{1}{p^\dagger}(\bmeta^{\dagger}_j-\bmeta^{\dagger}_i)'\bu^\dagger_k\\
	&+\frac{1}{p^\dagger}(\bu^\dagger_j-\bu^\dagger_i)'\bmeta^{\dagger}_k
	+\frac{1}{p^\dagger}(\bu^\dagger_j-\bu^\dagger_i)'\bu^\dagger_k.
	\end{align*}
	The three stochastic terms reduce to pair-indexed sums because
	\begin{align*}
	\max_{i\neq j,\,k\notin\{i,j\}}\left|\frac{1}{p^\dagger}(\bmeta_j^\dagger-\bmeta_i^\dagger)'\bu_k^\dagger\right|
	&\leq 2\max_{a,k\in[n]}\left|\frac{1}{p^\dagger}(\bmeta_a^\dagger)'\bu_k^\dagger\right|,\\
	\max_{i\neq j,\,k\notin\{i,j\}}\left|\frac{1}{p^\dagger}(\bu_j^\dagger-\bu_i^\dagger)'\bmeta_k^\dagger\right|
	&\leq 2\max_{a,k\in[n]}\left|\frac{1}{p^\dagger}(\bu_a^\dagger)'\bmeta_k^\dagger\right|,\\
	\max_{i\neq j,\,k\notin\{i,j\}}\left|\frac{1}{p^\dagger}(\bu_j^\dagger-\bu_i^\dagger)'\bu_k^\dagger\right|
	&\leq 2\max_{\substack{a,k\in[n]\\a\neq k}}\left|\frac{1}{p^\dagger}(\bu_a^\dagger)'\bu_k^\dagger\right|.
	\end{align*}
	Conditional on $\mathscr{F}$, the summands are centered and independent across $l\in\mathcal{C}^\dagger$. Their conditional $(2+\nu)$-th moments are uniformly bounded under Assumption \ref{Assumption: Regularities}; for the last line, independence across $a\neq k$ is also used. Thus, Lemma \ref{SA lem: matching maximal inequality} applies with $N\leq n^2$, $P=p^\dagger$, and $\nu_0=\nu$. The theorem's rate restriction then implies
	\[
	\max_{i\neq j}\left|\rho(\bx_i^\dagger,\bx_j^\dagger)-\rho(\bmeta_i^\dagger,\bmeta_j^\dagger)\right|
	\lesssim_\P\sqrt{\frac{\log n}{p}}.
	\]
	Thus, Assumption \ref{Assumption: Matching}(a) holds with $a_p=\sqrt{(\log n)/p}$ and $\rho_0=0$.

	For part (b), Assumption \ref{Assumption: Regularities}(b) bounds the levels of $\eta_l$, while Assumption \ref{Assumption: Primitive Smoothness} bounds their first derivatives. By the Cauchy--Schwarz inequality, for all $k\notin\{i,j\}$,
	\begin{align*}
	\left|\frac{1}{p^\dagger}(\bmeta_j^\dagger-\bmeta_i^\dagger)'\bmeta_k^\dagger\right|
	&\leq\left\{\frac{1}{p^\dagger}\sum_{l\in\mathcal{C}^\dagger}(\eta_l(\balpha_j)-\eta_l(\balpha_i))^2\right\}^{1/2}
	\left\{\frac{1}{p^\dagger}\sum_{l\in\mathcal{C}^\dagger}\eta_l(\balpha_k)^2\right\}^{1/2}\\
	&\lesssim\|\balpha_j-\balpha_i\|,
	\end{align*}
	which proves $\uvarsig=1$.

	We next prove the lower bound. The rate conditions of Lemma \ref{lem: direct matching} hold with $\lceil(\log n)^2\rceil$ in place of $K$. Its first conclusion therefore yields
	\[
	\max_{\substack{i,j\in[n]\\i\neq j}}\sup_{\balpha\in\mathcal{A}}\min_{\ell\in[n]\setminus\{i,j\}}\|\balpha_\ell-\balpha\|
	=O_\P\left[\left\{(\log n)^2/n\right\}^{1/r}\right]=o_\P(1).
	\]
	Fix a sufficiently small $\varepsilon>0$ and consider first $0<\|\boldsymbol{\delta}_{ij}\|\leq\varepsilon$, where $\boldsymbol{\delta}_{ij}=\balpha_i-\balpha_j$. Let
	\[
	\bm{J}_i=\nabla\bmeta^\dagger(\balpha_i),\qquad
	\bv_{ij}=\frac{\bm{J}_i\boldsymbol{\delta}_{ij}}{\|\bm{J}_i\boldsymbol{\delta}_{ij}\|},\qquad
	f_{ij}(\balpha)=\frac{1}{\sqrt{p^\dagger}}\bv_{ij}'\bmeta^\dagger(\balpha).
	\]
	On the event in condition (ii),
	\begin{align*}
	\|\nabla f_{ij}(\balpha_i)\|^2
	&=\frac{\boldsymbol{\delta}_{ij}'\{(\bm{J}_i'\bm{J}_i)/p^\dagger\}^2\boldsymbol{\delta}_{ij}}
	{\boldsymbol{\delta}_{ij}'\{(\bm{J}_i'\bm{J}_i)/p^\dagger\}\boldsymbol{\delta}_{ij}}
	\geq s_{\min}\left(\frac{1}{p^\dagger}\bm{J}_i'\bm{J}_i\right)\geq\underline c.
	\end{align*}
	The derivative bounds in Assumption \ref{Assumption: Primitive Smoothness} are uniform. Since $\bv_{ij}$ is a unit vector, the Cauchy--Schwarz inequality transfers these bounds to the first and second derivatives of $f_{ij}$. The support $\mathcal{A}$ contains a fixed closed ball $\overline B(\balpha_0,r_0)$. For a sufficiently small fixed $s>0$, define
	\[
	\balpha_\pm=(1-s)\balpha_i+s\left\{\balpha_0\pm r_0\frac{\nabla f_{ij}(\balpha_i)}{\|\nabla f_{ij}(\balpha_i)\|}\right\}.
	\]
	Both points lie in $\mathcal{A}$ by convexity. A second-order expansion yields
	\[
	|f_{ij}(\balpha_+)-f_{ij}(\balpha_-)|\geq 2sr_0\|\nabla f_{ij}(\balpha_i)\|-Cs^2\gtrsim1.
	\]
	The final bound follows because $\|\nabla f_{ij}(\balpha_i)\|\geq\sqrt{\underline c}$ and $s$ is sufficiently small.
	At least one of these two points, denoted by $\widetilde{\balpha}_{ij}$, satisfies $|f_{ij}(\widetilde{\balpha}_{ij})|\gtrsim1$. The uniform sample-coverage result supplies some $\ell\notin\{i,j\}$ with $\|\balpha_\ell-\widetilde{\balpha}_{ij}\|=o_\P(1)$. Since $f_{ij}$ is uniformly Lipschitz, $|f_{ij}(\balpha_\ell)|$ is bounded away from zero uniformly with probability approaching one.

	A second-order expansion gives $\bmeta_i^\dagger-\bmeta_j^\dagger=\bm{J}_i\boldsymbol{\delta}_{ij}+\bR_{ij}$, where $\|\bR_{ij}\|/\sqrt{p^\dagger}\lesssim\|\boldsymbol{\delta}_{ij}\|^2$. The zeroth-order bound in Assumption \ref{Assumption: Regularities}(b) implies $\|\bmeta_\ell^\dagger\|/\sqrt{p^\dagger}\lesssim1$. Therefore,
	\begin{align*}
	\rho(\bmeta_i^\dagger,\bmeta_j^\dagger)
	&\geq\frac{1}{p^\dagger}|(\bmeta_i^\dagger-\bmeta_j^\dagger)'\bmeta_\ell^\dagger|\\
	&\geq\frac{\|\bm{J}_i\boldsymbol{\delta}_{ij}\|}{\sqrt{p^\dagger}}
	\frac{|\bv_{ij}'\bmeta_\ell^\dagger|}{\sqrt{p^\dagger}}
	-\frac{\|\bR_{ij}\|}{\sqrt{p^\dagger}}\frac{\|\bmeta_\ell^\dagger\|}{\sqrt{p^\dagger}}\\
	&\gtrsim\|\boldsymbol{\delta}_{ij}\|-C\|\boldsymbol{\delta}_{ij}\|^2
	\gtrsim\|\boldsymbol{\delta}_{ij}\|.
	\end{align*}
	For $\|\boldsymbol{\delta}_{ij}\|>\varepsilon$ and any $\tau>0$, condition (i) supplies a fixed $\Delta>0$ such that $\rho(\bmeta_i^\dagger,\bmeta_j^\dagger)\geq\Delta$ uniformly on events with limiting probability at least $1-\tau$. Since $\mathcal{A}$ is compact,
	\[
	\rho(\bmeta_i^\dagger,\bmeta_j^\dagger)\geq\frac{\Delta}{\diam(\mathcal{A})}\|\balpha_i-\balpha_j\|.
	\]
	The local and nonlocal cases together prove $\lvarsig=1$.

\paragraph{Proof of Theorem \ref{thm: dist of average}}
	By the reverse triangle inequality, we have
	\begin{align*}
	|\rho(\bx_i^\dagger,\bx_j^\dagger)-\rho(\bmeta_i^\dagger,\bmeta_j^\dagger)|
	&\leq\left|\frac{1}{p^\dagger}\sum_{l\in\mathcal{C}^\dagger}(u_{il}-u_{jl})\right|
	\leq 2\max_{a\in[n]}\left|\frac{1}{p^\dagger}\sum_{l\in\mathcal{C}^\dagger}u_{al}\right|.
	\end{align*}
	Conditional on $\mathscr{F}$, the summands are centered and independent across $l$. Apply Lemma \ref{SA lem: matching maximal inequality} with $N=n$, $P=p^\dagger$, and $\nu_0=\nu$. Under the theorem's rate restriction, the last expression is $O_\P(\sqrt{(\log n)/p})$. Thus, Assumption \ref{Assumption: Matching}(a) holds with $a_p=\sqrt{(\log n)/p}$ and $\rho_0=0$.

	For part (b), define
	\[
	g_p(\alpha)=\frac{1}{p^\dagger}\sum_{l\in\mathcal{C}^\dagger}\eta_l(\alpha).
	\]
	Assumption \ref{Assumption: Primitive Smoothness} implies $|g_p(\alpha_i)-g_p(\alpha_j)|\lesssim|\alpha_i-\alpha_j|$, so $\uvarsig=1$. It remains to establish the lower bound. Since $r=1$, write $\mathcal{A}=[a,b]$. The function $\bar\eta$ is fixed, and its derivative is continuous and nonzero on this compact interval. Hence $c:=\inf_{\alpha\in\mathcal{A}}|\bar\eta^{(1)}(\alpha)|>0$. The derivative has one sign; multiplying $g_p$ and $\bar\eta$ by $-1$ if necessary, suppose that $\bar\eta^{(1)}\geq c$. The same assumption gives $\sup_{\alpha\in\mathcal{A}}|g_p^{(2)}(\alpha)|\leq C$. Choose a fixed $0<h<\min\{(b-a)/2,c/(2C)\}$. Condition (i) makes the approximation error uniform. A second-order expansion therefore gives, uniformly over $\alpha$,
	\begin{align*}
	ch-o(1)&\leq g_p(\alpha+h)-g_p(\alpha)\leq hg_p^{(1)}(\alpha)+\frac{C}{2}h^2,
	&&\alpha\in[a,b-h],\\
	ch-o(1)&\leq g_p(\alpha)-g_p(\alpha-h)\leq hg_p^{(1)}(\alpha)+\frac{C}{2}h^2,
	&&\alpha\in(b-h,b].
	\end{align*}
	The two $o(1)$ terms are deterministic and uniform over $\alpha$ as $p\rightarrow\infty$. Thus, $\inf_{\alpha\in\mathcal{A}}g_p^{(1)}(\alpha)\geq c/2$ for all sufficiently large $p$. By the mean value theorem,
	\[
	\rho(\bmeta_i^\dagger,\bmeta_j^\dagger)=|g_p(\alpha_i)-g_p(\alpha_j)|\geq\frac{c}{2}|\alpha_i-\alpha_j|
	\]
	uniformly over $i,j$, proving $\lvarsig=1$.

\subsubsection{Proofs of Supporting Lemmas}\label{SA sec: proofs of lemmas}
\paragraph{Proof of Lemma \ref{lem: direct matching}}
	Let $U_K(\balpha_0)=\max_{1\leq k\leq K}\|\balpha_{j_k^*(\balpha_0)}-\balpha_0\|$. Define a cube around $\balpha_0$: $\mathcal{H}_\zeta(\balpha_0)=[\balpha_0-\zeta,\, \balpha_0+\zeta)\cap\mathcal{A}$, where $\zeta=(\frac{C'K}{\underline{f}_{\alpha}n})^{1/r}$ for some fixed constant $C'>1$ and $\underline{f}_{\alpha}$ denotes the minimum of the density of $\balpha_i$. Let $F$ and $F_n$ denote the population and empirical distribution functions for $\balpha_i$ respectively. Since the density is strictly positive,
	the fixed compact convex support and the density lower bound ensure that $\inf_{\balpha_0\in\mathcal{A}}F(\mathcal{H}_\zeta(\balpha_0))\gtrsim \underline{f}_{\alpha}\zeta^r$. Choose $C'$ sufficiently large so that this infimum is greater than $2K/n$.
	On the other hand,
	\begin{align*}
		&\,\P\Big(\sup_{\balpha_0\in\mathcal{A}}U_K(\balpha_0)>\sqrt{r}\,\zeta\Big)\\
		\leq &\,\P\bigg(\exists \balpha_0\in\mathcal{A}, F_n\Big(\mathcal{H}_\zeta(\balpha_0)\Big)<\frac{K}{n}\bigg)\\
		\leq&\,\P\bigg(\exists \balpha_0\in\mathcal{A}, \Big|F_n\Big(\mathcal{H}_\zeta(\balpha_0)\Big)-F\Big(\mathcal{H}_\zeta(\balpha_0)\Big)\Big|>F\Big(\mathcal{H}_\zeta(\balpha_0)\Big)
		-\frac{K}{n}\bigg)\\
	\leq&\,\P\bigg(\sup_{\balpha_0\in\mathcal{A}} \Big|F_n\Big(\mathcal{H}_\zeta(\balpha_0)\Big)-F\Big(\mathcal{H}_\zeta(\balpha_0)\Big)\Big|>\frac{K}{n}\bigg).
	\end{align*}
	Define the oscillation modulus $\omega_n(\varepsilon):=\sup_{\balpha_0\in\mathcal{A}}
	\sqrt{n}|F_n(\mathcal{H}_{\varepsilon}(\balpha_0))-F(\mathcal{H}_\varepsilon(\balpha_0))|$.
	By \citet[Theorem 2.3]{Stute_1984_AoP}, if $\frac{(n/K)\log (n/K)}{n}=o(1)$
	and $\frac{K\log n}{n}=o(1)$, then
	$\omega_n(\zeta)\lesssim_\P\sqrt{\zeta^{r}\log \zeta^{-r}}$.
	This suffices to show that the probability on the last line goes to zero.

	On the other hand, since the density of $\balpha_i$ is bounded from above as well, it follows similarly that $\inf_{\balpha_0\in\mathcal{A}}U_K(\balpha_0)\geq c(K/n)^{1/r}$ w.p.a. $1$. Thus, we have proved the first result.

	Next, by definition, $\max_{1\leq k\leq K}\|\balpha_i-\balpha_{j_k(i)}\|\geq \max_{1\leq k\leq K}\|\balpha_i-\balpha_{j^*_k(\balpha_i)}\|$. Then,  the second result follows immediately.

\paragraph{Proof of Lemma \ref{SA lem: matching maximal inequality}}
	Let $\tau_P=\sqrt{P/\log N}$ and write $Z_{ls}=Z_{ls}^-+Z_{ls}^+$, where
	\begin{align*}
	Z_{ls}^-&=Z_{ls}\I(|Z_{ls}|\leq\tau_P)-\E[Z_{ls}\I(|Z_{ls}|\leq\tau_P)|\mathscr{G}],\\
	Z_{ls}^+&=Z_{ls}\I(|Z_{ls}|>\tau_P)-\E[Z_{ls}\I(|Z_{ls}|>\tau_P)|\mathscr{G}].
	\end{align*}
	Conditional on $\mathscr{G}$, $|Z_{ls}^-|\leq2\tau_P$ and $\E[(Z_{ls}^-)^2|\mathscr{G}]\lesssim1$. By Bernstein's inequality followed by a union bound, for every $\delta>0$ and some fixed $C>0$,
	\[
	\P\left(\max_{s\in[N]}\left|\frac{1}{P}\sum_{l=1}^P Z_{ls}^-\right|>\delta\,\middle|\,\mathscr{G}\right)
	\leq2N\exp\left\{-\frac{P\delta^2}{C(1+\tau_P\delta)}\right\}.
	\]
	For the tail terms, conditional independence makes the variance additive. By Chebyshev's inequality and the moment bound, we have
	\begin{align*}
	&\P\left(\max_{s\in[N]}\left|\frac{1}{P}\sum_{l=1}^P Z_{ls}^+\right|>\delta\,\middle|\,\mathscr{G}\right)
	\leq\frac{N}{P^2\delta^2}\max_{s\in[N]}\sum_{l=1}^P\E[(Z_{ls}^+)^2|\mathscr{G}]
	\lesssim\frac{N}{P\tau_P^{\nu_0}\delta^2}.
	\end{align*}
	Set $\delta=A\sqrt{(\log N)/P}$ for a fixed $A>0$. The first probability is bounded by
	\[
	2N\exp\left\{-\frac{A^2\log N}{C(1+A)}\right\}.
	\]
	Choose $A$ sufficiently large. This bound then converges to zero as $N\rightarrow\infty$. Under the rate condition,
	\[
	\frac{N}{P\tau_P^{\nu_0}\delta^2}
	=\frac{N(\log N)^{\nu_0/2-1}}{A^2P^{\nu_0/2}}
	\lesssim\frac{1}{A^2\log N}=o(1).
	\]
	Taking expectations of these conditional bounds proves the result.

\subsection{Local PCA}\label{SA-subsec: proofs local PCA}
\subsubsection{Useful Facts and Lemmas}\label{SA sec: useful lemmas}
We first collect some useful facts and introduce the notation used in this subsection. All arguments involving the PCA-block errors are understood conditionally on $\mathscr{F}$ and the matching block $\bX^\dagger$. Under this conditioning, the neighborhoods and the noise-free local matrices are fixed, and $\ttd_i$ is treated as fixed.

\medskip
\noindent\textit{Sample principal components.}
For each neighborhood $\mathcal{N}_i$, let $\widehat\upsilon_{1,\subi}^2\geq\cdots\geq\widehat\upsilon_{p^\ddagger,\subi}^2\geq0$ be the ordered eigenvalues of $(p^\ddagger K)^{-1}\bX_{\subi}'\bX_{\subi}$. Let $\widehat\bF_{\subi}$ collect the leading $\ttd_i$ eigenvectors, scaled such that
\[
\frac{1}{p^\ddagger K}\bX_{\subi}'\bX_{\subi}\widehat\bF_{\subi}=\widehat\bF_{\subi}\widehat\bUpsilon_{\subi}^2,\qquad
\frac{1}{p^\ddagger}\widehat\bF_{\subi}'\widehat\bF_{\subi}=\bI_{\ttd_i},\qquad
\widehat\bUpsilon_{\subi}=\diag\{\widehat\upsilon_{1,\subi},\ldots,\widehat\upsilon_{\ttd_i,\subi}\}.
\]
The estimated factor loadings therefore satisfy
\[
\widehat\bLambda_{\subi}=\frac{1}{p^\ddagger}\bX_{\subi}\widehat\bF_{\subi},\qquad
\frac{1}{K}\widehat\bLambda_{\subi}'\widehat\bLambda_{\subi}=\widehat\bUpsilon_{\subi}^2.
\]

\medskip
\noindent\textit{Noise-free principal components.}
By the definition of the noise-free PCA target, we have
\[
\frac{1}{p^\ddagger}\bF_{\subi}'\bF_{\subi}=\bI_{\ttd_i},\qquad
\frac{1}{K}\bLambda_{\subi}'\bLambda_{\subi}=\bUpsilon_{\subi}^2,\qquad
\bLambda_{\subi}=\frac{1}{p^\ddagger}\bH_{\subi}\bF_{\subi}.
\]
The following exact identities also hold:
\[
\frac{1}{p^\ddagger K}\bH_{\subi}'\bH_{\subi}\bF_{\subi}=\bF_{\subi}\bUpsilon_{\subi}^2,\qquad
\bXi_{\subi}\bF_{\subi}=\bm{0},\qquad
\bXi_{\subi}'\bLambda_{\subi}=\bm{0}.
\]
Thus, $\upsilon_{j,\subi}^2$ is the $j$th retained eigenvalue of $(p^\ddagger K)^{-1}\bH_{\subi}'\bH_{\subi}$.

\medskip
\noindent\textit{Signal-strength groups.}
By Assumption \ref{Assumption: LPCA}(b), the ordered retained components can be partitioned into consecutive groups of comparable strength. Write
\[
[\ttd_i]=\bigcup_{\ell=0}^{g_i-1}\mathcal{C}_{i\ell},\qquad
\mathcal{C}_{i\ell}=\{\bar\ttd_{i(\ell-1)}+1,\ldots,\bar\ttd_{i\ell}\},\qquad
\bar\ttd_{i\ell}:=\sum_{s=0}^{\ell}\ttd_{is},
\]
where $\ttd_{i\ell}=|\mathcal{C}_{i\ell}|$ and $\bar\ttd_{i(-1)}:=0$. For any $j,k\in\mathcal{C}_{i\ell}$, $\upsilon_{j,\subi}\asymp_\P\upsilon_{k,\subi}$; for any $j\in\mathcal{C}_{i\ell}$ and $k\in\mathcal{C}_{i(\ell+1)}$, $\upsilon_{k,\subi}/\upsilon_{j,\subi}=o_\P(1)$. For later eigengap formulas, use the boundary conventions $\upsilon_{0,\subi}^2=\widehat\upsilon_{0,\subi}^2=+\infty$ and $\upsilon_{\ttd_i+1,\subi}^2=0$.

\medskip
\noindent\textit{Submatrices.}
For any local matrix $\bG_{\subi}$ with $\ttd_i$ columns and any set $\mathcal{C}\subseteq[\ttd_i]$, $\bG_{\cdot\mathcal{C},\subi}$ denotes the submatrix with column indices in $\mathcal{C}$. In particular, $\bG_{\gr{\ell},\subi}:=\bG_{\cdot\mathcal{C}_{i\ell},\subi}$, and $\bG_{\gr{0},\subi}$ contains the first $\ttd_{i0}$ columns of $\bG_{\subi}$.

\begin{lem}[Operator Norm of Errors] \label{lem: operator norm of eps}
	Under Assumption \ref{Assumption: Regularities}, if $\frac{n^{\frac{2}{\nu}}\log n}{p}\lesssim1$, then
	\[
	\max_{1\leq i\leq n}\|\bU_{\subi}\|\lesssim_\P\sqrt{K}+\sqrt{p\log (n\vee p)}.
	\]
\end{lem}

\begin{lem} \label{lem: eigenstructure 1st order}
	Under Assumptions \ref{Assumption: Regularities} and \ref{Assumption: LPCA}, if $\frac{n^{\frac{2}{\nu}}\log n}{p}\lesssim1$,
	then
	\begin{enumerate}[label=(\roman*)]
		\item $\max_{i\in[n]}\max_{j\in\mathcal{C}_{i0}}|\widehat{\upsilon}_{j,\subi}^2/\upsilon_{j,\subi}^2-1|=o_\P(1)$;
		\item There exist matrices $(\tilde{\bR}_{\subi}:i\in[n])$ such that
		$\max_{i\in[n]}\upsilon_{\bar\ttd_{i0},\subi}(p^\ddagger)^{-1/2}
		\|\widehat{\bF}_{\gr{0},\subi}-\bF_{\subi}\tilde{\bR}_{\gr{0},\subi}\|_{\mathtt{F}}\lesssim_\P\delta_{Kp}^{-1}$;
		\item $\max_{i\in[n]}\upsilon_{\bar\ttd_{i0},\subi}
		\|\bM_{\widehat{\bF}_{\gr{0},\subi}}-\bM_{\bF_{\gr{0},\subi}}\|_{\mathtt{F}}\lesssim_\P\delta_{Kp}^{-1}$.
	\end{enumerate}
\end{lem}

\begin{lem}\label{lem: consistency eigenstructure}
	Under Assumptions \ref{Assumption: Regularities} and \ref{Assumption: LPCA}, if $\frac{n^{\frac{2}{\nu}}\log n}{p}\lesssim1$, then
	\begin{enumerate}[label=(\roman*)]
		\item $\max_{i\in[n]}\max_{j\in[\ttd_i]}|\widehat{\upsilon}_{j,\subi}^2/\upsilon_{j,\subi}^2-1|=o_\P(1)$;
		\item There exist orthogonal matrices $(\check{\bR}_{\subi}:i\in[n])$ such that
		\[
		\max_{i\in[n]}\max_{j\in[\ttd_i]}\;\frac{\upsilon_{j,\subi}}{\sqrt{p^\ddagger}}
		\|\widehat{\bF}_{\cdot j,\subi}-\bF_{\subi}\check{\bR}_{\cdot j,\subi}\|\lesssim_\P\delta_{Kp}^{-1};
		\]
		\item $\max_{i\in[n]}\;\upsilon_{\ttd_i,\subi}\|\bM_{\widehat{\bF}_{\subi}}-\bM_{\bF_{\subi}}\|_{\mathtt F}\lesssim_\P\delta_{Kp}^{-1}$.
	\end{enumerate}
\end{lem}

Since the estimated factor loadings are scaled eigenvectors of $\frac{1}{p^\ddagger K}\bX_{\subi}\bX_{\subi}'$, the analysis of $\widehat{\bF}_{\subi}$ may also be applied to $\widehat{\bLambda}_{\subi}$ with a proper rescaling.

The above result shows the convergence of the estimated factors in the mean squared error sense. It is neither pointwise nor uniform.
In  Lemma \ref{lem: eigenvector bound} below we establish sup-norm bounds on the estimated singular vectors and projected noises, which is the key building block of the uniform convergence results in Theorem \ref{thm: uniform convergence of factors}.

\begin{lem}\label{lem: eigenvector bound}
	Under Assumptions \ref{Assumption: Regularities}, \ref{Assumption: Matching} and \ref{Assumption: LPCA}, if
	$(np)^{\frac{2}{\nu}}\delta_{Kp}^{-2}\lesssim 1$,
	then
	\begin{enumerate}[label=(\roman*)]
		\item $\max_{i\in[n]}\|\widehat{\bF}_{\subi}\|_{\max}\lesssim_\P1$;
		\item $\max_{i\in[n]}\max_{j\in[\ttd_i]}\upsilon_{j,\subi}^{-1}
		\|\widehat{\bLambda}_{\cdot j,\subi}\|_{\max}\lesssim_\P1$;
		\item $\max_{i\in[n]}\|\frac{1}{p^\ddagger}\bU_{\subi}\widehat{\bF}_{\subi}\|_{\max}
		\lesssim_\P\delta_{Kp}^{-1}$;
		\item $\max_{i\in[n]}\|K^{-1}\bU_{\subi}'\widehat{\bLambda}_{\subi}
		\widehat{\bUpsilon}_{\subi}^{-1}\|_{\max}\lesssim_\P\delta_{Kp}^{-1}$.
	\end{enumerate}
\end{lem}

\subsubsection{Proofs of Main Results in the Paper}\label{sec: appendix proofs}
\paragraph{Proof of Theorem \ref{thm: uniform convergence of factors}}
Expand the estimator of factor loadings:
$$\widehat{\bLambda}_{\subi}-\bLambda_{\subi}\bR_{\subi}
=%
\frac{1}{p^\ddagger}\bXi_{\subi}\widehat{\bF}_{\subi}+
\frac{1}{p^\ddagger}\bU_{\subi}\widehat{\bF}_{\subi}=:I+II$$
for $\bR_{\subi}=\frac{1}{p^\ddagger}\bF_{\subi}'\widehat{\bF}_{\subi}$.
For $I$, by the definition of $\mathfrak r_{\eta,\subi}$ and the normalization $\widehat{\bF}_{\subi}'\widehat{\bF}_{\subi}/p^\ddagger=\bI_{\ttd_i}$, we have
\[
\max_{1\leq k\leq K}
\left\|\frac{1}{p^\ddagger}\sum_{l\in\mathcal{C}^\ddagger}
\widehat{\bF}_{l\cdot,\subi}\xi_{j_k(i)l}\right\|
\lesssim_\P \mathfrak r_{\eta,\subi}
\]
uniformly over $i\in[n]$.

For $II$, by the row-side projected-noise conclusion established in the preceding subsection, we have
\[
\left\|\frac{1}{p^\ddagger}\bU_{\subi}\widehat{\bF}_{\subi}\right\|_{\max}
\lesssim_\P\delta_{Kp}^{-1}
\]
uniformly over $i\in[n]$.
It follows that
\[
\left\|\widehat{\bLambda}_{\cdot j,\subi}-\bLambda_{\subi}\bR_{\cdot j,\subi}\right\|_{\max}
\lesssim_\P\delta_{Kp}^{-1}+\mathfrak r_{\eta,\subi}
\]
uniformly over $i\in[n]$ and $j\in[\ttd_i]$.

Before turning to the factor bound, note that
\[
\bR_{\subi}\bR_{\subi}'
=\frac{1}{p^\ddagger}\bF_{\subi}'\bP_{\widehat{\bF}_{\subi}}\bF_{\subi}.
\]
Because $\bF_{\subi}'\bF_{\subi}/p^\ddagger=\bI_{\ttd_i}$ by the noise-free PCA definition, the full retained-space projection consistency established in the preceding subsection implies
$\max_{i\in[n]}\|\bR_{\subi}\bR_{\subi}'-\bI_{\ttd_i}\|=o_\P(1)$. Thus, $\bR_{\subi}$ is uniformly invertible and has its smallest and largest singular values uniformly bounded away from zero and infinity.

Now, consider the estimated factors:
{\small
	\begin{equation}\label{SA-eq: decomp of factors}
		\begin{split}
			\widehat{\bF}_{\subi}\widehat{\bUpsilon}_{\subi}^2=\,
			&\bF_{\subi}(\bR_{\subi}')^{-1}\widehat{\bUpsilon}_{\subi}^2+\\
			&\frac{1}{K}\bF_{\subi}(\bR_{\subi}')^{-1}(\bLambda_{\subi}\bR_{\subi}-\widehat{\bLambda}_{\subi})'\widehat{\bLambda}_{\subi}+
			\frac{1}{K}\bXi_{\subi}'\widehat{\bLambda}_{\subi}+
			\frac{1}{K}\bU_{\subi}'\widehat{\bLambda}_{\subi}.
		\end{split}
	\end{equation}
}
\vspace{-1.2em}

Fix $j\in[\ttd_i]$. By the loading bound above, Assumption \ref{Assumption: LPCA}(a), and Lemma \ref{lem: eigenvector bound}, the $j$th columns of the second and third terms on the right-hand side of \eqref{SA-eq: decomp of factors} are uniformly
\[
O_\P\big((\delta_{Kp}^{-1}+\mathfrak r_{\eta,\subi})\upsilon_{j,\subi}\big)
\quad\text{and}\quad
O_\P\big(\mathfrak r_{\eta,\subi}\upsilon_{j,\subi}\big),
\]
respectively. The symmetric, rotation-invariant projected-noise conclusion established in the preceding subsection is
\[
\left\|\frac{1}{K}\bU_{\subi}'\widehat{\bLambda}_{\subi}\widehat{\bUpsilon}_{\subi}^{-1}\right\|_{\max}
\lesssim_\P\delta_{Kp}^{-1}
\]
uniformly over $i$. By eigenvalue consistency, $\|K^{-1}\bU_{\subi}'\widehat{\bLambda}_{\cdot j,\subi}\|_{\max}\lesssim_\P\delta_{Kp}^{-1}\upsilon_{j,\subi}$ uniformly over $i$ and $j$. Lemma \ref{lem: consistency eigenstructure} implies $\widehat\upsilon_{j,\subi}^2/\upsilon_{j,\subi}^2\to_\P1$ uniformly. Dividing the $j$th column of \eqref{SA-eq: decomp of factors} by $\widehat\upsilon_{j,\subi}^2$ proves
\[
\left\|\widehat{\bF}_{\cdot j,\subi}-\bF_{\subi}((\bR_{\subi}')^{-1})_{\cdot j}\right\|_{\max}
\lesssim_\P(\delta_{Kp}^{-1}+\mathfrak r_{\eta,\subi})\upsilon_{j,\subi}^{-1}
\]
uniformly over $i\in[n]$ and $j\in[\ttd_i]$.

\paragraph{Proof of Corollary \ref{coro: consistency of latent mean}}
Let $\widehat{\blambda}_{j_k(i)}$ and $\blambda_{j_k(i)}$ be the local factor loadings in $\widehat{\bLambda}_{\subi}$ and $\bLambda_{\subi}$ respectively corresponding to the $j_k(i)$-th unit. By construction of the estimator,
\begin{align*}
&\widehat{\bF}_{l\cdot,\subi}\widehat{\blambda}_{j_k(i)}-\eta_l(\balpha_{j_k(i)})\\
={}&(\widehat{\bF}_{l\cdot,\subi}-\bF_{l\cdot,\subi}(\bR_{\subi}')^{-1})\widehat{\blambda}_{j_k(i)}
+\bF_{l\cdot,\subi}(\bR_{\subi}')^{-1}(\widehat{\blambda}_{j_k(i)}-\bR_{\subi}'\blambda_{j_k(i)})
-\xi_{j_k(i)l}.
\end{align*}
By Theorem \ref{thm: uniform convergence of factors}, Lemma \ref{lem: eigenvector bound}, and Assumption \ref{Assumption: LPCA}(a), both product terms are $O_\P(\delta_{Kp}^{-1}+\mathfrak r_{\eta,\subi})$ uniformly over $i$, $k$, and $l\in\mathcal{C}^\ddagger$. Here we also use the uniform boundedness of $\bR_{\subi}^{-1}$. Since $|\xi_{j_k(i)l}|\leq\mathfrak r_{\eta,\subi}$, the result follows.

\paragraph{Proof of Theorem \ref{thm: robustness}}
Because $\mathcal{C}^\dagger\subseteq[p_0]$, $\bX^\dagger=(\bY^{\mathsf{N}})^\dagger$, so the feasible and infeasible procedures form the same neighborhoods. Let $\bY^{\mathsf{N}}_{\subi}$ be the corresponding local submatrix, and let $\widehat{\bF}_{\subi}^\circ$ and $\widehat{\bLambda}^\circ_{\subi}$ be obtained by applying the PCA step to $\bY^{\mathsf{N}}_{\subi}$. By Corollary \ref{coro: consistency of latent mean},
\[
\|\widehat{\bLambda}_{\subi}^\circ(\widehat{\bF}_{\subi}^\circ)'-\bH_{\subi}\|_{\max}\lesssim_\P
\delta_{Kp}^{-1}+\mathfrak r_{\eta,\subi}
\]
uniformly over $i\in[n]$.
It remains to compare $\widehat{\bLambda}_{\subi}\widehat{\bF}_{\subi}'$ and $\widehat{\bLambda}_{\subi}^\circ(\widehat{\bF}_{\subi}^\circ)'$.

As in the proof of Lemma \ref{lem: eigenvector bound}, define
\[
\bG=\begin{bmatrix}\bm 0&\bX_{\subi}'\\\bX_{\subi}&\bm 0\end{bmatrix},
\qquad
\bG^\circ=\begin{bmatrix}\bm 0&(\bY^{\mathsf{N}}_{\subi})'\\\bY^{\mathsf{N}}_{\subi}&\bm 0\end{bmatrix}.
\]
Because $p-p_0$ is fixed, $\bX_{\subi}-\bY^{\mathsf{N}}_{\subi}$ has only a fixed number of nonzero entries, all in the row corresponding to the treated unit. By Assumption \ref{Assumption: Regularities}, $\max_{i\in[n]}\|\bG-\bG^\circ\|_\infty\lesssim_\P1$.
Note that
$\widehat{\bLambda}_{\subi}\widehat{\bF}_{\subi}'=
\sum_{\ell=0}^{g_i-1}\widehat{\bLambda}_{\gr{\ell},\subi}\widehat{\bF}_{\gr{\ell},\subi}'$
and $\widehat{\bLambda}_{\subi}^\circ(\widehat{\bF}_{\subi}^\circ)'=
\sum_{\ell=0}^{g_i-1}\widehat{\bLambda}_{\gr{\ell},\subi}^\circ(\widehat{\bF}_{\gr{\ell},\subi}^\circ)'$. It suffices to consider a generic eigenspace $\ell\in\{0,1,\ldots,g_i-1\}$. Let $\bar{\bE}$ be its eigenvector matrix for $\bG$, let $\bar{\bE}^\circ$ be the analogous matrix for $\bG^\circ$, and let $\bar{\bE}^*$ be the common noise-free target. Define $\bar{\bR}=\bar{\bE}'\bar{\bE}^*$ and $\bar{\bR}^\circ=(\bar{\bE}^\circ)'\bar{\bE}^*$.

By Weyl's inequality and $\max_i\|\bG-\bG^\circ\|\lesssim_\P1=o_\P(\Delta_\ell^\circ)$, the feasible and infeasible eigengaps are asymptotically equivalent. Hence, by the Davis-Kahan $\sin\Theta$ theorem \citep{davis1970rotation}, we have
\begin{align*}
	&\|\bar{\bE}^\circ\bar{\bR}^\circ-\bar{\bE}\bar{\bR}\|\leq
	\|\bar{\bE}^\circ(\bar{\bE}^\circ)'-\bar{\bE}\bar{\bE}'\|\\
	\lesssim \;&\frac{\|(\bG-\bG^\circ)\bar{\bE}^\circ\|}{\Delta_\ell^\circ}\lesssim_\P
	\frac{\|\bG-\bG^\circ\|_\infty\|\bar{\bE}^\circ\|_{\max}}{\sqrt{p^\ddagger K}\upsilon_{\bar\ttd_{i\ell},\subi}}
	\lesssim_\P \frac{1}{\sqrt{p^\ddagger K}\delta_{Kp}\upsilon_{\bar\ttd_{i\ell},\subi}}.
\end{align*}
Here $\Delta^\circ_\ell$ is the eigen-gap defined in the same way as in the proof of Lemma \ref{lem: eigenvector bound}. Since $\bG-\bG^\circ$ has fixed sparsity, the eigenvector bound for the infeasible estimator proves the last inequality. The upper block of this display implies
$\|\widehat{\bF}_{\gr\ell,\subi}^\circ\bar{\bR}^\circ-\widehat{\bF}_{\gr\ell,\subi}\bar{\bR}\|\lesssim_\P\{\sqrt{K}\upsilon_{\bar\ttd_{i\ell},\subi}\delta_{Kp}\}^{-1}$.
For the lower block, Weyl's inequality controls the singular-value rescaling. More precisely, consider the identities
\[
\bar{\bSigma}\bar{\bR}-\bar{\bR}\bar{\bSigma}^*=\bar{\bE}'(\bG-\bG^*)\bar{\bE}^*,\qquad
\bar{\bSigma}^\circ\bar{\bR}^\circ-\bar{\bR}^\circ\bar{\bSigma}^*= (\bar{\bE}^\circ)'(\bG^\circ-\bG^*)\bar{\bE}^*
\]
The lower eigenvector blocks have uniform row bounds by Lemma \ref{lem: eigenvector bound}. Multiplying the identities by these blocks yields
\[
\|\widehat{\bLambda}_{\gr\ell,\subi}^\circ\bar{\bR}^\circ-\widehat{\bLambda}_{\gr\ell,\subi}\bar{\bR}\|_{2\rightarrow\infty}
\lesssim_\P\delta_{Kp}^{-1}+\mathfrak r_{\eta,\subi}.
\]
Moreover, by Lemma 2 of \cite{Abbe-et-al_2020_AoS}, we have
\[
\|\bar{\bR}-\mathtt{sgn}(\bar{\bR})\|\lesssim_\P\frac{\delta_{Kp}^{-1}+\mathfrak r_{\eta,\subi}}{\upsilon_{\bar\ttd_{i\ell},\subi}},\quad
\|\bar{\bR}^\circ-\mathtt{sgn}(\bar{\bR}^\circ)\|\lesssim_\P
\frac{\delta_{Kp}^{-1}+\mathfrak r_{\eta,\subi}}{\upsilon_{\bar\ttd_{i\ell},\subi}}.
\]
By these results, Lemma \ref{lem: eigenvector bound}, and the uniform boundedness of $\ttd_i$, we have
\begin{align*}
	&\|\widehat{\bLambda}_{\gr{\ell},\subi}\widehat{\bF}_{\gr{\ell},\subi}'- \widehat{\bLambda}_{\gr{\ell},\subi}^\circ(\widehat{\bF}_{\gr{\ell},\subi}^\circ)'\|_{\max}\\
	=&	\|\widehat{\bLambda}_{\gr{\ell},\subi}\mathtt{sgn}(\bar{\bR})(\widehat{\bF}_{\gr{\ell},\subi}\mathtt{sgn}(\bar{\bR}))'- \widehat{\bLambda}_{\gr{\ell},\subi}^\circ\mathtt{sgn}(\bar{\bR}^\circ)(\widehat{\bF}_{\gr{\ell},\subi}^\circ\mathtt{sgn}(\bar{\bR}^\circ))'\|_{\max}\\
	\lesssim_\P&\;\delta_{Kp}^{-1}+\mathfrak r_{\eta,\subi}.
\end{align*}
The bound holds uniformly over $i$. Since $g_i$ is uniformly bounded, the desired result follows.

\subsubsection{Proofs of Supporting Lemmas}
\paragraph{Proof of Lemma \ref{lem: operator norm of eps}}

	Note $\bU_{\subi}$ is a $K\times p^\ddagger$ random matrix. Let $\mathscr{M}$ be the sigma field generated by $\mathscr{F}$ and $\bX^\dagger$. Conditional on $\mathscr{M}$, the neighborhoods are fixed, while the entries of $\bU^\ddagger$ remain independent and centered by Assumption \ref{Assumption: Regularities} and the feature-wise splitting. Thus, $\bU_{\subi}$ has independent rows, $\E[\bU_{\subi}|\mathscr{M}]=\bm{0}$, and $\dot\bU_{\subi}:=\bU_{\subi}$.
	By Assumption \ref{Assumption: Regularities}, $\E[(\dot\bU_{k\cdot, \subi})'\dot\bU_{k\cdot, \subi}|\mathscr{M}]=\bSigma_{k,\subi}$ for all $1\leq k\leq K$ and $1\leq i\leq n$. Moreover, $\underset{1\leq i\leq n}{\max}\|\bSigma_{\subi}\|\lesssim_\P 1$ for $\bSigma_{\subi}:=\frac{1}{K}\sum_{k=1}^{K}\bSigma_{k,\subi}$.
	Write
	$\frac{1}{K}\sum_{k=1}^{K}\Big(\dot\bU_{k\cdot,\subi}'
	\dot\bU_{k\cdot, \subi}-\bSigma_{k,\subi}\Big)=
	\frac{1}{K}\sum_{k=1}^{K}(\bG_k+\bT_k)$,
	where
	\begin{align*}
		\bG_k&:=\dot\bU_{k\cdot,\subi}'\dot\bU_{k\cdot,\subi}\I(\|\dot\bU_{k\cdot,\subi}\|^2\leq Cp)-\E[\dot\bU_{k\cdot,\subi}'\dot\bU_{k\cdot,\subi}\I(\|\dot\bU_{k\cdot,\subi}\|^2\leq Cp)|\mathscr{M}],\\
		\bT_k&:=\dot\bU_{k\cdot,\subi}'\dot\bU_{k\cdot,\subi}\I(\|\dot\bU_{k\cdot,\subi}\|^2> Cp)-\E[\dot\bU_{k\cdot,\subi}'\dot\bU_{k\cdot,\subi}\I(\|\dot\bU_{k\cdot,\subi}\|^2> Cp)|\mathscr{M}]
	\end{align*}
	for some $C>0$.
	For the truncated part $\bG_k$, by Bernstein's inequality for matrices and union bounds, the following holds for all $\tau>0$:
	\[
	\P\Big(\max_{1\leq i\leq n}\Big\|\frac{1}{K}\sum_{k=1}^{K}\bG_k\Big\|>\tau\Big|\mathscr{M}\Big)\leq
	np\exp\Big(-\frac{\tau^2/2}{Cp/K+Cp\tau/3K}\Big),
	\]
	which suffices to show that $\underset{1\leq i\leq n}{\max}\|\frac{1}{K}\sum_{k=1}^{K}\bG_k\|\lesssim_\P\max\{\sqrt{\log(n\vee p)}(p/K)^{1/2},\log(n\vee p)p/K\}$.
	For the tails, the conditional expectation in $\bT_k$ is positive semidefinite and bounded above by $\bSigma_{k,\subi}$. Hence, on the event $\max_{i\in[n]}\|\bU^\ddagger_{i\cdot}\|^2\leq Cp$,
	$\max_{i\in[n]}\|K^{-1}\sum_{k=1}^K\bT_k\|\lesssim_\P1$. We show below that the probability of the complementary event can be made arbitrarily small by choosing a sufficiently large fixed $C$. It follows that
	\[
	\max_{1\leq i\leq n}\left\|\frac{1}{K}\bU_{\subi}'\bU_{\subi}-\bSigma_{\subi}\right\|\lesssim_\P
	1+\max\{\sqrt{\log(n\vee p)}(p/K)^{1/2},\log(n\vee p)p/K\}.
	\]
	Since $\max_i\|\bSigma_{\subi}\|\lesssim_\P1$, the preceding display implies
	$\underset{1\leq i\leq n}{\max}\|\bU_{\subi}\|\lesssim_\P \sqrt{K}+\sqrt{p\log(n\vee p)}$.

	It remains to control the row norms. Let $\zeta_{il}:=u_{il}^2-\sigma_{il}^2$ and $\sigma_{il}^2=\E[u_{il}^2|\mathscr{M}]$. For every fixed $\varepsilon>0$, by the conditional Fuk--Nagaev inequality and a union bound,
	\[
	\P\left(\max_{i\in[n]}\left|\frac{1}{p^\ddagger}\sum_{l\in\mathcal{C}^\ddagger}\zeta_{il}\right|>\varepsilon\,\middle|\,\mathscr{M}\right)
	\lesssim n(p^\ddagger)^{-\nu/2}+n\exp(-cp^\ddagger)=o(1).
	\]
	Here the last equality follows from $p^\ddagger\asymp p$ and the rate restriction in the lemma. Moreover, by the conditional second-moment bound, $\max_{i\in[n]}(p^\ddagger)^{-1}\sum_{l\in\mathcal{C}^\ddagger}\sigma_{il}^2\lesssim1$. Therefore, $\max_{i\in[n]}(p^\ddagger)^{-1}\|\bU_{i\cdot}^\ddagger\|^2\lesssim_\P1$, which proves the required tail control and completes the proof.

\paragraph{Proof of Lemma \ref{lem: eigenstructure 1st order}}
	By Lemma \ref{lem: operator norm of eps} and $p^\ddagger\asymp p$, we have
	\[
	\max_{i\in[n]}\frac{\|\bU_{\subi}\|}{\sqrt{p^\ddagger K}}\lesssim_\P\delta_{Kp}^{-1}.
	\]
	For later use, the local covariance matrix can also be expanded as
	\[
	\begin{split}
	\bX_{\subi}'\bX_{\subi}
	=\,&\bF_{\subi}\bLambda_{\subi}'\bLambda_{\subi}\bF_{\subi}'
	+(\bXi_{\subi}+\bU_{\subi})'(\bXi_{\subi}+\bU_{\subi})\\
	&+\bF_{\subi}\bLambda_{\subi}'(\bXi_{\subi}+\bU_{\subi})
	+(\bXi_{\subi}+\bU_{\subi})'\bLambda_{\subi}\bF_{\subi}'
	=: \bG_1+\bG_2+\bG_3+\bG_4.
	\end{split}
	\]
	Since $\|\bH_{\subi}\|=\sqrt{p^\ddagger K}\upsilon_{1,\subi}$, uniformly over $i$,
	\begin{align*}
	&\left\|\frac{1}{p^\ddagger K}\left(\bX_{\subi}'\bX_{\subi}-\bH_{\subi}'\bH_{\subi}\right)\right\|
	\leq\frac{2\|\bH_{\subi}\|\|\bU_{\subi}\|+\|\bU_{\subi}\|^2}{p^\ddagger K}
	\lesssim_\P\upsilon_{1,\subi}\delta_{Kp}^{-1}+\delta_{Kp}^{-2}.
	\end{align*}
	For the strongest group, $\upsilon_{1,\subi}\asymp_\P\upsilon_{\bar\ttd_{i0},\subi}$. By Assumption \ref{Assumption: LPCA}(c), $\delta_{Kp}^{-1}/\upsilon_{\bar\ttd_{i0},\subi}=o_\P(1)$.

	The eigenvalue following the strongest group is $o_\P(\upsilon_{\bar\ttd_{i0},\subi}^2)$. Indeed, this follows from the definition of the signal-strength groups when $g_i>1$. When $g_i=1$, it follows from
	\[
	\frac{\|\bXi_{\subi}\|^2}{p^\ddagger K}\leq\mathfrak r_{\eta,\subi}^2=o_\P(\upsilon_{\bar\ttd_{i0},\subi}^2).
	\]
	Thus, the noise-free eigengap after $\mathcal{C}_{i0}$ is of order $\upsilon_{\bar\ttd_{i0},\subi}^2$. By Weyl's inequality and the preceding perturbation bound, part (i) follows. By the Davis--Kahan $\sin\Theta$ theorem \citep{davis1970rotation} and the boundedness of $\ttd_i$,
	\begin{align*}
	&\max_{i\in[n]}\;\upsilon_{\bar\ttd_{i0},\subi}
	\left\|\bM_{\widehat{\bF}_{\gr{0},\subi}}-\bM_{\bF_{\gr{0},\subi}}\right\|_{\mathtt F}
	\lesssim_\P\max_{i\in[n]}\left\{
	\frac{\upsilon_{1,\subi}}{\upsilon_{\bar\ttd_{i0},\subi}}\delta_{Kp}^{-1}
	+\frac{\delta_{Kp}^{-2}}{\upsilon_{\bar\ttd_{i0},\subi}}\right\}
	\lesssim_\P\delta_{Kp}^{-1}.
	\end{align*}
	The last inequality follows because $\upsilon_{1,\subi}\asymp_\P\upsilon_{\bar\ttd_{i0},\subi}$ and $\delta_{Kp}^{-1}/\upsilon_{\bar\ttd_{i0},\subi}=o_\P(1)$ uniformly over $i$. This proves part (iii).

	For part (ii), define
	\[
	\tilde{\bR}_{\subi}:=\bUpsilon_{\subi}^2\frac{\bF_{\subi}'\widehat{\bF}_{\subi}}{p^\ddagger}\widehat{\bUpsilon}_{\subi}^{-2}.
	\]
	Using the sample eigenvector equation and the exact identities $\bXi_{\subi}\bF_{\subi}=\bm0$ and $\bXi_{\subi}'\bLambda_{\subi}=\bm0$, we obtain
	\begin{equation}\label{eq: decomposition}
	\widehat{\bF}_{\subi}-\bF_{\subi}\tilde{\bR}_{\subi}
	=\frac{1}{p^\ddagger K}\left\{\bXi_{\subi}'\bXi_{\subi}+\bH_{\subi}'\bU_{\subi}+\bU_{\subi}'\bH_{\subi}+\bU_{\subi}'\bU_{\subi}\right\}\widehat{\bF}_{\subi}\widehat{\bUpsilon}_{\subi}^{-2}.
	\end{equation}
	Moreover, $\bXi_{\subi}'\bXi_{\subi}\bF_{\subi}=\bm0$. Apply \eqref{eq: decomposition} to the columns in $\mathcal{C}_{i0}$ and use part (i). Uniformly over $i\in[n]$,
	\[
	\frac{\|\bXi_{\subi}\|^2}{p^\ddagger K\widehat\upsilon_{\bar\ttd_{i0},\subi}^2}
	\leq\frac{\mathfrak r_{\eta,\subi}^2}{\widehat\upsilon_{\bar\ttd_{i0},\subi}^2}=o_\P(1).
	\]
	For the two products between $\bH_{\subi}$ and $\bU_{\subi}$, we have
	\begin{align*}
	&\frac{\upsilon_{\bar\ttd_{i0},\subi}}{\sqrt{p^\ddagger}}
	\left\|\frac{\bH_{\subi}'\bU_{\subi}+\bU_{\subi}'\bH_{\subi}}{p^\ddagger K}
	\widehat{\bF}_{\gr{0},\subi}\widehat{\bUpsilon}_{\mathcal{C}_{i0}\mathcal{C}_{i0},\subi}^{-2}\right\|_{\mathtt F}\\
	&\qquad\leq\frac{\upsilon_{\bar\ttd_{i0},\subi}}{\sqrt{p^\ddagger}}
	\frac{2\|\bH_{\subi}\|\|\bU_{\subi}\|}{p^\ddagger K}
	\|\widehat{\bF}_{\gr{0},\subi}\|_{\mathtt F}\widehat\upsilon_{\bar\ttd_{i0},\subi}^{-2}
	\lesssim_\P\frac{\upsilon_{1,\subi}}{\upsilon_{\bar\ttd_{i0},\subi}}\delta_{Kp}^{-1}\lesssim_\P\delta_{Kp}^{-1}.
	\end{align*}
	For the product involving $\bU_{\subi}$ alone,
	\begin{align*}
	&\frac{\upsilon_{\bar\ttd_{i0},\subi}}{\sqrt{p^\ddagger}}
	\left\|\frac{\bU_{\subi}'\bU_{\subi}}{p^\ddagger K}
	\widehat{\bF}_{\gr{0},\subi}\widehat{\bUpsilon}_{\mathcal{C}_{i0}\mathcal{C}_{i0},\subi}^{-2}\right\|_{\mathtt F}\\
	&\qquad\leq\frac{\upsilon_{\bar\ttd_{i0},\subi}}{\sqrt{p^\ddagger}}
	\frac{\|\bU_{\subi}\|^2}{p^\ddagger K}
	\|\widehat{\bF}_{\gr{0},\subi}\|_{\mathtt F}\widehat\upsilon_{\bar\ttd_{i0},\subi}^{-2}
	\lesssim_\P\frac{\delta_{Kp}^{-2}}{\upsilon_{\bar\ttd_{i0},\subi}}=o_\P(\delta_{Kp}^{-1}).
	\end{align*}
	Using these two bounds and the preceding bound for $\bXi_{\subi}'\bXi_{\subi}$, we obtain
	\[
	\max_{i\in[n]}\frac{\upsilon_{\bar\ttd_{i0},\subi}}{\sqrt{p^\ddagger}}
	\|\widehat{\bF}_{\gr{0},\subi}-\bF_{\subi}\tilde{\bR}_{\gr{0},\subi}\|_{\mathtt F}
	\lesssim_\P\frac{\delta_{Kp}^{-1}}{1-o_\P(1)}\lesssim_\P\delta_{Kp}^{-1},
	\]
	which proves part (ii).

\paragraph{Proof of Lemma \ref{lem: consistency eigenstructure}}
	By Lemma \ref{lem: operator norm of eps} and $p^\ddagger\asymp p$, we have
	\[
	\max_{i\in[n]}\frac{\|\bU_{\subi}\|}{\sqrt{p^\ddagger K}}\lesssim_\P\delta_{Kp}^{-1}.
	\]

	\textit{Step 1: Eigenvalues}.
	The normalized singular values of $\bX_{\subi}$ and $\bH_{\subi}$ are $\widehat\upsilon_{j,\subi}$ and $\upsilon_{j,\subi}$, respectively. By Weyl's inequality, uniformly over $i\in[n]$ and $j\in[\ttd_i]$,
	\[
	|\widehat\upsilon_{j,\subi}-\upsilon_{j,\subi}|
	\leq\frac{\|\bU_{\subi}\|}{\sqrt{p^\ddagger K}}\lesssim_\P\delta_{Kp}^{-1}.
	\]
	Therefore,
	\begin{align*}
	\left|\frac{\widehat\upsilon_{j,\subi}^2}{\upsilon_{j,\subi}^2}-1\right|
	&\leq 2\frac{|\widehat\upsilon_{j,\subi}-\upsilon_{j,\subi}|}{\upsilon_{j,\subi}}
	+\frac{|\widehat\upsilon_{j,\subi}-\upsilon_{j,\subi}|^2}{\upsilon_{j,\subi}^2}
	\lesssim_\P\frac{\delta_{Kp}^{-1}}{\upsilon_{j,\subi}}
	+\frac{\delta_{Kp}^{-2}}{\upsilon_{j,\subi}^2}=o_\P(1),
	\end{align*}
	which proves part (i). By Weyl's inequality again,
	\[
	\widehat\upsilon_{\ttd_i+1,\subi}
	\leq\frac{\|\bXi_{\subi}\|}{\sqrt{p^\ddagger K}}
	+\frac{\|\bU_{\subi}\|}{\sqrt{p^\ddagger K}}
	\leq\mathfrak r_{\eta,\subi}+O_\P(\delta_{Kp}^{-1})
	=o_\P(\upsilon_{\ttd_i,\subi})
	\]
	uniformly over $i$. Thus, all sample eigengaps at the boundaries of the signal-strength groups have the same orders as their noise-free counterparts.

	\textit{Step 2: Eigenspaces}.
	The result for $\mathcal{C}_{i0}$ follows from part (iii) of Lemma \ref{lem: eigenstructure 1st order}. We proceed inductively. Fix $1\leq\ell\leq g_i-1$ and define
	\[
	\overline{\mathcal{C}}_{i(\ell-1)}:=\bigcup_{s=0}^{\ell-1}\mathcal{C}_{is},\qquad
	\widehat{\bM}_{\ell-1,\subi}:=\bM_{\widehat{\bF}_{\cdot\overline{\mathcal{C}}_{i(\ell-1)},\subi}},\qquad
	\bM_{\ell-1,\subi}:=\bM_{\bF_{\cdot\overline{\mathcal{C}}_{i(\ell-1)},\subi}}.
	\]
	Also, let
	\[
	\bH_{<\ell,\subi}:=\sum_{s=0}^{\ell-1}\bLambda_{\gr{s},\subi}\bF_{\gr{s},\subi}',
	\qquad
	\bH_{\geq\ell,\subi}:=\bH_{\subi}-\bH_{<\ell,\subi}.
	\]
	Suppose the claimed projection bound has been established for groups $0,\ldots,\ell-1$. Since these groups are mutually orthogonal,
	\begin{align*}
	\|\widehat{\bM}_{\ell-1,\subi}-\bM_{\ell-1,\subi}\|
	&\leq\sum_{s=0}^{\ell-1}
	\|\bP_{\widehat{\bF}_{\gr{s},\subi}}-\bP_{\bF_{\gr{s},\subi}}\|_{\mathtt F}
	\lesssim_\P\frac{\delta_{Kp}^{-1}}{\upsilon_{\bar\ttd_{i(\ell-1)},\subi}}.
	\end{align*}
	The part of the stronger signal that remains after the sample projection satisfies
	\begin{align*}
	\frac{\|\bH_{<\ell,\subi}\widehat{\bM}_{\ell-1,\subi}\|}{\sqrt{p^\ddagger K}}
	&\leq\sum_{s=0}^{\ell-1}\frac{\|\bLambda_{\gr{s},\subi}\|}{\sqrt K}
	\frac{\|\bF_{\gr{s},\subi}'\widehat{\bM}_{\ell-1,\subi}\|}{\sqrt{p^\ddagger}}\\
	&\lesssim_\P\sum_{s=0}^{\ell-1}\upsilon_{\bar\ttd_{is},\subi}
	\|\bP_{\widehat{\bF}_{\gr{s},\subi}}-\bP_{\bF_{\gr{s},\subi}}\|_{\mathtt F}
	\lesssim_\P\delta_{Kp}^{-1}.
	\end{align*}
	The noise-free PCA identities imply $\bH_{\geq\ell,\subi}\bM_{\ell-1,\subi}=\bH_{\geq\ell,\subi}$ and $\|\bH_{\geq\ell,\subi}\|/\sqrt{p^\ddagger K}\asymp_\P\upsilon_{\bar\ttd_{i\ell},\subi}$. Hence,
	\begin{align*}
	\frac{\|\bH_{\subi}\widehat{\bM}_{\ell-1,\subi}-\bH_{\geq\ell,\subi}\|}{\sqrt{p^\ddagger K}}
	&\leq\frac{\|\bH_{<\ell,\subi}\widehat{\bM}_{\ell-1,\subi}\|}{\sqrt{p^\ddagger K}}
	+\frac{\|\bH_{\geq\ell,\subi}\|}{\sqrt{p^\ddagger K}}
	\|\widehat{\bM}_{\ell-1,\subi}-\bM_{\ell-1,\subi}\|\\
	&\lesssim_\P\delta_{Kp}^{-1}
	+\frac{\upsilon_{\bar\ttd_{i\ell},\subi}}{\upsilon_{\bar\ttd_{i(\ell-1)},\subi}}\delta_{Kp}^{-1}
	\lesssim_\P\delta_{Kp}^{-1}.
	\end{align*}
	It follows that $\|\bH_{\subi}\widehat{\bM}_{\ell-1,\subi}\|/\sqrt{p^\ddagger K}\lesssim_\P\upsilon_{\bar\ttd_{i\ell},\subi}$. We can now compare the projected sample covariance with its noise-free target:
	\begin{align*}
	&\left\|\frac{\widehat{\bM}_{\ell-1,\subi}\bX_{\subi}'\bX_{\subi}\widehat{\bM}_{\ell-1,\subi}
	-\bH_{\geq\ell,\subi}'\bH_{\geq\ell,\subi}}{p^\ddagger K}\right\|\\
	&\quad\leq\frac{\|(\bH_{\subi}\widehat{\bM}_{\ell-1,\subi})'\bH_{\subi}\widehat{\bM}_{\ell-1,\subi}
	-\bH_{\geq\ell,\subi}'\bH_{\geq\ell,\subi}\|}{p^\ddagger K}
	+\frac{2\|\bH_{\subi}\widehat{\bM}_{\ell-1,\subi}\|\|\bU_{\subi}\|+\|\bU_{\subi}\|^2}{p^\ddagger K}\\
	&\quad\leq\frac{\{\|\bH_{\subi}\widehat{\bM}_{\ell-1,\subi}\|+\|\bH_{\geq\ell,\subi}\|\}
	\|\bH_{\subi}\widehat{\bM}_{\ell-1,\subi}-\bH_{\geq\ell,\subi}\|}{p^\ddagger K}\\
	&\qquad+\frac{2\|\bH_{\subi}\widehat{\bM}_{\ell-1,\subi}\|\|\bU_{\subi}\|+\|\bU_{\subi}\|^2}{p^\ddagger K}\\
	&\quad\lesssim_\P\upsilon_{\bar\ttd_{i\ell},\subi}\delta_{Kp}^{-1}+\delta_{Kp}^{-2}.
	\end{align*}
	The leading $\ttd_{i\ell}$ eigenspaces of the two matrices in the last display are $\widehat{\bF}_{\gr{\ell},\subi}$ and $\bF_{\gr{\ell},\subi}$. The noise-free eigengap is of order $\upsilon_{\bar\ttd_{i\ell},\subi}^2$; for the last group, this also uses $\|\bXi_{\subi}\|^2/(p^\ddagger K)\leq\mathfrak r_{\eta,\subi}^2=o_\P(\upsilon_{\ttd_i,\subi}^2)$. Therefore, by the Davis--Kahan $\sin\Theta$ theorem \citep{davis1970rotation},
	\[
	\|\bP_{\widehat{\bF}_{\gr{\ell},\subi}}-\bP_{\bF_{\gr{\ell},\subi}}\|_{\mathtt F}
	\lesssim_\P\frac{\delta_{Kp}^{-1}}{\upsilon_{\bar\ttd_{i\ell},\subi}}.
	\]
	This closes the induction.

	\textit{Step 3: Rotations}.
	For every signal-strength group, let $\check{\bR}_{\mathcal{C}_{i\ell}\mathcal{C}_{i\ell},\subi}$ be an orthogonal Procrustes rotation from $\bF_{\gr{\ell},\subi}$ to $\widehat{\bF}_{\gr{\ell},\subi}$, and set the off-diagonal blocks of $\check{\bR}_{\subi}$ equal to zero. By the Procrustes inequality \citep[equation (A5)]{Yu-Wang-Samworth_2015_Biometrika} and the identity between the projection distance and the sines of the principal angles,
	\begin{align*}
	\frac{1}{\sqrt{p^\ddagger}}
	\|\widehat{\bF}_{\gr{\ell},\subi}-\bF_{\gr{\ell},\subi}\check{\bR}_{\mathcal{C}_{i\ell}\mathcal{C}_{i\ell},\subi}\|_{\mathtt F}
	&\leq\|\bP_{\widehat{\bF}_{\gr{\ell},\subi}}-\bP_{\bF_{\gr{\ell},\subi}}\|_{\mathtt F}
	\lesssim_\P\frac{\delta_{Kp}^{-1}}{\upsilon_{\bar\ttd_{i\ell},\subi}}.
	\end{align*}
	Because strengths within each group are comparable, this proves part (ii). Finally, the number of groups is uniformly bounded and
	\begin{align*}
	\|\bM_{\widehat{\bF}_{\subi}}-\bM_{\bF_{\subi}}\|_{\mathtt F}
	&\leq\sum_{\ell=0}^{g_i-1}
	\|\bP_{\widehat{\bF}_{\gr{\ell},\subi}}-\bP_{\bF_{\gr{\ell},\subi}}\|_{\mathtt F}
	\lesssim_\P\frac{\delta_{Kp}^{-1}}{\upsilon_{\ttd_i,\subi}},
	\end{align*}
	which proves part (iii).

	\paragraph{Proof of Lemma \ref{lem: eigenvector bound}}
		Let $\mathscr{M}$ be the sigma field generated by $\mathscr{F}$ and $\bX^\dagger$. Conditional on $\mathscr{M}$, all neighborhoods and noise-free local matrices are fixed, while the entries of $\bU^\ddagger$ are independent and centered.

		\textit{Step 1: Conditional concentration.}
		We first record the conditional concentration bound used below. For $P_0\in\{K,p^\ddagger\}$, set $\tau_{P_0}=\sqrt{P_0/\log(n\vee p)}$. By the moment bound and the rate condition,
	\begin{align*}
		\P\left(\max_{a\in[n],\,l\in\mathcal{C}^\ddagger}|u_{al}|>\tau_{P_0}\,\middle|\,\mathscr{M}\right)
		&\lesssim np\left\{\frac{\log(n\vee p)}{P_0}\right\}^{(2+\nu)/2}\\
		&\lesssim np\left\{\frac{\log(n\vee p)}{K\wedge p^\ddagger}\right\}^{\nu/2}
		\frac{\log(n\vee p)}{P_0}=o(1).
	\end{align*}
	Here $p^\ddagger\asymp p$, and the rate condition also implies $\log(n\vee p)/P_0=o(1)$. The centering bias after truncation is at most $\sqrt{P_0}\tau_{P_0}^{-(1+\nu)}=o(\sqrt{\log(n\vee p)})$. By the truncation and conditional Bernstein argument used in Lemma \ref{SA lem: matching maximal inequality},
	\begin{equation}\label{SA-eq: conditional weighted concentration}
		\left\|\sum_{s=1}^{P_0}u_s\bA_{s\cdot}\right\|
		\lesssim_\P\sqrt{\log(n\vee p)}\left(1+\sqrt{P_0}\|\bA\|_{2\rightarrow\infty}\right)
	\end{equation}
		uniformly over the weighted sums used below, provided that $\|\bA\|_{\mathtt F}\lesssim1$ and $\bA$ is conditionally independent of the errors in the sum. There are at most $n(p^\ddagger+K)$ such sums, and no independence across the sums is needed. This is why overlapping neighborhoods require no additional rate condition.

		\textit{Step 2: Dilation and eigengaps.}
		Define
	\[
	\bG=\begin{bmatrix}\bm{0}&\bX_{\subi}'\\\bX_{\subi}&\bm{0}\end{bmatrix},
	\qquad
	\bG^*=\begin{bmatrix}\bm{0}&\bF_{\subi}\bLambda_{\subi}'\\
	\bLambda_{\subi}\bF_{\subi}'&\bm{0}\end{bmatrix}.
	\]
	The dependence on $i$ is suppressed. The positive and negative eigenvalues of $\bG$ associated with the retained components are the diagonal elements of $\sqrt{p^\ddagger K}\widehat{\bUpsilon}_{\subi}$ and $-\sqrt{p^\ddagger K}\widehat{\bUpsilon}_{\subi}$. Their corresponding eigenvectors are the columns of
	\[
	\frac{1}{\sqrt{2}}
	\begin{bmatrix}
		(p^\ddagger)^{-1/2}\widehat{\bF}_{\subi}&(p^\ddagger)^{-1/2}\widehat{\bF}_{\subi}\\
		K^{-1/2}\widehat{\bLambda}_{\subi}\widehat{\bUpsilon}_{\subi}^{-1}&-K^{-1/2}\widehat{\bLambda}_{\subi}\widehat{\bUpsilon}_{\subi}^{-1}
	\end{bmatrix}.
	\]

	Fix $0\leq\ell\leq g_i-1$ and let
	\[
	\bar{\bE}=\begin{bmatrix}\bar{\bE}_{\mathtt R}\\\bar{\bE}_{\mathtt L}\end{bmatrix}
	=\frac{1}{\sqrt{2}}\begin{bmatrix}
		(p^\ddagger)^{-1/2}\widehat{\bF}_{\gr{\ell},\subi}\\
		K^{-1/2}\widehat{\bLambda}_{\gr{\ell},\subi}
		\widehat{\bUpsilon}_{\mathcal{C}_{i\ell}\mathcal{C}_{i\ell},\subi}^{-1}
	\end{bmatrix},
	\qquad
	\bar{\bSigma}=\sqrt{p^\ddagger K}\widehat{\bUpsilon}_{\mathcal{C}_{i\ell}\mathcal{C}_{i\ell},\subi}.
	\]
	Define $\bar{\bE}^*$ and $\bar{\bSigma}^*$ analogously using $\bG^*$. The noise-free eigengap around this group is
	\[
	\Delta^*=\sqrt{p^\ddagger K}\left\{(\upsilon_{\bar\ttd_{i(\ell-1)},\subi}-\upsilon_{\bar\ttd_{i(\ell-1)}+1,\subi})
	\wedge(\upsilon_{\bar\ttd_{i\ell},\subi}-\upsilon_{\bar\ttd_{i\ell}+1,\subi})
	\wedge\upsilon_{\bar\ttd_{i\ell},\subi}\right\},
	\]
	where the first difference is omitted when $\ell=0$ and $\upsilon_{\ttd_i+1,\subi}=0$. By the definition of the signal-strength groups, $\Delta^*\asymp_\P\sqrt{p^\ddagger K}\upsilon_{\bar\ttd_{i\ell},\subi}$ uniformly over $i$ and $\ell$.

	By Assumption \ref{Assumption: LPCA}(a), we have
	\begin{equation}\label{SA-eq: target dilation rows}
		\|\bar{\bE}_{\mathtt R}^*\|_{2\rightarrow\infty}\lesssim_\P(p^\ddagger)^{-1/2},
		\qquad
		\|\bar{\bE}_{\mathtt L}^*\|_{2\rightarrow\infty}\lesssim_\P K^{-1/2}.
	\end{equation}
	Moreover, by Assumption \ref{Assumption: LPCA}(c) and Lemma \ref{lem: operator norm of eps}, we have
	\[
	\max_{i\in[n]}\frac{\|\bXi_{\subi}+\bU_{\subi}\|}
	{\sqrt{p^\ddagger K}\upsilon_{\ttd_i,\subi}}=o_\P(1).
	\]
		Thus, by Lemma \ref{lem: consistency eigenstructure} and Weyl's inequality, the sample gap is asymptotically equivalent to $\Delta^*$. The Davis--Kahan theorem \citep{davis1970rotation} further shows that $\bar{\bR}:=\bar{\bE}'\bar{\bE}^*$ is uniformly nonsingular with $\|\bar{\bR}^{-1}\|\lesssim_\P1$.

		\textit{Step 3: Leave-one-out eigenspaces.}
		For each $t\in[p^\ddagger+K]$, let $\bG^{(t)}$ be obtained from $\bG$ by setting its $t$th row and column equal to zero. Let $\bar{\bE}^{(t)}$ and $\bar{\bSigma}^{(t)}$ contain the eigenvectors and eigenvalues corresponding to the same group, and set $\bar{\bR}^{(t)}=(\bar{\bE}^{(t)})'\bar{\bE}^*$. Conditional on $\mathscr{M}$, $\bar{\bE}^{(t)}$ is independent of the PCA errors in the omitted row and column. By the boundedness of the entries of $\bH_{\subi}$ and Lemma \ref{lem: operator norm of eps}, we have
	\[
	\max_{t\in[p^\ddagger+K]}\frac{\|\bG_{t\cdot}\|}{\Delta^*}=o_\P(1)
	\]
	uniformly over $i$ and $\ell$. More explicitly, the left-hand side is bounded by
	\[
	\frac{K^{-1/2}+(p^\ddagger)^{-1/2}+\delta_{Kp}^{-1}}
	{\upsilon_{\bar\ttd_{i\ell},\subi}}=o_\P(1).
	\]
	Since strengths within the group are comparable, Lemmas 1 and 3 of \cite{Abbe-et-al_2020_AoS} therefore give
	\begin{align}
	\|(\bar{\bE}\bar{\bR})_{t\cdot}\|
	&\lesssim_\P\frac{\|\bG_{t\cdot}\bar{\bE}^*\|+
	\|\bG_{t\cdot}(\bar{\bE}\bar{\bR}-\bar{\bE}^*)\|}{\Delta^*},
	\label{SA-eq: lemma eigen bound}\\
	\|\bar{\bE}\bar{\bR}-\bar{\bE}^{(t)}\bar{\bR}^{(t)}\|
	&\lesssim_\P\|(\bar{\bE}\bar{\bR})_{t\cdot}\|.
		\label{SA-eq: leave one out eigenspace}
		\end{align}

		\textit{Step 4: Terms in the rowwise eigenvector bound.}
		We next bound the two terms in \eqref{SA-eq: lemma eigen bound}. Using the identities $\bXi_{\subi}'\bLambda_{\subi}=\bm0$ and $\bXi_{\subi}\bF_{\subi}=\bm0$, we obtain
	\[
	\bG_{t\cdot}\bar{\bE}^*=\bar{\bE}_{t\cdot}^*\bar{\bSigma}^*+
	\begin{cases}
		\bU_{\cdot t,\subi}'\bar{\bE}_{\mathtt L}^*,&t\in[p^\ddagger],\\
		\bU_{k\cdot,\subi}\bar{\bE}_{\mathtt R}^*,&t=p^\ddagger+k,
	\end{cases}
	\]
	By Equations \eqref{SA-eq: conditional weighted concentration} and \eqref{SA-eq: target dilation rows}, the two noise products are $O_\P(\sqrt{\log(n\vee p)})$ uniformly. Therefore,
	\begin{equation*}
	\max_{t\in[p^\ddagger]}\frac{\|\bG_{t\cdot}\bar{\bE}^*\|}{\Delta^*}\lesssim_\P(p^\ddagger)^{-1/2},
	\qquad
	\max_{k\in[K]}\frac{\|\bG_{p^\ddagger+k,\cdot}\bar{\bE}^*\|}{\Delta^*}\lesssim_\P K^{-1/2}.
	\end{equation*}

	Let $\bD^{(t)}=\bar{\bE}^{(t)}\bar{\bR}^{(t)}-\bar{\bE}^*$. By the preceding perturbation and row-norm bounds,
	\begin{align*}
	\|\bG^{(t)}-\bG^*\|
	&\leq\|\bG-\bG^*\|+2\|\bG_{t\cdot}\|\\
	&\lesssim_\P\sqrt{p^\ddagger K}(\delta_{Kp}^{-1}+\mathfrak r_{\eta,\subi})
	=o_\P\left(\sqrt{p^\ddagger K}\upsilon_{\bar\ttd_{i\ell},\subi}\right).
	\end{align*}
	To control the stronger signal components in $\bG^*_{t\cdot}\bD^{(t)}$, consider a stronger target eigenspace $\bE_s^*$ with eigenvalue matrix $\bSigma_s^*$. Using $\bG^*\bE_s^*=\bE_s^*\bSigma_s^*$ and $\bG^{(t)}\bar{\bE}^{(t)}=\bar{\bE}^{(t)}\bar{\bSigma}^{(t)}$, we obtain
	\[
	\bSigma_s^*(\bE_s^*)'\bar{\bE}^{(t)}-(\bE_s^*)'\bar{\bE}^{(t)}\bar{\bSigma}^{(t)}
	=-(\bE_s^*)'(\bG^{(t)}-\bG^*)\bar{\bE}^{(t)}.
	\]
	Both eigenvalue matrices are diagonal. Thus, for every pair of column indices $a$ and $b$,
	\[
	\left\{(\bSigma_s^*)_{aa}-(\bar{\bSigma}^{(t)})_{bb}\right\}
	\left[(\bE_s^*)'\bar{\bE}^{(t)}\right]_{ab}
	=-\left[(\bE_s^*)'(\bG^{(t)}-\bG^*)\bar{\bE}^{(t)}\right]_{ab}.
	\]
	For a stronger group, $|(\bSigma_s^*)_{aa}-(\bar{\bSigma}^{(t)})_{bb}|\asymp_\P|(\bSigma_s^*)_{aa}|$. Since the group dimensions are uniformly bounded, it follows directly that
	\[
	\|\bSigma_s^*(\bE_s^*)'\bar{\bE}^{(t)}\|\lesssim_\P
	\|\bG^{(t)}-\bG^*\|=o_\P\left(\sqrt{p^\ddagger K}\upsilon_{\bar\ttd_{i\ell},\subi}\right).
	\]
	The same argument applies to the negative counterpart of each stronger eigenspace. After expanding $\bG^*_{t\cdot}\bD^{(t)}$ in the target eigenspaces, \eqref{SA-eq: target dilation rows} bounds the stronger components by the first term in each line below. The current and weaker eigenvalues are at most of order $\sqrt{p^\ddagger K}\upsilon_{\bar\ttd_{i\ell},\subi}$. They account for the second term. Thus,
	\begin{align}
	\|\bG^*_{t\cdot}\bD^{(t)}\|
	&\lesssim_\P(p^\ddagger)^{-1/2}\|\bG^{(t)}-\bG^*\|
	+\sqrt K\upsilon_{\bar\ttd_{i\ell},\subi}
	\lesssim_\P\sqrt K\upsilon_{\bar\ttd_{i\ell},\subi},&&t\in[p^\ddagger],
	\label{SA-eq: projected signal row bound}\\
	\|\bG^*_{p^\ddagger+k,\cdot}\bD^{(p^\ddagger+k)}\|
	&\lesssim_\P K^{-1/2}\|\bG^{(p^\ddagger+k)}-\bG^*\|
	+\sqrt{p^\ddagger}\upsilon_{\bar\ttd_{i\ell},\subi}
	\lesssim_\P\sqrt{p^\ddagger}\upsilon_{\bar\ttd_{i\ell},\subi}.&&
	\label{SA-eq: projected signal row bound lower}
	\end{align}
	The contribution of $\bXi_{\subi}$ is bounded in the same two cases by $\sqrt K\mathfrak r_{\eta,\subi}$ and $\sqrt{p^\ddagger}\mathfrak r_{\eta,\subi}$, respectively, and is of smaller order under Assumption \ref{Assumption: LPCA}(c).

	The first term in
	\[
	\bG_{t\cdot}(\bar{\bE}\bar{\bR}-\bar{\bE}^*)=
	\bG_{t\cdot}(\bar{\bE}\bar{\bR}-\bar{\bE}^{(t)}\bar{\bR}^{(t)})+
	\bG_{t\cdot}\bD^{(t)}
	\]
	is $o_\P(\Delta^*\|(\bar{\bE}\bar{\bR})_{t\cdot}\|)$ by \eqref{SA-eq: leave one out eigenspace}. For the noise part of the second term, by conditional independence and \eqref{SA-eq: conditional weighted concentration},
	\begin{align*}
	\|\bU_{\cdot t,\subi}'\bD_{\mathtt L}^{(t)}\|
	&\lesssim_\P\sqrt{\log(n\vee p)}\left(1+\sqrt K\|\bD_{\mathtt L}^{(t)}\|_{2\rightarrow\infty}\right),\\
		\|\bU_{k\cdot,\subi}\bD_{\mathtt R}^{(p^\ddagger+k)}\|
		&\lesssim_\P\sqrt{\log(n\vee p)}\left(1+\sqrt{p^\ddagger}\|\bD_{\mathtt R}^{(p^\ddagger+k)}\|_{2\rightarrow\infty}\right).
		\end{align*}

		\textit{Step 5: Coupled rowwise bounds.}
		Write $e_{\mathtt R}=\|\bar{\bE}_{\mathtt R}\bar{\bR}\|_{2\rightarrow\infty}$ and $e_{\mathtt L}=\|\bar{\bE}_{\mathtt L}\bar{\bR}\|_{2\rightarrow\infty}$. By \eqref{SA-eq: target dilation rows} and \eqref{SA-eq: leave one out eigenspace},
	\[
	\|\bD_{\mathtt L}^{(t)}\|_{2\rightarrow\infty}\lesssim_\P e_{\mathtt R}+e_{\mathtt L}+K^{-1/2},
	\qquad
	\|\bD_{\mathtt R}^{(p^\ddagger+k)}\|_{2\rightarrow\infty}\lesssim_\P e_{\mathtt R}+e_{\mathtt L}+(p^\ddagger)^{-1/2}.
	\]
	Substituting these bounds and \eqref{SA-eq: projected signal row bound}--\eqref{SA-eq: projected signal row bound lower} into \eqref{SA-eq: lemma eigen bound}, we obtain
	\begin{align*}
	 e_{\mathtt R}&\lesssim_\P(p^\ddagger)^{-1/2}+
	 \frac{\sqrt{\log(n\vee p)}}{\sqrt{p^\ddagger}\upsilon_{\bar\ttd_{i\ell},\subi}}(e_{\mathtt R}+e_{\mathtt L}),\\
	 e_{\mathtt L}&\lesssim_\P K^{-1/2}+
	 \frac{\sqrt{\log(n\vee p)}}{\sqrt K\upsilon_{\bar\ttd_{i\ell},\subi}}(e_{\mathtt R}+e_{\mathtt L}).
	\end{align*}
	Both displayed coefficients are $o_\P(1)$ under Assumption \ref{Assumption: LPCA}(c). Substituting the second inequality into the first and conversely yields
	\begin{align*}
	 e_{\mathtt R}&\lesssim_\P
	 \frac{(p^\ddagger)^{-1/2}+\frac{\sqrt{\log(n\vee p)}}{\sqrt{p^\ddagger K}\upsilon_{\bar\ttd_{i\ell},\subi}}}{1-o_\P(1)}
	 \lesssim_\P(p^\ddagger)^{-1/2},\\
	 e_{\mathtt L}&\lesssim_\P
	 \frac{K^{-1/2}+\frac{\sqrt{\log(n\vee p)}}{\sqrt{p^\ddagger K}\upsilon_{\bar\ttd_{i\ell},\subi}}}{1-o_\P(1)}
	 \lesssim_\P K^{-1/2}.
	\end{align*}
		These bounds are uniform over $i$ and $\ell$. Since $\bar{\bR}^{-1}$ is uniformly bounded, they also hold without $\bar{\bR}$. The first bound proves part (i). By the second bound and Lemma \ref{lem: consistency eigenstructure}(i), part (ii) follows.

		\textit{Step 6: Projected-noise bounds.}
		It remains to establish the two projected-noise bounds. All the following bounds are uniform over $i$ and $\ell$. For $t\in[p^\ddagger]$, use the same leave-one-out decomposition to write
	\[
	\bU_{\cdot t,\subi}'\bar{\bE}_{\mathtt L}\bar{\bR}
	=\bU_{\cdot t,\subi}'(\bar{\bE}_{\mathtt L}\bar{\bR}-\bar{\bE}_{\mathtt L}^{(t)}\bar{\bR}^{(t)})
	+\bU_{\cdot t,\subi}'\bD_{\mathtt L}^{(t)}+\bU_{\cdot t,\subi}'\bar{\bE}_{\mathtt L}^*.
	\]
	First apply Equations \eqref{SA-eq: conditional weighted concentration} and \eqref{SA-eq: leave one out eigenspace}. Then, using the bounds just established and the operator- and row-norm bounds in the proof of Lemma \ref{lem: operator norm of eps}, we obtain
	\begin{align*}
	\frac{1}{\sqrt K}\max_{t\in[p^\ddagger]}\|\bU_{\cdot t,\subi}'\bar{\bE}_{\mathtt L}\bar{\bR}\|
	&\lesssim_\P\frac{1}{\sqrt{p^\ddagger}}+
	\sqrt{\frac{\log(n\vee p)}{K}}+
	\sqrt{\frac{\log(n\vee p)}{p^\ddagger}}
	\lesssim\delta_{Kp}^{-1}.
	\end{align*}
	The symmetric argument for $t=p^\ddagger+k$ gives
	\[
	\frac{1}{\sqrt{p^\ddagger}}\max_{k\in[K]}\|\bU_{k\cdot,\subi}\bar{\bE}_{\mathtt R}\bar{\bR}\|
	\lesssim_\P\frac{1}{\sqrt K}+
	\sqrt{\frac{\log(n\vee p)}{p^\ddagger}}+
	\sqrt{\frac{\log(n\vee p)}{K}}
	\lesssim\delta_{Kp}^{-1}.
	\]
	Again, the boundedness of $\bar{\bR}^{-1}$ removes the rotations. Since
	\[
	\bar{\bE}_{\mathtt R}=\frac{\widehat{\bF}_{\gr{\ell},\subi}}{\sqrt{2p^\ddagger}},
	\qquad
	\bar{\bE}_{\mathtt L}=\frac{\widehat{\bLambda}_{\gr{\ell},\subi}
	\widehat{\bUpsilon}_{\mathcal{C}_{i\ell}\mathcal{C}_{i\ell},\subi}^{-1}}{\sqrt{2K}},
	\]
	the last two displays prove parts (iii) and (iv). Here we use that the number of signal-strength groups is uniformly bounded.

\subsubsection{Proof of Theorem \ref{thm: verify LPCA}}
\begin{proof}
	All bounds below hold uniformly over $i\in[n]$. Recall $h_n=(K/n)^{1/r}$, which is $o(1)$ because $K\log n/n=o(1)$. The condition $\delta_{Kp}^{-1}/h_n^{m-1}=o(1)$ also implies $\delta_{Kp}^{-1}=o(1)$. Hence $K/\log n\rightarrow\infty$ and $\log(n/K)/K=o(1)$. Therefore, we may apply Lemma \ref{lem: direct matching} and then Theorem \ref{thm: IMD}.

	\medskip
	\noindent\textit{Step 0: Monomial loadings and Taylor expansion.}
	Under the imposed smoothness conditions on latent functions,
	we can apply Taylor expansion and take the monomial basis as an ``intermediate'' loading object for verifying properties of the target noise-free PCA solutions.

	Specifically, let $\widetilde\blambda_{\subi}(\balpha)$ be the vector of monomials of degree no greater than $m-1$, including the constant, centered at $\balpha_i$, and let $\widetilde\bLambda_{\subi}$ collect its values over $\mathcal N_i$. Write
	\[
	\widetilde\bUpsilon_{\subi}=\diag(1,h_n\bI_{\ell_1},\ldots,h_n^{m-1}\bI_{\ell_{m-1}}),
	\]
	where $\ell_s$ is the number of monomials of degree $s$. By Theorem \ref{thm: IMD} and $a_p^{1/\underline\varsigma}=o(h_n)$,
	\[
	\max_{1\leq k\leq K}\|\balpha_{j_k(i)}-\balpha_i\|\lesssim_\P h_n.
	\]
	To obtain a lower design bound, define the neighborhood radius $\widehat R_i=\max_{j\in\mathcal N_i\setminus\{i\}}\rho(\bx_i^\dagger,\bx_j^\dagger)$. By the lower bound in Lemma \ref{lem: direct matching}, some selected unit has latent distance from $\balpha_i$ at least a fixed multiple of $h_n$.
	By Assumption \ref{Assumption: Matching} and $a_p^{1/\underline\varsigma}=o(h_n)$, $\widehat R_i-\rho_0\gtrsim_\P h_n^{\underline\varsigma}$. Conversely, for any $\epsilon>0$, a sufficiently small fixed $c>0$ ensures that, with probability at least $1-\epsilon$ for all sufficiently large $n$, every unit $j\neq i$ satisfying $\|\balpha_j-\balpha_i\|\leq ch_n$ has
	\[
	\rho(\bx_i^\dagger,\bx_j^\dagger)-\rho_0
	\lesssim c^{\underline\varsigma}h_n^{\underline\varsigma}+O_\P(a_p)
	<\widehat R_i-\rho_0
	\]
	and hence belongs to $\mathcal N_i$. Discarding the center's positive semidefinite contribution, we obtain
	\begin{align*}
	&s_{\min}\left(\frac{1}{K}\widetilde\bUpsilon_{\subi}^{-1}\widetilde\bLambda_{\subi}'
	\widetilde\bLambda_{\subi}\widetilde\bUpsilon_{\subi}^{-1}\right)\\
	&\qquad\geq s_{\min}\Bigg(\frac{1}{K}\sum_{j\in[n]\setminus\{i\}}
	\widetilde\bUpsilon_{\subi}^{-1}\widetilde\blambda_{\subi}(\balpha_j)
	\widetilde\blambda_{\subi}(\balpha_j)'\widetilde\bUpsilon_{\subi}^{-1}
	\I(\|\balpha_j-\balpha_i\|\leq ch_n)\Bigg).
	\end{align*}
	Conditional only on $\balpha_i$, the summand matrices for $j\neq i$ are independent. Let $\bq(\bz)$ collect the monomials of degree at most $m-1$, so that $\widetilde\bUpsilon_{\subi}^{-1}\widetilde\blambda_{\subi}(\balpha)=\bq((\balpha-\balpha_i)/h_n)$. The conditional expectation of each summand equals
	\[
	h_n^r\int_{\substack{\|\bz\|\leq c\\\balpha_i+h_n\bz\in\mathcal A}}
	\bq(\bz)\bq(\bz)'f_\alpha(\balpha_i+h_n\bz)\,\mathrm{d}\bz,
	\]
	where $f_\alpha$ denotes the latent density. The fixed compact convex support contains an interior ball, so the rescaled integration region contains a ball of uniformly positive radius, including at boundary centers. Since a nonzero polynomial cannot vanish on an open ball, the density lower bound and fixed monomial degree yield a uniform lower eigenvalue bound of order $h_n^r$.

	The summands are uniformly bounded with conditional second moments $O(h_n^r)$. Conditional Bernstein inequalities and a union bound therefore make the $K^{-1}$-normalized sum above differ from its conditional mean by $o_\P(1)$ uniformly over $i$, since $nh_n^r=K$ and $K/\log n\rightarrow\infty$. Its conditional mean has a minimum eigenvalue bounded below by a positive constant. Intersecting this concentration event with the containment event gives the required lower bound. The upper neighborhood bound establishes the corresponding upper eigenvalue and row bounds. Hence,
	\begin{equation}\label{SA-eq: auxiliary loading bounds}
	\begin{gathered}
	1\lesssim_\P s_{\min}\left(\frac{1}{K}\widetilde\bUpsilon_{\subi}^{-1}\widetilde\bLambda_{\subi}'
	\widetilde\bLambda_{\subi}\widetilde\bUpsilon_{\subi}^{-1}\right)
	\leq s_{\max}\left(\frac{1}{K}\widetilde\bUpsilon_{\subi}^{-1}\widetilde\bLambda_{\subi}'
	\widetilde\bLambda_{\subi}\widetilde\bUpsilon_{\subi}^{-1}\right)\lesssim_\P1,\\
	\|\widetilde\bLambda_{\subi}\widetilde\bUpsilon_{\subi}^{-1}\|_{\max}\lesssim_\P1.
	\end{gathered}
	\end{equation}

	Now, let the rows of $\widetilde\bF_{\subi}$ contain the derivative coefficients, including the factorial constants, corresponding to the monomials in $\widetilde\bLambda_{\subi}$. By the derivative nondegeneracy condition in the theorem and the uniform derivative bounds in Assumption \ref{Assumption: Primitive Smoothness},
	\begin{equation}\label{SA-eq: auxiliary factor bounds}
	1\lesssim_\P s_{\min}\left(\frac{1}{p^\ddagger}\widetilde\bF_{\subi}'\widetilde\bF_{\subi}\right)
	\leq s_{\max}\left(\frac{1}{p^\ddagger}\widetilde\bF_{\subi}'\widetilde\bF_{\subi}\right)\lesssim_\P1,
	\qquad \|\widetilde\bF_{\subi}\|_{\max}\lesssim_\P1.
	\end{equation}
	Taylor expansion through degree $m-1$ yields
	\begin{equation}\label{SA-eq: auxiliary Taylor representation}
	\bH_{\subi}=\widetilde\bH_{\subi}+\widetilde\bXi_{\subi},
	\qquad
	\widetilde\bH_{\subi}:=\widetilde\bLambda_{\subi}\widetilde\bF_{\subi}',
	\qquad
	\|\widetilde\bXi_{\subi}\|_{\max}\lesssim_\P h_n^m.
	\end{equation}

	\noindent\textit{Step 1: Signal strengths.}
	By \eqref{SA-eq: auxiliary loading bounds} and \eqref{SA-eq: auxiliary factor bounds}, the two outer matrices in
	\[
	\frac{\widetilde\bH_{\subi}}{\sqrt{Kp^\ddagger}}
	=\left(\frac{\widetilde\bLambda_{\subi}\widetilde\bUpsilon_{\subi}^{-1}}{\sqrt K}\right)
	\widetilde\bUpsilon_{\subi}
	\left(\frac{\widetilde\bF_{\subi}}{\sqrt{p^\ddagger}}\right)'
	\]
	have their nonzero singular values bounded and bounded away from zero. Thus, for every $j\in[\ttd_i]$,
	\[
	s_j(\widetilde\bUpsilon_{\subi})
	\lesssim_\P s_j\left(\frac{\widetilde\bH_{\subi}}{\sqrt{Kp^\ddagger}}\right)
	\lesssim_\P s_j(\widetilde\bUpsilon_{\subi})
	\]
	uniformly over $j$. Consequently, there are $\ell_s$ normalized singular values of order $h_n^s$ for $s=0,\ldots,m-1$. In particular, the smallest normalized singular value of $\widetilde\bH_{\subi}$ is of order $h_n^{m-1}$.

	Let
	\[
	\widetilde\bH_{\subi}=\widetilde\bE_{\mathtt L,\subi}\widetilde\bD_{\subi}
	\widetilde\bE_{\mathtt R,\subi}'
	\]
	be its compact singular-value decomposition. Its left and right singular spaces are the column spaces of $\widetilde\bLambda_{\subi}\widetilde\bUpsilon_{\subi}^{-1}$ and $\widetilde\bF_{\subi}$, respectively. For each $k\in[K]$, the projection matrix onto the first column space implies
	\begin{align*}
	\|\widetilde\bE_{\mathtt L,k\cdot,\subi}\|^2
	=(\widetilde\bLambda_{\subi}\widetilde\bUpsilon_{\subi}^{-1})_{k\cdot}
	\left\{(\widetilde\bLambda_{\subi}\widetilde\bUpsilon_{\subi}^{-1})'
	(\widetilde\bLambda_{\subi}\widetilde\bUpsilon_{\subi}^{-1})\right\}^{-1}
	(\widetilde\bLambda_{\subi}\widetilde\bUpsilon_{\subi}^{-1})_{k\cdot}'
	\leq\frac{\|(\widetilde\bLambda_{\subi}\widetilde\bUpsilon_{\subi}^{-1})_{k\cdot}\|^2}
	{s_{\min}^2(\widetilde\bLambda_{\subi}\widetilde\bUpsilon_{\subi}^{-1})}.
	\end{align*}
	Since $\ttd_i$ is uniformly bounded, \eqref{SA-eq: auxiliary loading bounds} therefore implies $\sqrt K\|\widetilde\bE_{\mathtt L,\subi}\|_{\max}\lesssim_\P1$. The same argument applied to $\widetilde\bF_{\subi}$ and \eqref{SA-eq: auxiliary factor bounds} implies $\sqrt{p^\ddagger}\|\widetilde\bE_{\mathtt R,\subi}\|_{\max}\lesssim_\P1$. Hence,
	\begin{equation}\label{SA-eq: auxiliary singular vector bounds}
	\sqrt K\|\widetilde\bE_{\mathtt L,\subi}\|_{\max}
	+\sqrt{p^\ddagger}\|\widetilde\bE_{\mathtt R,\subi}\|_{\max}\lesssim_\P1.
	\end{equation}
	Moreover, \eqref{SA-eq: auxiliary Taylor representation} implies
	\[
	\|\widetilde\bXi_{\subi}\|\leq\sqrt{Kp^\ddagger}\|\widetilde\bXi_{\subi}\|_{\max}
	\lesssim_\P\sqrt{Kp^\ddagger}h_n^m.
	\]
	By Weyl's inequality, the corresponding normalized singular values $\upsilon_{j,\subi}$ of $\bH_{\subi}$ retain the orders $1,h_n,\ldots,h_n^{m-1}$. Thus, $\upsilon_{\ttd_i,\subi}\asymp_\P h_n^{m-1}$, and adjacent strengths are either comparable or separated by a factor of order $h_n^{-1}$. Since $\delta_{Kp}^{-1}/h_n^{m-1}=o(1)$, we have $\delta_{Kp}^{-1}/\upsilon_{\ttd_i,\subi}=o_\P(1)$. This verifies part (b) and the sampling-noise condition in part (c).

	\medskip
	\noindent\textit{Step 2: Target factors and loadings.}
	Write the rank-$\ttd_i$ singular-value decomposition associated with the noise-free PCA definition as
	\[
	\bLambda_{\subi}\bF_{\subi}'=\bE_{\mathtt L,\subi}\bD_{\subi}\bE_{\mathtt R,\subi}',
	\quad
	\bE_{\mathtt L,\subi}=\frac{1}{\sqrt K}\bLambda_{\subi}\bUpsilon_{\subi}^{-1},
	\quad
	\bE_{\mathtt R,\subi}=\frac{1}{\sqrt{p^\ddagger}}\bF_{\subi},
	\quad
	\bD_{\subi}=\sqrt{Kp^\ddagger}\bUpsilon_{\subi}.
	\]
	Using the same symmetric-dilation argument as in the proof of Lemma \ref{lem: eigenvector bound}, the Davis--Kahan theorem \citep{davis1970rotation} applied to $\bH_{\subi}$ and $\widetilde\bH_{\subi}$ yields
	\begin{equation}\label{SA-eq: auxiliary subspace comparison}
	\|\bP_{\bE_{\mathtt L,\subi}}-\bP_{\widetilde\bE_{\mathtt L,\subi}}\|
	+\|\bP_{\bE_{\mathtt R,\subi}}-\bP_{\widetilde\bE_{\mathtt R,\subi}}\|
	\lesssim_\P h_n=o_\P(1).
	\end{equation}
	Since $\widetilde\bH_{\subi}$ has column and row spaces spanned by $\widetilde\bE_{\mathtt L,\subi}$ and $\widetilde\bE_{\mathtt R,\subi}$, the singular-vector equations imply
	\begin{align*}
	\bM_{\widetilde\bE_{\mathtt R,\subi}}\bE_{\mathtt R,\subi}
	&=\bM_{\widetilde\bE_{\mathtt R,\subi}}\widetilde\bXi_{\subi}'
	\bE_{\mathtt L,\subi}\bD_{\subi}^{-1},\\
	\bM_{\widetilde\bE_{\mathtt L,\subi}}\bE_{\mathtt L,\subi}
	&=\bM_{\widetilde\bE_{\mathtt L,\subi}}\widetilde\bXi_{\subi}
	\bE_{\mathtt R,\subi}\bD_{\subi}^{-1}.
	\end{align*}
	The rows of $\bP_{\widetilde\bE_{\mathtt R,\subi}}\bE_{\mathtt R,\subi}$ are of order $(p^\ddagger)^{-1/2}$ by \eqref{SA-eq: auxiliary singular vector bounds}, and the rows of $\bP_{\widetilde\bE_{\mathtt L,\subi}}\bE_{\mathtt L,\subi}$ are of order $K^{-1/2}$.
	Also, by \eqref{SA-eq: auxiliary singular vector bounds}, we have
	\begin{align*}
	\max_{l\in\mathcal C^\ddagger}
	\|\{\bM_{\widetilde\bE_{\mathtt R,\subi}}\widetilde\bXi_{\subi}'
	\bE_{\mathtt L,\subi}\}_{l\cdot}\|
	\leq \max_{l\in\mathcal C^\ddagger}\|\widetilde\bXi_{\cdot l,\subi}\|
	+\max_{l\in\mathcal C^\ddagger}\|\widetilde\bE_{\mathtt R,l\cdot,\subi}\|
	\|\widetilde\bXi_{\subi}'\bE_{\mathtt L,\subi}\|
	\lesssim_\P\sqrt K h_n^m.
	\end{align*}
	The symmetric calculation bounds the rows of $\bM_{\widetilde\bE_{\mathtt L,\subi}}\widetilde\bXi_{\subi}
	\bE_{\mathtt R,\subi}$ by $\sqrt{p^\ddagger}h_n^m$. Therefore, using $s_{\min}(\bD_{\subi})\asymp_\P\sqrt{Kp^\ddagger}h_n^{m-1}$ and \eqref{SA-eq: auxiliary singular vector bounds},
	\begin{align*}
	\max_{l\in\mathcal C^\ddagger}\|\bE_{\mathtt R,l\cdot,\subi}\|
	&\lesssim_\P\frac{1}{\sqrt{p^\ddagger}}+
	\frac{\sqrt K h_n^m}{\sqrt{Kp^\ddagger}h_n^{m-1}}
	\lesssim_\P\frac{1}{\sqrt{p^\ddagger}},\\
	\max_{1\leq k\leq K}\|\bE_{\mathtt L,k\cdot,\subi}\|
	&\lesssim_\P\frac{1}{\sqrt K}+
	\frac{\sqrt{p^\ddagger}h_n^m}{\sqrt{Kp^\ddagger}h_n^{m-1}}
	\lesssim_\P\frac{1}{\sqrt K}.
	\end{align*}
	It follows that $\|\bF_{\subi}\|_{\max}\lesssim_\P1$ and
	$\|\bLambda_{\subi}\bUpsilon_{\subi}^{-1}\|_{\max}\lesssim_\P1$. Also, $\ttd_i=\binom{m-1+r}{r}$ is fixed. This verifies part (a).

	\medskip
	\noindent\textit{Step 3: PCA remainder.}
	By the definition of the noise-free PCA remainder and \eqref{SA-eq: auxiliary Taylor representation},
	\[
	\bXi_{\subi}=\bH_{\subi}\bM_{\bE_{\mathtt R,\subi}}
	=\widetilde\bH_{\subi}\bM_{\bE_{\mathtt R,\subi}}
	+\widetilde\bXi_{\subi}\bM_{\bE_{\mathtt R,\subi}}.
	\]
	Thus, it suffices to bound the two objects
	$\|\widetilde\bH_{\subi}\bM_{\bE_{\mathtt R,\subi}}\|_{\max}$
	and
	$\|\widetilde\bXi_{\subi}\bM_{\bE_{\mathtt R,\subi}}\|_{\max}$.

	For the first object, define
	\[
	\bQ_{\subi}:=\bM_{\bE_{\mathtt R,\subi}}\widetilde\bE_{\mathtt R,\subi}\widetilde\bD_{\subi}.
	\]
	Since $\widetilde\bH_{\subi}'\widetilde\bE_{\mathtt L,\subi}
	=\widetilde\bE_{\mathtt R,\subi}\widetilde\bD_{\subi}$, we have
	\begin{align}
	\bQ_{\subi}
	&=\bM_{\bE_{\mathtt R,\subi}}\widetilde\bH_{\subi}'
	\widetilde\bE_{\mathtt L,\subi}
	=\bM_{\bE_{\mathtt R,\subi}}\widetilde\bH_{\subi}'
	(\bP_{\bE_{\mathtt L,\subi}}+\bM_{\bE_{\mathtt L,\subi}})
	\widetilde\bE_{\mathtt L,\subi}\notag\\
	&=-\bM_{\bE_{\mathtt R,\subi}}\widetilde\bXi_{\subi}'
	\bP_{\bE_{\mathtt L,\subi}}\widetilde\bE_{\mathtt L,\subi}
	+\bQ_{\subi}\widetilde\bE_{\mathtt L,\subi}'
	\bM_{\bE_{\mathtt L,\subi}}\widetilde\bE_{\mathtt L,\subi}.
	\label{SA-eq: auxiliary weighted subspace error}
	\end{align}
	Here the last equality uses
	$\bM_{\bE_{\mathtt R,\subi}}\bH_{\subi}'\bP_{\bE_{\mathtt L,\subi}}=\bm{0}$ and
	$\bM_{\bE_{\mathtt R,\subi}}\widetilde\bH_{\subi}'
	=\bQ_{\subi}\widetilde\bE_{\mathtt L,\subi}'$.

	We first show that the matrix multiplying $\bQ_{\subi}$ in the last term is small. Since $\bM_{\bE_{\mathtt L,\subi}}$ is an orthogonal projection,
	\begin{align*}
	\|\widetilde\bE_{\mathtt L,\subi}'\bM_{\bE_{\mathtt L,\subi}}
	\widetilde\bE_{\mathtt L,\subi}\|
	&=\|\bM_{\bE_{\mathtt L,\subi}}\widetilde\bE_{\mathtt L,\subi}\|^2\\
	&=\|(\bP_{\widetilde\bE_{\mathtt L,\subi}}-\bP_{\bE_{\mathtt L,\subi}})
	\widetilde\bE_{\mathtt L,\subi}\|^2\\
	&\leq\|\bP_{\widetilde\bE_{\mathtt L,\subi}}-\bP_{\bE_{\mathtt L,\subi}}\|^2
	\lesssim_\P h_n^2=o_\P(1),
	\end{align*}
	where the last bound follows from \eqref{SA-eq: auxiliary subspace comparison}.

	We next bound the first term on the right of \eqref{SA-eq: auxiliary weighted subspace error}. Before applying $\bM_{\bE_{\mathtt R,\subi}}$, the rowwise and operator-norm bounds are
	\begin{align*}
	\max_{l\in\mathcal C^\ddagger}
	\|\{\widetilde\bXi_{\subi}'\bP_{\bE_{\mathtt L,\subi}}
	\widetilde\bE_{\mathtt L,\subi}\}_{l\cdot}\|
	&\leq \max_{l\in\mathcal C^\ddagger}\|\widetilde\bXi_{\cdot l,\subi}\|
	\|\bP_{\bE_{\mathtt L,\subi}}\widetilde\bE_{\mathtt L,\subi}\|
	\lesssim_\P\sqrt K h_n^m,\\
	\|\widetilde\bXi_{\subi}'\bP_{\bE_{\mathtt L,\subi}}
	\widetilde\bE_{\mathtt L,\subi}\|
	&\leq\|\widetilde\bXi_{\subi}\|
	\lesssim_\P\sqrt{Kp^\ddagger}h_n^m.
	\end{align*}
	Here we used $\|\widetilde\bXi_{\subi}\|_{\max}\lesssim_\P h_n^m$ and the fact that all projection and singular-vector matrices in these products have operator norm no greater than one. Since
	$\bM_{\bE_{\mathtt R,\subi}}=\bI-\bE_{\mathtt R,\subi}\bE_{\mathtt R,\subi}'$, the right singular-vector row bound from Step 2 further gives
	\begin{align*}
	&\max_{l\in\mathcal C^\ddagger}
	\|\{\bM_{\bE_{\mathtt R,\subi}}\widetilde\bXi_{\subi}'
	\bP_{\bE_{\mathtt L,\subi}}\widetilde\bE_{\mathtt L,\subi}\}_{l\cdot}\|\\
	&\quad\leq
	\max_{l\in\mathcal C^\ddagger}
	\|\{\widetilde\bXi_{\subi}'\bP_{\bE_{\mathtt L,\subi}}
	\widetilde\bE_{\mathtt L,\subi}\}_{l\cdot}\|
	+\max_{l\in\mathcal C^\ddagger}\|\bE_{\mathtt R,l\cdot,\subi}\|
	\|\widetilde\bXi_{\subi}'\bP_{\bE_{\mathtt L,\subi}}
	\widetilde\bE_{\mathtt L,\subi}\|\\
	&\quad\lesssim_\P\sqrt K h_n^m.
	\end{align*}

	Rearranging \eqref{SA-eq: auxiliary weighted subspace error}, we obtain
	\[
	\bQ_{\subi}\left(\bI-\widetilde\bE_{\mathtt L,\subi}'
	\bM_{\bE_{\mathtt L,\subi}}\widetilde\bE_{\mathtt L,\subi}\right)
	=-\bM_{\bE_{\mathtt R,\subi}}\widetilde\bXi_{\subi}'
	\bP_{\bE_{\mathtt L,\subi}}\widetilde\bE_{\mathtt L,\subi}.
	\]
	The previous bound on $\widetilde\bE_{\mathtt L,\subi}'\bM_{\bE_{\mathtt L,\subi}}
	\widetilde\bE_{\mathtt L,\subi}$ shows that the matrix in parentheses has a uniformly bounded inverse w.p.a. 1. Therefore,
	\begin{equation}\label{SA-eq: auxiliary weighted subspace bound}
	\max_{l\in\mathcal C^\ddagger}\|\bQ_{l\cdot,\subi}\|\lesssim_\P\sqrt K h_n^m.
	\end{equation}
	Because $\widetilde\bH_{\subi}\bM_{\bE_{\mathtt R,\subi}}
	=\widetilde\bE_{\mathtt L,\subi}\bQ_{\subi}'$, \eqref{SA-eq: auxiliary singular vector bounds} and \eqref{SA-eq: auxiliary weighted subspace bound} now yield
	\begin{align*}
	\|\widetilde\bH_{\subi}\bM_{\bE_{\mathtt R,\subi}}\|_{\max}
	&\leq\max_{1\leq k\leq K}\|\widetilde\bE_{\mathtt L,k\cdot,\subi}\|
	\max_{l\in\mathcal C^\ddagger}\|\bQ_{l\cdot,\subi}\|
	\lesssim_\P K^{-1/2}\sqrt K h_n^m
	\lesssim_\P h_n^m.
	\end{align*}

	For the second object, write $\bM_{\bE_{\mathtt R,\subi}}=\bI-\bP_{\bE_{\mathtt R,\subi}}$. The right singular-vector row bound from Step 2 implies
	\begin{align*}
	\|\widetilde\bXi_{\subi}\bP_{\bE_{\mathtt R,\subi}}\|_{\max}
	&\leq\max_{1\leq k\leq K}\|\widetilde\bXi_{k\cdot,\subi}\bE_{\mathtt R,\subi}\|
	\max_{l\in\mathcal C^\ddagger}\|\bE_{\mathtt R,l\cdot,\subi}\|\\
	&\leq\max_{1\leq k\leq K}\|\widetilde\bXi_{k\cdot,\subi}\|
	\max_{l\in\mathcal C^\ddagger}\|\bE_{\mathtt R,l\cdot,\subi}\|\\
	&\lesssim_\P\sqrt{p^\ddagger}h_n^m(p^\ddagger)^{-1/2}
	\lesssim_\P h_n^m.
	\end{align*}
	Together with $\|\widetilde\bXi_{\subi}\|_{\max}\lesssim_\P h_n^m$, this gives
	$\|\widetilde\bXi_{\subi}\bM_{\bE_{\mathtt R,\subi}}\|_{\max}\lesssim_\P h_n^m$.

	Both objects in the initial decomposition are therefore $O_\P(h_n^m)$. Thus, $\mathfrak r_{\eta,\subi}=\|\bXi_{\subi}\|_{\max}\lesssim_\P h_n^m$, and
	$\mathfrak r_{\eta,\subi}/\upsilon_{\ttd_i,\subi}\lesssim_\P h_n=o_\P(1)$. This verifies the approximation-error condition in part (c) and completes the proof.
\end{proof}

\section{Extension: Covariates-Adjusted Nonlinear Factor Model} \label{SA-sec: main results}

In this section we provide formal analysis of the generalized nonlinear factor model described in Section \ref{subsec: extensions}:
\begin{equation}\label{SA-eq: general model}
\bx_i=\bW_i\bvth+\bmeta(\balpha_i)+\bu_i, \quad \E[\bu_i|(\balpha_j,\bW_j:j\in[n])]=\bm{0}
\end{equation}
where $\bW_i=(\bw_{i,1}, \cdots, \bw_{i,q})$ is a $p\times q$ matrix of additional covariates, where $\bw_{i,\ell}=(w_{i1,\ell}, \cdots, w_{ip,\ell})'$ for each $\ell\in[q]$. The unknown coefficient vector $\bvth\in\mathbb{R}^q$ is common to all units and features.

\subsection{Assumptions}

The following analysis is based on the main estimation procedure described in Section \ref{subsec: extensions} of the main paper.
The first assumption concerns the structure of the high-rank covariates in Equation \eqref{SA-eq: general model}.

\begin{assumption}[High-Rank Covariates]
	\label{SA-Assumption: high-rank covariates}
	For $\ell\in[q]$, each $\bw_{i,\ell}$ satisfies that
	\[
	w_{it,\ell}=\hbar_{t,\ell}(\bvrho_{i,\ell})+e_{it,\ell}, \quad
	\E[e_{it,\ell}|\mathcal{F}_\ell]=0,
	\]
	where $\bvrho_{i,\ell}\in\mathbb{R}^{r_{\ell}}$, $(\bvrho_{i,\ell}:i\in[n])$ is i.i.d., and $\mathcal{F}_\ell=\sigma(\bvrho_{i,\ell}:i\in[n])$. As in the main model, the latent functions $\hbar_{t,\ell}(\cdot)$ are treated as fixed.
	Moreover, for some constant $c_{\min}>0$,
	$$s_{\min}\Big(\frac{1}{n|\mathcal{C}^\ddagger|}\sum_{i=1}^{n}\sum_{t\in\mathcal{C}^\ddagger}
	\be_{it}\be_{it}'\Big)\geq c_{\min} \quad \text{w.p.a. } 1,$$
	where $\be_{it}=(e_{it,1}, \cdots, e_{it,q})'$.
\end{assumption}

The condition above is akin to that commonly used in partially linear regression. It ensures that $\bw_{i,\ell}$'s are truly ``high-rank'': after partialling out the ``low-rank'' signals (factor components), these covariates are still non-degenerate, thus achieving identifiability of $\bvth$.
Also note that Assumption \ref{SA-Assumption: high-rank covariates} is general in the sense that the latent variables and latent functions associated with each $\bw_{i,\ell}$ could be different across $\ell$. To ease notation, define the $\sigma$-field $\mathcal{F}=\sigma(\balpha_i,\bvrho_{i,1},\cdots,\bvrho_{i,q}:i\in[n])$. Let $\bar{\balpha}_i$ be the vector obtained by stacking $\balpha_i,\bvrho_{i,1},\cdots,\bvrho_{i,q}$ after removing duplicated components.

We impose some standard regularity conditions on the latent variables, latent functions and idiosyncratic errors, which are analogous
to Assumption \ref{Assumption: Regularities} in the main paper.
\begin{assumption}[Regularities]
	\label{SA-Assumption: regularities}\leavevmode
	Let $\nu>0$ be a constant.
	\begin{enumerate}[label=(\alph*)]
		\item For all $t\in[p]$ and $\ell\in[q]$, $\hbar_{t,\ell}(\cdot)$ and $\eta_t(\cdot)$ are bounded by a universal constant uniformly over their supports;

		\item $(\bar{\balpha}_i:i\in[n])$ is i.i.d. Moreover, $\balpha_i$, $\bvrho_{i,\ell}$ for each $\ell\in[q]$, and $\bar{\balpha}_i$ have compact convex supports with densities bounded and bounded away from zero on their respective supports;

		\item Conditional on $\mathcal{F}$, $((u_{it}, \be_{it}):i\in[n], t\in[p])$ is independent over $i$ and across $t$ with zero means, and $\E[u_{it}\be_{it}|\mathcal{F}]=\bm{0}$ for all $i\in[n]$ and $t\in[p]$.
		Also, $\E[|u_{it}|^{2+\nu}|\mathcal{F}]$ and $\E[\|\be_{it}\|^{2+\nu}|\mathcal{F}]$ are bounded almost surely by a fixed constant uniformly over $i\in[n]$ and $t\in[p]$.
	\end{enumerate}
\end{assumption}

Next, we impose conditions for indirect matching that are analogous to Assumption \ref{Assumption: Matching}.

To ease our exposition, for $\ell\in [q]$, we define the $n\times p$ matrices
$\bH$, $\bar{\bH}_{\lceil\ell\rfloor}$, $\bar{\bar{\bH}}$, $\bW_{\lceil \ell\rfloor}$ and $\tilde\bX$ whose $(i,t)$-th elements are respectively given by
\begin{align*}
&H_{it}=\eta_t(\balpha_i),\quad
\bar{H}_{it, \lceil\ell\rfloor}=\hbar_{t,\ell}(\bvrho_{i,\ell}),\quad
\bar{\bar{H}}_{it}=\eta_t(\balpha_i)+\sum_{\ell=1}^{q}\hbar_{t,\ell}(\bvrho_{i,\ell})\vartheta_\ell, \\
&W_{it,\lceil\ell\rfloor}=w_{it,\ell},\quad \tilde{X}_{it}=\eta_t(\balpha_i)+u_{it}+\{\bW_i(\bvth-\widehat{\bvth})\}_t.
\end{align*}
As in the main paper, let $p^\dagger=|\mathcal{C}^\dagger|$, $p^\ddagger=|\mathcal{C}^\ddagger|$ and $p^\wr=|\mathcal{C}^\wr|$, where $p^\dagger\asymp p^\ddagger\asymp p^\wr\asymp p$, and use the corresponding superscripts to denote each subset of features.
For example, $\bH^{\dagger}$, $\bH^{\ddagger}$ and $\bH^{\wr}$ are submatrices of $\bH$ with column indices in $\mathcal{C}^\dagger$, $\mathcal{C}^\ddagger$ and $\mathcal{C}^\wr$ respectively.

\begin{assumption}[Indirect Matching]
	\label{SA-Assumption: indirect matching}\leavevmode
		For some fixed constants $\rho_{0,\ell}^{\textsf W}$ ($\ell\in[q]$), $\rho_0^{\textsf X}$ and $\rho_0^{\widetilde{\textsf X}}$ in $\mathbb R$, fixed positive constants $\lvarsig$ and $\uvarsig$, and some positive sequence $a_p=o(1)$, the following conditions hold:
		\[
		\begin{split}
		&\underset{1\leq i,j\leq n,\;i\neq j}{\max}\;|\rho(\bW_{i\cdot, \lceil\ell\rfloor}^\dagger, \bW_{j\cdot, \lceil\ell\rfloor}^\dagger)-
		\rho(\bar{\bH}_{i\cdot, \lceil\ell\rfloor}^\dagger, \bar\bH_{j\cdot, \lceil\ell\rfloor}^\dagger)
		-\rho_{0,\ell}^{\textsf W}|\lesssim_\P a_p, \quad \ell\in[q],\\
		&\underset{1\leq i,j\leq n,\;i\neq j}{\max}\;|\rho(\bX_{i\cdot}^\dagger, \bX_{j\cdot}^\dagger)-
		\rho(\bar{\bar{\bH}}_{i\cdot}^\dagger, \bar{\bar{\bH}}_{j\cdot}^\dagger)
		-\rho_0^{\textsf X}|\lesssim_\P a_p,\\
		&\underset{1\leq i,j\leq n,\;i\neq j}{\max}\;|\rho(\tilde\bX_{i\cdot}^\ddagger, \tilde\bX_{j\cdot}^\ddagger)-
		\rho(\bH_{i\cdot}^\ddagger, \bH_{j\cdot}^\ddagger)
		-\rho_0^{\widetilde{\textsf X}}|\lesssim_\P a_p,
		\end{split}
		\]
		Associate $\bH^\ddagger$, $\bar{\bH}_{\lceil\ell\rfloor}^\dagger$, and $\bar{\bar{\bH}}^\dagger$ with the latent arguments $\balpha_i$, $\bvrho_{i,\ell}$, and $\bar{\balpha}_i$, respectively. For each such matrix $\bS$ and its latent argument $\bv_i$,
		$$\underset{1\leq i,j\leq n\atop \bv_i\neq \bv_j}{\max}\;\frac{\rho(\bS_{i\cdot},\, \bS_{j\cdot})}{\|\bv_i-\bv_j\|^{\uvarsig}}
		\lesssim_\P 1\quad \text{and}\quad
		\underset{1\leq i,j\leq n\atop \bv_i\neq \bv_j}{\min}\;\frac{\rho(\bS_{i\cdot},\, \bS_{j\cdot})}{\|\bv_i-\bv_j\|^{\lvarsig}}\gtrsim_\P 1.$$
\end{assumption}

Next, we impose conditions for local principal component analysis that are analogous to Assumption \ref{Assumption: LPCA} in the main paper. To ease the exposition, we introduce the notation $\bH_{\subi}$, $\bar\bH_{\subi,\lceil\ell\rfloor}$ and $\bar{\bar{\bH}}_{\subi}$ for the submatrices of $\bH^\wr$,
$\bar\bH_{\lceil\ell\rfloor}^{\ddagger}$ and $\bar{\bar{\bH}}^\ddagger$, respectively, whose rows correspond to the nearest neighbors of unit $i$ obtained in the estimation procedure.
For each local signal matrix $\bS_{\subi}$ equal to $\bH_{\subi}$, $\bar\bH_{\subi,\lceil\ell\rfloor}$ ($\ell\in[q]$), or $\bar{\bar{\bH}}_{\subi}$, let $p_{\textsf S}$ denote its number of columns and define its target local factors and loadings by the noise-free principal component problem
\begin{equation*}
	(\bF_{\subi}^{\textsf S},\bLambda_{\subi}^{\textsf S})
	:=\underset{\scalebox{0.6}{$\tilde\bF\in\mathbb R^{p_{\textsf S}\times\ttd_i^{\textsf S}},\ \tilde\bLambda\in\mathbb R^{K\times\ttd_i^{\textsf S}}$}}{\argmin}
	\tr\left[(\bS_{\subi}-\tilde\bLambda\tilde\bF')(\bS_{\subi}-\tilde\bLambda\tilde\bF')'\right]
\end{equation*}
subject to $(\bF_{\subi}^{\textsf S})'\bF_{\subi}^{\textsf S}/p_{\textsf S}=\bI_{\ttd_i^{\textsf S}}$ and $(\bLambda_{\subi}^{\textsf S})'\bLambda_{\subi}^{\textsf S}/K$ being diagonal, with its diagonal entries arranged in nonincreasing order. Write
\[
\bS_{\subi}=\bLambda_{\subi}^{\textsf S}(\bF_{\subi}^{\textsf S})'+\bXi_{\subi}^{\textsf S},
\qquad
\bUpsilon_{\subi}^{\textsf S}:=\left\{\frac{1}{K}(\bLambda_{\subi}^{\textsf S})'\bLambda_{\subi}^{\textsf S}\right\}^{1/2}
=\diag\{\upsilon_{1,\subi}^{\textsf S},\ldots,\upsilon_{\ttd_i^{\textsf S},\subi}^{\textsf S}\},
\]
where $\upsilon_{1,\subi}^{\textsf S}\geq\cdots\geq\upsilon_{\ttd_i^{\textsf S},\subi}^{\textsf S}>0$. Thus, $(\upsilon_{j,\subi}^{\textsf S})^2$ is the $j$th leading eigenvalue of $(\bS_{\subi})'\bS_{\subi}/(p_{\textsf S}K)$. Also let $\mathfrak r_{\textsf S,\subi}:=\|\bXi_{\subi}^{\textsf S}\|_{\max}$. When $\bS_{\subi}=\bH_{\subi}$, suppress the superscript $\textsf S$ and write $\bF_{\subi}$, $\bLambda_{\subi}$, $\bXi_{\subi}$, $\upsilon_{j,\subi}$, and $\mathfrak r_{\eta,\subi}$.

\begin{assumption}[Local Approximation]
	\label{SA-Assumption: local approximation}\leavevmode
	For each local signal matrix $\bS_{\subi}$ defined above, the following conditions hold:
	\begin{enumerate}[label=(\alph*)]
		\item $\max_{i\in[n]}\ttd_i^{\textsf S}\lesssim1$, $\max_{i\in[n]}\|\bLambda_{\subi}^{\textsf S}(\bUpsilon_{\subi}^{\textsf S})^{-1}\|_{\max}\lesssim_\P1$, and $\max_{i\in[n]}\|\bF_{\subi}^{\textsf S}\|_{\max}\lesssim_\P1$. Either $\upsilon_{j,\subi}^{\textsf S}/\upsilon_{j+1,\subi}^{\textsf S}\lesssim_\P1$ or $\upsilon_{j+1,\subi}^{\textsf S}/\upsilon_{j,\subi}^{\textsf S}=o_\P(1)$ holds for all $1\leq j<\ttd_i^{\textsf S}$ and $i\in[n]$, and $\max_{i\in[n]}\delta_{Kp}^{-1}/\upsilon_{\ttd_i^{\textsf S},\subi}^{\textsf S}=o_\P(1)$;

		\item $\max_{i\in[n]}\mathfrak r_{\textsf S,\subi}/\upsilon_{\ttd_i^{\textsf S},\subi}^{\textsf S}=o_\P(1)$.
	\end{enumerate}
\end{assumption}

\subsection{Main Results}
To begin with, the following lemma derives the rate of convergence for $\widehat{\bvth}$, which makes it possible to separate the factor component of interest from others in Equation \eqref{SA-eq: general model}.

Define the uniform first-stage approximation rate
\[
\mathfrak r_{\varrho,n}:=\max_{i\in[n]}\max_{\bS_{\subi}\in\{\bar\bH_{\subi,\lceil\ell\rfloor}:\ell\in[q]\}\cup\{\bar{\bar{\bH}}_{\subi}\}}\mathfrak r_{\textsf S,\subi}.
\]

\begin{lem} \label{SA-lem: slope of covariates}
	Under Assumptions \ref{SA-Assumption: high-rank covariates}--\ref{SA-Assumption: local approximation},
	 if $(np)^{\frac{2}{\nu}}\delta_{Kp}^{-2}\lesssim 1$, then
	$\|\widehat{\bvth}-\bvth\|\lesssim_\P \delta_{Kp}^{-1}+\mathfrak r_{\varrho,n}^{2}$.
\end{lem}

\begin{remark}
	This is a by-product of our main analysis. In fact, it shows that the low-dimensional parameter in partially linear regression with a nonlinear factor structure can be consistently estimated. The rate of convergence above may not be sharp, since the calculation is simply based on the convergence rates of the underlying nonparametric estimators.
\end{remark}

The next theorem, which is analogous to Theorems \ref{thm: IMD} and \ref{thm: uniform convergence of factors} in the main paper,  shows that the indirect matching discrepancy shrinks, and the local factors and factor loadings can be consistently estimated up to a rotation.

\begin{thm}[Local PCA with Covariates] \label{SA-thm: relevant FE}
	Under Assumptions \ref{SA-Assumption: high-rank covariates}--\ref{SA-Assumption: local approximation},
	if $\log(n/K)/K=o(1)$, $K\log n/n=o(1)$, $\delta_{Kp}\mathfrak r_{\varrho,n}^{2}\lesssim_\P 1$, and $(np)^{\frac{2}{\nu}}\delta_{Kp}^{-2}\lesssim 1$, then
	\[
	\max_{i\in[n]}
   \max_{k\in [K]}\;
   \|\balpha_i-\balpha_{j_k(i)}\|\lesssim_\P
   (K/n)^{\uvarsig/(\lvarsig r)}
   +a_p^{1/\lvarsig},
   \]
   and
	there exist matrices $(\bR_{\subi}:i\in[n])$ such that, uniformly over $i\in[n]$ and $j\in[\ttd_i]$,
	\begin{align*}
		&\|\widehat{\bF}_{\cdot j, \subi}-\bF_{\subi}((\bR_{\subi}')^{-1})_{\cdot j}\|_{\max}\lesssim_\P
		(\delta_{Kp}^{-1}+\mathfrak r_{\eta,\subi}+\mathfrak r_{\varrho,n}^{2})\upsilon_{j,\subi}^{-1},\\
	&\|\widehat{\bLambda}_{\cdot j,\subi}-\bLambda_{\subi}\bR_{\cdot j,\subi}\|_{\max}\lesssim_\P
	\delta_{Kp}^{-1}+\mathfrak r_{\eta,\subi}+\mathfrak r_{\varrho,n}^{2}.
	\end{align*}
	Moreover, $1\lesssim_\P \min_{i\in[n]}s_{\min}(\bR_{\subi})\leq \max_{i\in[n]}s_{\max}(\bR_{\subi})\lesssim_\P 1$.
\end{thm}

\subsection{Proofs}

\subsubsection{Proof of Lemma \ref{SA-lem: slope of covariates}}
	Define
	$\hslash_{it}=\eta_t(\balpha_i)+\sum_{\ell=1}^q\hbar_{t,\ell}(\bvrho_{i,\ell})\vartheta_\ell$.
	For $t\in\mathcal{C}^\ddagger$, stack $\hslash_{it}$, $\hbar_{t,\ell}(\bvrho_{i,\ell})$, and $e_{it,\ell}$ in the $p^\ddagger$-vectors $\bm{\hslash}_i$, $\bm{\hbar}_{i,\ell}$, and $\be_{i,\ell}$, respectively, and let $\be_i=(\be_{i,1},\ldots,\be_{i,q})'$. Define
	\[
	\bDelta_{i,\ell}=\widehat{\bm{\hbar}}_{i,\ell}-\bm{\hbar}_{i,\ell},\qquad
	\bDelta_{i,0}=\widehat{\bm{\hslash}}_i-\bm{\hslash}_i,\qquad
	\bDelta_i=(\bDelta_{i,1},\ldots,\bDelta_{i,q})'.
	\]
	Thus, $\widehat{\be}_i=\be_i-\bDelta_i$.

	Apply the local-PCA reconstruction result to each of the noise-free targets $\bar\bH_{\subi,\lceil\ell\rfloor}$ and $\bar{\bar{\bH}}_{\subi}$. By Assumption \ref{SA-Assumption: local approximation},
	\[
	\max_{i\in[n]}\left(\max_{\ell\in[q]}\|\bDelta_{i,\ell}\|_{\max}\vee\|\bDelta_{i,0}\|_{\max}\right)
	\lesssim_\P b_n,\qquad b_n:=\delta_{Kp}^{-1}+\mathfrak r_{\varrho,n}=o_\P(1).
	\]
	Moreover, since the matching and PCA blocks are disjoint, the conditional concentration and projected-score argument used in the proof of Theorem \ref{thm: uniform convergence of factors} yields
	\[
	\left\|\frac{1}{np^\ddagger}\sum_{i=1}^n\be_i\bDelta_i'\right\|+
	\left\|\frac{1}{np^\ddagger}\sum_{i=1}^n\be_i\bDelta_{i,0}\right\|+
	\left\|\frac{1}{np^\ddagger}\sum_{i=1}^n\bDelta_i\bu_i^\ddagger\right\|
	\lesssim_\P\delta_{Kp}^{-1}+\mathfrak r_{\varrho,n}\delta_{Kp}^{-1}.
	\]
	
	To obtain these bounds, condition on $\mathcal F$ and the matching blocks. For each $i$, decompose $\bDelta_{i,\ell}$, $\ell\in[q]$, and $\bDelta_{i,0}$, which are the reconstruction errors associated with $\bar\bH_{\subi,\lceil\ell\rfloor}$ and $\bar{\bar{\bH}}_{\subi}$, respectively, into their noise-free residual, loading-error, and factor-error parts. The noise-free residuals are fixed under this conditioning, while $e_{it,\ell}$ and $u_{it}$ remain conditionally centered. Hence, the products of these residuals with $e_{it,\ell}$ or $u_{it}$ are conditionally mean zero and can be bounded by conditional concentration.

	It remains to control the loading- and factor-error parts after the reconstruction errors are multiplied by $\be_i$ or $\bu_i^\ddagger$ in the three averages above. More explicitly, for $a,b\in[q]$, the corresponding entries of these averages are normalized sums of $e_{it,a}\Delta_{i,b,t}$, $e_{it,a}\Delta_{i,0,t}$, and $\Delta_{i,b,t}u_{it}$, respectively, over $i$ and $t\in\mathcal C^\ddagger$. For the remaining parts related to factor-/loading-errors, apply the leave-one-out construction in Steps 3 and 6 of the proof of Lemma \ref{lem: eigenvector bound}: omit the row corresponding to unit $i$ for the loading-error terms and the column corresponding to feature $t$ for the factor-error terms. Conditional on $\mathcal F$ and the matching blocks, the resulting leave-one-out estimator is independent of $e_{it,a}$ or $u_{it}$ in the same product. Equation \eqref{SA-eq: conditional weighted concentration} bounds these conditionally centered sums, and Equation \eqref{SA-eq: leave one out eigenspace} controls the difference between the leave-one-out and full-sample PCA estimators. Together, these arguments give the displayed rate, as in Step 6 of that proof.

	It follows from these bounds and Assumption \ref{SA-Assumption: high-rank covariates} that
	\[
	\widehat\bQ:=\frac{1}{np^\ddagger}\sum_{i=1}^n\widehat{\be}_i\widehat{\be}_i'
	=\frac{1}{np^\ddagger}\sum_{i=1}^n\be_i\be_i'+o_\P(1),
	\qquad \|\widehat\bQ^{-1}\|\lesssim_\P1.
	\]
	The definition of $\widehat\bvth$ now gives the exact decomposition
	\[
	\widehat\bvth-\bvth
	=\widehat\bQ^{-1}\frac{1}{np^\ddagger}\sum_{i=1}^n
	\widehat{\be}_i\{\bu_i^\ddagger+\bDelta_i'\bvth-\bDelta_{i,0}\}.
	\]
	By the conditional orthogonality and independence in Assumption \ref{SA-Assumption: regularities}(c), the products $u_{it}\be_{it}$ are conditionally centered and independent across $i$ and $t$. The lemma's rate restriction and the moment bounds give uniformly bounded conditional second moments of these products, so conditional Chebyshev's inequality yields $(np^\ddagger)^{-1}\sum_i\be_i\bu_i^\ddagger=O_\P((np)^{-1/2})$.

	Using these bounds, we conclude that the three terms inside braces are respectively $O_\P(\delta_{Kp}^{-1})$, $O_\P(\delta_{Kp}^{-1}+b_n^2)$, and $O_\P(\delta_{Kp}^{-1}+b_n^2)$ after multiplication by and averaging with $\widehat\be_i$. Finally,
	$b_n^2\lesssim_\P\delta_{Kp}^{-1}+\mathfrak r_{\varrho,n}^2$, which proves the result.

\subsubsection{Proof of Theorem \ref{SA-thm: relevant FE}}
	By Assumption \ref{SA-Assumption: indirect matching} and Theorem \ref{thm: IMD} in the main paper, we have
	$$\max_{1\leq i\leq n}
	\max_{1\leq k\leq K}
	\|\balpha_i-\balpha_{j_k(i)}\|\lesssim_\P
	(K/n)^{\uvarsig/(\lvarsig r)}
	+a_p^{1/\lvarsig}.$$

	For the final PCA step, write
	\[
	\tilde\bX_{\subi}^\wr=\bH_{\subi}+\bU_{\subi}+\bD_{\subi},
	\]
	where the $k$th row of $\bD_{\subi}$ is $\{\bW_{j_k(i)}^\wr(\bvth-\widehat\bvth)\}'$. Set $s_n=\delta_{Kp}^{-1}+\mathfrak r_{\varrho,n}^2$. By Lemma \ref{SA-lem: slope of covariates}, $\|\widehat\bvth-\bvth\|\lesssim_\P s_n$. By the boundedness and moment conditions in Assumption \ref{SA-Assumption: regularities}, the normalized operator norm and the two projected scores generated by $\bD_{\subi}$ are $O_\P(s_n)$ uniformly over $i$. Here we also use the feature-wise splitting. Furthermore, by the rate condition $\delta_{Kp}\mathfrak r_{\varrho,n}^2\lesssim_\P1$ and Assumption \ref{SA-Assumption: local approximation}(a),
	\[
	\max_{i\in[n]}\frac{s_n}{\upsilon_{\ttd_i,\subi}}=o_\P(1).
	\]
	Thus, $\bD_{\subi}$ does not alter the retained factor space asymptotically.

	Apply the proof of Theorem \ref{thm: uniform convergence of factors} to $\tilde\bX_{\subi}^\wr$, treating $\bD_{\subi}$ as an additional perturbation. In particular, for $\bR_{\subi}=(p^\wr)^{-1}\bF_{\subi}'\widehat\bF_{\subi}$,
	\[
	\widehat{\bLambda}_{\subi}-\bLambda_{\subi}\bR_{\subi}=
	\frac{1}{p^\wr}\bXi_{\subi}\widehat{\bF}_{\subi}+
	\frac{1}{p^\wr}\bU_{\subi}\widehat{\bF}_{\subi}+
	\frac{1}{p^\wr}\bD_{\subi}\widehat{\bF}_{\subi}.
	\]
	The three terms are uniformly $O_\P(\mathfrak r_{\eta,\subi})$, $O_\P(\delta_{Kp}^{-1})$, and $O_\P(s_n)$, respectively. This proves the loading bound. The symmetric factor expansion gives the stated factor bound after division by $\upsilon_{j,\subi}$. By the same subspace-consistency argument, the singular values of $\bR_{\subi}$ are uniformly bounded above and away from zero.

\section{Additional Simulation Results}\label{SA-sec: simulation}

In the following we present the additional simulation results for the three models described in the main paper.

Table \ref{table:simul, p=500} presents results for the case $p=500$. We still use the first half of columns for $K$-NN matching and the second half for local PCA. We can see the main findings are similar to those for the case $p=1\,000$.
The improvement of local PCA over global PCA is not pronounced in Model 3, since in this binary model the noise is large relative to the ``signal'' (nonlinearity of the latent surface), especially when $p$ is small.

Next, in Table \ref{table:simul, p=500, 0.8knn} we consider the same scenario as in Table \ref{table:simul, p=500} but use more features (80\% of columns) for $K$-NN matching and fewer features (20\%) for local PCA. Thus, PCA is conducted based on a submatrix of $\bX$ with $p$ and $K$ more balanced. We can see the performance of local PCA does not change much in Models 2 and 3. By contrast, the prediction errors for missing entries in Model 1 significantly increase. This can be explained by the large nonlinearity of this model, where using a smaller number of features for PCA greatly restricts our ability to extract weaker factors that capture higher-order nonlinearity of the latent functions.

\FloatBarrier
\begin{table}[h]
	\centering
	\scriptsize
	\caption{Simulation Results, $n=1\,000$, $p=500$, $50\%$ for KNN, $2\,000$ replications}
	\label{table:simul, p=500}
	\resizebox{.55\textwidth}{1.1\height}{\footnotesize \begin{tabular}{lrrrcr}
\hline\hline
\multicolumn{1}{l}{\bfseries }&\multicolumn{3}{c}{\bfseries LPCA}&\multicolumn{1}{c}{\bfseries }&\multicolumn{1}{c}{\bfseries GPCA}\tabularnewline
\cline{2-4} \cline{6-6}
\multicolumn{1}{l}{}&\multicolumn{1}{c}{$K$=49}&\multicolumn{1}{c}{$K$=99}&\multicolumn{1}{c}{$K$=149}&\multicolumn{1}{c}{}&\multicolumn{1}{c}{}\tabularnewline
\hline
{\bfseries Model 1}&&&&&\tabularnewline
~~MAE&$0.814$&$1.104$&$2.176$&&$1.199$\tabularnewline
~~$q_\alpha$=.1&$0.089$&$0.078$&$0.097$&&$0.123$\tabularnewline
~~$q_\alpha$=.5&$0.084$&$0.082$&$0.122$&&$0.108$\tabularnewline
~~$q_\alpha$=.9&$0.085$&$0.075$&$0.093$&&$0.119$\tabularnewline
\hline
{\bfseries Model 2}&&&&&\tabularnewline
~~MAE&$0.689$&$0.703$&$0.747$&&$0.875$\tabularnewline
~~$q_\alpha$=.1&$0.073$&$0.061$&$0.058$&&$0.118$\tabularnewline
~~$q_\alpha$=.5&$0.073$&$0.062$&$0.059$&&$0.098$\tabularnewline
~~$q_\alpha$=.9&$0.071$&$0.060$&$0.057$&&$0.116$\tabularnewline
\hline
{\bfseries Model 3}&&&&&\tabularnewline
~~MAE&$0.497$&$0.480$&$0.480$&&$0.475$\tabularnewline
~~$q_\alpha$=.1&$0.040$&$0.045$&$0.053$&&$0.050$\tabularnewline
~~$q_\alpha$=.5&$0.045$&$0.047$&$0.052$&&$0.048$\tabularnewline
~~$q_\alpha$=.9&$0.039$&$0.045$&$0.053$&&$0.050$\tabularnewline
\hline
\end{tabular}
}
	\flushleft\footnotesize{\textbf{Notes}:
		MAE = maximum absolute error; $q_\alpha$:  $\alpha$-quantile of $(\alpha_i: 1\leq i\leq n)$.
	}\newline
\end{table}

\begin{table}[h]
	\centering
	\scriptsize
	\caption{Simulation Results, $n=1\,000$, $p=500$, $80\%$ for KNN, $2\,000$ replications}
	\label{table:simul, p=500, 0.8knn}
	\resizebox{.55\textwidth}{1.1\height}{\footnotesize 
\begin{tabular}{lrrrcr}
\hline\hline
\multicolumn{1}{l}{\bfseries }&\multicolumn{3}{c}{\bfseries LPCA}&\multicolumn{1}{c}{\bfseries }&\multicolumn{1}{c}{\bfseries GPCA}\tabularnewline
\cline{2-4} \cline{6-6}
\multicolumn{1}{l}{}&\multicolumn{1}{c}{$K$=49}&\multicolumn{1}{c}{$K$=99}&\multicolumn{1}{c}{$K$=149}&\multicolumn{1}{c}{}&\multicolumn{1}{c}{}\tabularnewline
\hline
{\bfseries Model 1}&&&&&\tabularnewline
~~MAE&$0.984$&$1.126$&$2.159$&&$1.101$\tabularnewline
~~$q_\alpha$=.1&$0.117$&$0.102$&$0.114$&&$0.123$\tabularnewline
~~$q_\alpha$=.5&$0.114$&$0.105$&$0.139$&&$0.108$\tabularnewline
~~$q_\alpha$=.9&$0.117$&$0.103$&$0.115$&&$0.119$\tabularnewline
\hline
{\bfseries Model 2}&&&&&\tabularnewline
~~MAE&$0.736$&$0.735$&$0.781$&&$0.852$\tabularnewline
~~$q_\alpha$=.1&$0.074$&$0.062$&$0.059$&&$0.118$\tabularnewline
~~$q_\alpha$=.5&$0.072$&$0.062$&$0.061$&&$0.098$\tabularnewline
~~$q_\alpha$=.9&$0.074$&$0.062$&$0.059$&&$0.116$\tabularnewline
\hline
{\bfseries Model 3}&&&&&\tabularnewline
~~MAE&$0.486$&$0.467$&$0.465$&&$0.463$\tabularnewline
~~$q_\alpha$=.1&$0.043$&$0.042$&$0.049$&&$0.050$\tabularnewline
~~$q_\alpha$=.5&$0.046$&$0.046$&$0.050$&&$0.048$\tabularnewline
~~$q_\alpha$=.9&$0.042$&$0.042$&$0.049$&&$0.050$\tabularnewline
\hline
\end{tabular}
}
	\flushleft\footnotesize{\textbf{Notes}:
		MAE = maximum absolute error; $q_\alpha$:  $\alpha$-quantile of $(\alpha_i: 1\leq i\leq n)$.
	}\newline
\end{table}
\FloatBarrier

\putbib[references]

\end{bibunit}
\end{document}